\documentclass[sn-mathphys-num]{sn-jnl}

\usepackage{graphicx}%
\usepackage{multirow}%
\usepackage{amsmath,amssymb,amsfonts}%
\usepackage{amsthm}%
\usepackage{mathrsfs}%
\usepackage[title]{appendix}%
\usepackage{xcolor}%
\usepackage{textcomp}%
\usepackage{manyfoot}%
\usepackage{booktabs}%
\usepackage{algorithm}%
\usepackage{algorithmicx}%
\usepackage{algpseudocode}%
\usepackage{listings}%
\usepackage{arydshln}

\theoremstyle{thmstyleone}%
\newtheorem{theorem}{Theorem}
\newtheorem{proposition}[theorem]{Proposition}%

\newtheorem{lemma}{Lemma}[section]

\theoremstyle{thmstyletwo}%
\newtheorem{remark}{Remark}%

\theoremstyle{thmstylethree}%
\newtheorem{definition}{Definition}%

\begin{document}

\title[Article Title ]{Subspace Newton's Method for $\ell_0$-Regularized Optimization Problems with Box Constraint}


\author[1]{\fnm{} \sur{Yuge Ye}}\email{fujianyyg@163.com}


\author*[1]{\fnm{} \sur{Qingna Li}}\email{qnl@bit.edu.cn}

\affil[1]{\orgdiv{Department of Mathematics and Statistics}, \orgname{Beijing Institute of Technology}, \orgaddress{\street{No.5, ZhongGuanCun South Street}, \city{Beijing}, \postcode{100081}, \state{Beijing}, \country{China}}}




\abstract{This paper investigates the box-constrained $\ell_0$-regularized sparse optimization problem. We introduce the concept of a $\tau$-stationary point and establish its connection to the local and global minima of the box-constrained $\ell_0$-regularized sparse optimization problem. We utilize the $\tau$-stationary points to define the support set, which we divide into active and inactive components. Subsequently, the Newton's method is employed to update the non-active variables, while the proximal gradient method is utilized to update the active variables. If the Newton's method fails, we use the proximal gradient step to update all variables. Under some mild conditions, we prove the global convergence and the local quadratic convergence rate. Finally, experimental results demonstrate the efficiency of our method.
}





\maketitle

\section{Introduction}\label{sec1}
\numberwithin{equation}{section}
Sparse optimization has important applications in many fields, such as compressed sensing \cite{candes2006robust, 1542412, donoho2006compressed}, machine learning \cite{wright2010sparse, yuan2012visual}, neural networks \cite{bian2012smoothing, dinh2020sparsity, lin2019toward}, signal and image processing \cite{bian2015linearly, chen2012non, elad2010sparse, elad2010role}, matrix completion \cite{10.1145/2184319.2184343}, variable selection\cite{Liu01122007}, pattern recognition \cite{1561263} and regression \cite{10.1111/j.2517-6161.1996.tb02080.x}.
In this paper, we consider the sparse solutions of the following box-constraint $\ell_0$ regularized optimization problem.
\begin{align}\label{BL0}
\min\limits_{x\in \mathbb{R}^n} &\phi(x) := f(x) + \lambda \|x\|_0, \nonumber \\
\mbox{s.t.} & -l \le x \le u,
\end{align}
where $f: \mathbb{R}^n \to \mathbb{R}$ is twice continuously differentiable and bounded from below, $\lambda > 0$ is the penalty parameter and $\|x\|_0$ is the $\ell_0$ norm of $x$, counting the number of non-zero elements of $x$. The bound vector $l,u \in \mathbb{R}_{+}^{n}$.

For the unconstraint $\ell_0$-regularized optimization problems, there are various existing methods, such as the iterative hard-thresholding algorithm (IHT, \cite{blumensath2008iterative}), the forward-backward splitting (FBS, \cite{attouch2013convergence}), mixed integer optimization method (MIO, \cite{bertsimas2016best}), active set Barzilar
Borwein (ABB, \cite{cheng2020active}). These methods are known as the first-order methods.
There are also several
second-order methods that have been proposed, such as the primal dual active set (PDAS, \cite{Ito_2014}), which analyzes the speed of error decay and the computational complexity of each iteration. Additionally, the support detection and rootfinding (SDAR, \cite{huang2018constructive}), which proves that under certain conditions, the algorithm converges in a finite number of steps. However these methods lack of theoretical guarantees for local quadratic convergence rates. For sparse optimization problems with the $\ell_0$ norm, Zhou et al. have proposed many innovative methods, such as
Newton hard-thresholding  pursuit (NHTP \cite{zhou2021global}), smoothing Newton's method for $\ell_{0/1}$ loss optimization (NM01 \cite{zhou2021quadratic}), subspace Newton's method (NL0R, \cite{zhou2021newton}). Notably, NL0R is the first algorithm proven to exhibit global convergence and local quadratic convergence rate for unconstraint $\ell_0$-regularized optimization problems.

For the box-constrained $\ell_0$-regularized optimization problems. To our best knowledge, there are also some methods for \eqref{BL0}, such as the Frank-Wolfe reduced dimension method (FW-RD,  \cite{liuzzi2015solving}), smoothing fast iterative hard thresholding algorithm (SFIHT, \cite{wu2021smoothing}),  adaptive projected gradient thresholding methods (APGT, \cite{zhao2015adaptive}), smoothing proximal gradient algorithm (SPG, \cite{doi:10.1137/18M1186009}),  proximal iterative hard thresholding methods for $\ell_0$-regularized convex cone programming (PIHT, \cite{lu2014iterative}),  smoothing neural network \cite{LI2021678}, extrapolation proximal iterative hard thresholding method (EPIHT, \cite{zhang2018acceleratedproximaliterativehard}) , accelerated iterative hard thresholding method (AIHT, \cite{wu2020accelerated}).
Clearly, these methods are also first-order algorithms and do not possess local quadratic convergence rates.

Motivated by the subspace Newton's method in \cite{zhou2021newton} and proximal iterative hard thresholding methods in \cite{lu2014iterative}, a natural question is whether we can develop a subspace Newton's method for \eqref{BL0}. This motivate the work in this paper. We develop a novel subspace Newton's method (BNL0R) for \eqref{BL0}.

The main contributions of this paper are summarized as the following three aspects. Firstly, we introduce a $\tau$-stationary point for \eqref{BL0} and provide its closed-form. Then we explore the relationship among the $\tau$-stationary point and the local/global minimizers of \eqref{BL0}. Secondly, we propose a novel subspace Newton's method (BNL0R) for \eqref{BL0}. In particular, we identify a support set in each iteration and divide the support set into active and inactive parts, then we apply Newton's method within the inactive parts, to search for an improved iteration point. If the Newton direction fails to be a descent direction, we utilize the iteration direction provided by the PIHT method in \cite{lu2014iterative}. Thirdly, by combing with an Armijo linesearch, we establish its global and quadratic convergence properties under proper assumptions. Finally, we demonstrate the efficiency of our proposed method through extensive numerical experiment results.

The organization of the paper is as follows. In section 2, we introduce the $\tau$-stationary point and characterise its equivalent formulation. In section 3, we propose the so-called BNL0R for problem \eqref{BL0}. In section 4, we analyze the global convergence and local quadratic convergence rate. We conduct various numerical experiments in section 5 to verify the efficiency of the proposed method. Final conclusions are given in section 6.

Notations. For $x\in \mathbb{R}^n$, $|x| := ( |x_1|, |x_2|, \cdots, |x_n|)^{\top}$ denotes the absolute value of each component of $x $ and ${\rm{supp}} (x)$ be its support set consisting of the indices of the non-zero elements. Let $\mathbb{N}_{n} := [1,2,\cdots, n]$. Given a set $\mathcal{I} \subseteq \mathbb{N}_{n}$, we denote $|\mathcal{I}|$ as its cardinality set and $\overline{\mathcal{I}}$ as its complementary set. Given a marix $H \in \mathbb{R}^{m\times n}$, let $H_{\Theta, \Gamma}$ represent its sub-matrix containing rows indexed by $\Theta \subseteq \mathbb{N}_{m}$ and columns indexed by $\Gamma \subseteq \mathbb{N}_{n}$. Let $I_{\Theta} \in \mathbb{R}^{|\Theta|\times |\Theta|}$ represent the identity matrix.  In particular, we define the sub-gradient and sub-Hessian by
\begin{align*}
\nabla_{\Theta} f(x) := [\nabla f(x)]_{\Theta},\ \nabla^2_{\Theta,\Gamma} f(x) := [\nabla^2 f(x)]_{\Theta,\Gamma},\ \nabla^2_{\Theta :} f(x) := [\nabla^2 f(x)]_{\Theta,\mathbb{N}_n} .
\end{align*}
Let $\| \cdot \|$ represents the $\ell_2$ norm. Denote $\Pi_{\Omega}(x) := \arg \min\limits_{y\in \Omega} \|x-y\|^2$ which is a project mapping to $\Omega$ .

\section{$\tau$-stationary point and its properties}
In this section, we introduce the $\tau$-stationary point and its equivalent formulation, as well as the relation with local and global minimizers of \eqref{BL0}.

Problem \eqref{BL0} can be equivalent written as the following equivalent form
\begin{align}\label{PBL0}
\min\limits_{x\in \mathbb{R}^{n}} f(x) + \lambda p(x),
\end{align}
where $p(\cdot) := \|\cdot \|_0 + \delta_{\Omega}(\cdot)$, $\Omega := \{x\in \mathbb{R}^{n}\ | \ -l \le x \le u\}$, and  $\delta_{\Omega}(\cdot)$ denotes the indicator function of $\Omega$ defined as \begin{align*}
\delta_{\Omega}(x) =
\begin{cases}
0, \        &\mbox{if}\  x\in \Omega , \\
+\infty, \  &\mbox{if}\  x\notin \Omega .\\
\end{cases}
\end{align*}
\subsection{Formulation of $\tau$-stationary point}
Inspired by \cite{zhou2021newton}, we introduce the following definition of the so-called $\tau$-stationary point.
\begin{definition}
Given $\tau > 0$, we say that $x$ is a $\tau$-stationary point of \eqref{PBL0} if the following holds
\begin{align}\label{eq-2-1}
x &\in {\rm Prox}_{\tau \lambda p(\cdot)}(x - \tau \nabla f(x)) \nonumber\\
&:= \arg \min\limits_{y\in \mathbb{R}^n} \frac{1}{2}\|y - (x-\tau \nabla f(x))\|^2 + \tau \lambda p(y) .
\end{align}
\end{definition}
Throughout this paper, we choose $\tau >0$ to satisfy  the following conditions
\begin{align}\label{eq-tau}
\tau < \frac{1}{2\lambda} a,\ a := \min\limits_{1\le i \le n} \min (l_i^2, u_i^2).
\end{align}
This is an important setting, and is easily achieved, which allows us to make the proximal operator more concise. To characterise ${\rm Prox}_{\tau \lambda p(\cdot)} (\cdot)$, we need the following Lemma.
\begin{lemma}\label{lem-prox}
For problem \eqref{eq-2-1}, let $\tau > 0$ satisfy \eqref{eq-tau}. Then
\begin{align}\label{eq-2-2-2}
[{\rm Prox}_{\tau \lambda p(\cdot )}(z)]_i =
\begin{cases}
z_i,  & \mbox{if } -l_i< z_i< u_i,\ |z_i| > \sqrt{2\tau \lambda} , \\
u_i , & \mbox{if } z_i \ge u_i ,\\
-l_i, & \mbox{if } z_i \le -l_i ,\\
z_i\ \mbox{or}\ 0, & \mbox{if }  |z_i| = \sqrt{2\tau \lambda},\\
0,  & \mbox{if }  \sqrt{2\tau \lambda} > |z_i| .
\end{cases}
\end{align}
\end{lemma}
\begin{proof}
One can observe that the objective function \eqref{PBL0} is separable. Thus, we can decompose the solution of \eqref{eq-2-2-2} and consider them individually. That is, we consider the one-dimensional minimization problem as follows
\begin{align*}
\min\limits_{y_i \in \mathbb{R}} \frac{1}{2}(y_i - z_i)^2 + \tau \lambda \|y_i\|_0 + \delta_{\Omega_i}(y_i) := h(y_i) .
\end{align*}
where $\Omega_i := \left\{ x_i \in \mathbb{R}\, \mid \, -l_i \le x_i \le u_i \right\}$.

It holds that $h(y_i)$ have the following  three cases based on different range of $y_i$
\begin{align*}
h(y_i) =
\begin{cases}
\frac{1}{2}z_i^2,\ &y_i = 0 ,\\
\frac{1}{2}(y_i - z_i)^2 + \tau \lambda,\ &y_i \in \Omega_i\setminus \{0\} ,\\
+\infty,\ &y_i \notin \Omega_i .\\
\end{cases}
\end{align*}
If $y_i \in \Omega_i\setminus \{0\}$, then the minimization of $h(y_i)$ can be achieved at $y_i = z_i$, provided that $z_i\in \Omega_i\setminus \{0\}$. The minimal value is $h(y_i) = \tau \lambda$. Therefore, the minimization of $h(y_i)$ can be only achieved when $y\in [-l_i,u_i]$.

Now, we consider more details based on the value of $z_i$.

\textbf{Case 1}. If $-l_i< z_i< u_i,\ |z_i| > \sqrt{2\tau \lambda}$, then  $h(0) = \frac{1}{2}z_i^2 > h(z_i) = \tau \lambda$. In this case, $z_i$ is the minimizer of $h(y_i)$ with the minimal value $h(z_i) = \tau \lambda$. That is $\left[\Pi_{\Omega}(z)\right]_i = z_i$.

\textbf{Case 2}. If $z_i \ge u_i$, then $\min\limits_{y_i \in \Omega_i\setminus \{0\}} \frac{1}{2}(y_i - z_i)^2 + \tau \lambda = \frac{1}{2}(u_i - z_i)^2 + \tau \lambda = h(u_i)$. In this case,
$h(u_i) - h(0) = \frac{1}{2}(u_i - z_i)^2 + \tau \lambda - \frac{1}{2}z_i^2 = \frac{1}{2}u_i^2 -u_iz_i + \tau \lambda < 0,$
$$ z_i \ge u_i \Rightarrow -u_iz_i \le -u_i^2 < -2\tau \lambda.$$
Therefore, $\frac{1}{2}u_i^2 -u_iz_i + \tau \lambda < -\tau\lambda + \tau\lambda = 0$.
That is  $h(u_i) < h(0)$. As a result, $u_i$ is the minimizer in this case.

\textbf{Case 3}. Similar to \textbf{Case 2}, we can show that if $z_i \le -l_i$, then $h(-l_i) < h(0)$. Therefore, $-l_i$ is the minimizer for $z_i \le -l_i$.

\textbf{Case 4}. If $|z_i| = \sqrt{2\tau \lambda}$. In this case, we can see that similar to \textbf{Case 1}, $h(0) = \frac{1}{2}z_i^2 = h(z_i) = \tau \lambda$. Therefore, 0 and $z_i$ are both minimizer of $h(y_i)$.

\textbf{Case 5}. If $\sqrt{2\tau \lambda} > |z_i|$. In this case, $h(z_i) - h(0) > 0$ and 0 is the minimizer of $h(y_i)$. The proof is complete.
\end{proof}
Then we can equivalently characterize a $\tau$-stationary point by the conditions below.
\begin{lemma}\label{lemma-2}
$x$ is $\tau$-stationary point for \eqref{BL0}, if only if
\begin{align}\label{eq-2-3}
\begin{cases}
\nabla_i f(x) = 0,  & \mbox{if } -l_i < x_i < u_i,\ |x_i | > \sqrt{2\tau \lambda}, \\
x_i - u_i = 0 , & \mbox{if }   \nabla_i f(x) \le 0,\\
x_i + l_i = 0, & \mbox{if }  \nabla_i f(x) \ge 0 ,\\
x_i = 0,  & \mbox{if } \sqrt{2\lambda/ \tau } \ge | \nabla_i f(x)|.
\end{cases}
\end{align}
\end{lemma}
\begin{proof}
Since $x$ is $\tau$-stationary point, it is equivalent to $x \in {\rm Prox}_{\tau \lambda p(\cdot)}(x - \tau \nabla f(x))$. By Lemma \ref{lem-prox}, we obtain that
\begin{align*}
0 &\in x_i - {\rm Prox}_{\tau \lambda p(\cdot )}(x - \tau \nabla f(x))_i \\
&=
\begin{cases}
\tau \nabla_i f(x),  & \mbox{if } -l_i < x_i - \tau \nabla_i f(x) < u_i, |x_i - \tau \nabla_i f(x)| > \sqrt{2\tau \lambda}, \\
x_i - u_i , & \mbox{if } x_i - \tau \nabla_i f(x) \ge u_i ,\\
x_i + l_i, & \mbox{if } x_i - \tau \nabla_i f(x) \le -l_i ,\\
\tau \nabla_i f(x)\ \mbox{or}\ x_i, & \mbox{if }  |x_i - \tau \nabla_i f(x)| = \sqrt{2\tau \lambda},\\
x_i,  & \mbox{if }  \sqrt{2\tau \lambda} > |x_i - \tau \nabla_i f(x)|,
\end{cases}
\end{align*}
where $i\in [1,\cdots,n]$. It is equivalent to the following form
\begin{align}
\begin{cases}
\nabla_i f(x) = 0,  & \mbox{if } -l_i< x_i< u_i,\ |x_i| > \sqrt{2\tau \lambda}, \\
x_i - u_i = 0 , & \mbox{if } x_i - \tau \nabla_i f(x) \ge u_i ,\\
x_i + l_i = 0, & \mbox{if } x_i - \tau \nabla_i f(x) \le -l_i ,\\
\nabla_i f(x) = 0, & \mbox{if }  |x_i | = \sqrt{2\tau \lambda},\\
x_i = 0, & \mbox{if }  | \tau \nabla_i f(x)| = \sqrt{2\tau \lambda},\\
x_i = 0,  & \mbox{if }  \sqrt{2\tau \lambda} > |\tau \nabla_i f(x)|.
\end{cases}
\end{align}
Where case $|x_i| = \sqrt{2\tau\lambda}$ and $| \tau \nabla_i f(x)| = \sqrt{2\tau \lambda}$ are defined from $|x_i- \tau \nabla_i f(x)| = \sqrt{2\tau \lambda}$. For $x_i = u_i$ while $x_i- \tau \nabla_i f(x) \ge u_i$, there is $- \tau \nabla_i f(x) \ge 0$. Similarly, there is $- \tau \nabla_i f(x) \le 0$ for $x_i = -l_i$ while $x_i - \tau \nabla_i f(x) \le -l_i$. Summarizing the case of $x_i=0$ leads to $x_i = 0$, if $|\tau \nabla_i f(x)| \le \sqrt{2\tau \lambda}$. This gives the result.
\end{proof}
Our subsequent results require the strong smoothness and convexity of $f$.
\begin{definition}
$f$ is strongly smooth about constant $L>0$ if
\begin{align}\label{eq-2-4}
f(z) \le f(x) + \left\langle \nabla f(x),z-x\right\rangle + (L/2)\|z-x\|^2,\ \forall \ x,z\in \mathbb{R}^n ,
\end{align}
$f$ is strongly convex about constant $\ell > 0$ if
\begin{align}\label{eq-2-5}
f(z) \ge f(x) + \left\langle \nabla f(x),z-x\right\rangle + (\ell/2)\|z-x\|^2,\ \forall \ x,z\in \mathbb{R}^n .
\end{align}
\end{definition}
We also need the following proposition.
\begin{proposition} {\rm{\cite[Lemma 2.1]{calamai1987projected}} } \label{pro-projection}
If $\Omega \subset \mathbb{R}^n$ is a nonempty closed convex set. Let  $\Pi_{\Omega}(\cdot)$ be the projection into $\Omega$.
If $y \in \Omega$ then
$$\left\langle \Pi_{\Omega}(x) - x, y - \Pi_{\Omega}(x) \right\rangle \ge 0, \forall x \in \mathbb{R}^n .$$
\end{proposition}
To express the solution of \eqref{eq-2-1} more explicitly, we define the following indexes
\begin{align}\label{eq-2-6}
&\Theta :=\Theta_{\tau}(x) = \{ i\in \mathbb{N}_n\ | \ -l_i < x_i - \tau \nabla_{i} f(x) < u_i, |x_i - \tau \nabla_{i} f(x)|\ge \sqrt{2\tau \lambda} \} , \nonumber\\
&\Gamma :=\Gamma_{\tau}(x) = \{ i\in \mathbb{N}_n \ | \ x_i - \tau \nabla_{i} f(x) \ge u_i\ \text{or}\  x_i - \tau \nabla_{i} f(x) \le -l_i \} ,\nonumber\\
&\Gamma^u := \Gamma^u_{\tau}(x) = \{ i\in \mathbb{N}_n \ | \ x_i - \tau \nabla_{i} f(x) \ge u_i \} ,\nonumber\\
&\Gamma^l :=\Gamma^l_{\tau}(x) = \{i\in \mathbb{N}_n \ |  x_i - \tau \nabla_{i} f(x) \le -l_i \} ,\nonumber\\
&\mathcal{I} := \Theta\cup \Gamma  ,\nonumber \\
&\overline{\mathcal{I}} :=\mathbb{N}_{n} \setminus (\Theta\cup \Gamma) .
\end{align}
We call $\Gamma$ as the active set and $\Theta$ as the inactive set. When the iterate gets gradually close to a minimizer, these variables in the active and non-active set can be accurately identified under some mild conditions. In particular, we use $\Theta_*$, $\Gamma_*$, $\mathcal{I}_*$, $\overline{\mathcal{I}}_*$ to denote the corresponding index sets at $x^*$. Based on the above set, we introduce the following a system of equations
\begin{align}\label{eq-2-7}
F_{\tau}(x;\Theta ; \Gamma ; \overline{\mathcal{I}})  :=
 \begin{bmatrix}
 \nabla_{\Theta}f(x) \\
 x_{\Gamma} - \left[\Pi_{\Omega}(x - \tau \nabla f(x))\right]_{\Gamma} \\
 x_{\overline{\mathcal{I}}}
 \end{bmatrix}
 =
  \begin{bmatrix}
 \nabla_{\Theta}f(x) \\
 x_{\Gamma^u} - u_{\Gamma^u} \\
 x_{\Gamma^l} - u_{\Gamma^l} \\
 x_{\overline{\mathcal{I}}}
 \end{bmatrix}
 = 0.
\end{align}
The relationship between \eqref{eq-2-1} and \eqref{eq-2-7} is revealed by the following theorem.
\begin{proposition} \label{pro-3}
For any $x \in \Omega$, by letting $z = x - \tau \nabla f(x)$, it holds that
\begin{align*}
 0 = x-{\rm{Prox}}_{\tau \lambda p(\cdot )}(z) \Rightarrow
 F_{\tau}(x;\Theta ; \Gamma ; \overline{\mathcal{I}}) = 0
 \Rightarrow
 x \in {\rm{Prox}}_{\tau \lambda p(\cdot )}(z).
\end{align*}
\end{proposition}
\begin{proof}
For the first part, if we have $x = {\rm Prox}_{\tau \lambda p(\cdot)} (z)$, namely, ${\rm Prox}_{\tau \lambda p(\cdot)} (z)$ is a singleton, then there is no index $i \in \Theta$ such that $|z_i| = \sqrt{2\tau \lambda}$ by \eqref{eq-2-2-2}. This and \eqref{eq-2-2-2} give rise to $ [{\rm Prox}_{\tau \lambda p(\cdot)} (z)]_{\Theta} = z_{\Theta} $. As a consequence,
\begin{align*}
0 = x - {\rm{Prox}}_{\tau \lambda p(\cdot)} (z) \stackrel{\eqref{eq-2-2-2}}{=}  \begin{bmatrix}
 x_{\Theta} \\
 x_{\Gamma} \\
 x_{\overline{\mathcal{I}}}
 \end{bmatrix}-
  \begin{bmatrix}
 z_{\Theta} \\
 \left[\Pi_{\Omega}(z)\right]_{\Gamma} \\
 0
 \end{bmatrix}
 =
  \begin{bmatrix}
 \tau \nabla_{\Theta}f(x) \\
 x_{\Gamma} - \left[\Pi_{\Omega}(z)\right]_{\Gamma} \\
 x_{\overline{\mathcal{I}}}
 \end{bmatrix},
\end{align*}
which suffices to obtain that $F_{\tau}(x;\Theta ; \Gamma ; \overline{\mathcal{I}}) = 0$.

For the second part, for any $i\in \Theta$, we have $\nabla_i f(x)=0$ from \eqref{eq-2-7}. Together with \eqref{eq-2-6}, there is $-l_i < x_i <u_i$ and $|x_i| \ge \sqrt{2\tau \lambda}$ . For any $i\in \Gamma$, we have $x_i = \left[\Pi_{\Omega}(x - \tau \nabla f(x))\right]_i$ from \eqref{eq-2-7} and $ x_i - \tau \nabla_i f(x) > u_i $ or $ x_i - \tau \nabla_i f(x) < -l_i $ from \eqref{eq-2-6}. If $ x_i - \tau \nabla_i f(x) > u_i $, there is $x_i = u_i$. If $ x_i - \tau \nabla_i f(x) < -l_i $, there is $x_i = -l_i$. For any $i \in \overline{\mathcal{I}}$, we have $x_i = 0$ from \eqref{eq-2-7} and $|\tau \nabla_i f(x)| = |x_i - \tau \nabla_i f(x)| < \sqrt{2\tau \lambda}$ from \eqref{eq-2-6}. Those together with Lemma \ref{lemma-2} completes the proof.
\end{proof}
\begin{remark} Note that, if $\nabla f(0)= 0$, then $0$ is a $\tau$-stationary point of the problem \eqref{BL0}. This case is trivial. However, we are interested in the non-trivial case. Thus, from now on, we always suppose $\nabla f(0) \neq 0$. Let $\tau \le \frac{a}{2\max\limits_{i} |\nabla_i f(0)|}$ where $a$ defined in \eqref{eq-tau}. Similar to \cite{zhou2021newton}, denote
\begin{align}\label{eq-2.13}
\underline{\lambda} := \min\limits_{i} \left\{ \frac{\tau}{2}|\nabla_i f(0)|^2\ :\ \nabla_i f(0) \neq 0  \right\}, \overline{\lambda} := \max\limits_{i} \frac{\tau}{2}|\nabla_i f(0)|^2 .
\end{align}
Then $a > a/2 \ge | 0 - \tau \nabla_i f(0)| \ge \sqrt{2\tau \lambda}$ for any $i\in \mathbb{K} := \{i \in \mathbb{N}_n : \nabla_i f(0) \neq 0\}$, resulting in $\mathbb{K} \subseteq \Theta_{\tau}(0)$ and consequently, $ F_{\tau}(0; \Theta_{\tau}(0); \Gamma_{\tau}(0); \overline{\mathcal{I}}_{\tau}(0)) \neq 0$ due to $\nabla_{\mathbb{K}} f(0) \neq 0$. Namely, 0 is not a $\tau$-stationary point of \eqref{PBL0}.
\end{remark}
\subsection{Relation to local/global minimizer}
Now we ready to establish the relationships between $\tau$-stationary point and a local/global minimizer of \eqref{PBL0}.
\begin{lemma}
Let $f:\mathbb{R}^n \to \mathbb{R}$ is continuously differentiable, $\mathcal{G}_* \subseteq \mathbb{N}_{n}$. If $x^* \in \mathbb{R}^n$ satisfies $x^*_{\mathcal{G}_*} = \left[\Pi_{\Omega}(x^* - \tau \nabla f(x^*))\right]_{\mathcal{G}_*} $, then it holds that
\begin{align}\label{eq-2-6-2}
\left\langle \nabla_{\mathcal{G}_*}f(x^*), (x - x^*)_{\mathcal{G}_*} \right\rangle \ge 0 ,\ \forall \ x\in \Omega.
\end{align}
\end{lemma}
\begin{proof}
For convenience denote $z^* = x^* - \tau \nabla f(x^*)$. By Proposition \ref{pro-projection}, we obtain that
\begin{align}\label{eq-2-6-2-3}
\left\langle x_{\mathcal{G}_*} - \left[\Pi_{\Omega}(z^*)\right]_{\mathcal{G}_*}, z^*_{\mathcal{G}_*} - \left[\Pi_{\Omega}(z^*)\right]_{\mathcal{G}_*} \right\rangle \le 0  .
\end{align}
then, by the definition of $z^*$ and \eqref{eq-2-6-2-3}, it holds that
\begin{align*}
&-\left\langle \nabla_{\mathcal{G}_*}f(x^*), (x - x^*)_{\mathcal{G}_*} \right\rangle \\
=& \left\langle x_{\mathcal{G}_*}^* - \tau \nabla_{\mathcal{G}_*}f(x^*) - x_{\mathcal{G}_*}^* + (\tau - 1)\nabla_{\mathcal{G}_*}f(x^*), (x - x^*)_{\mathcal{G}_*} \right\rangle \\
=& \left\langle z^*_{\mathcal{G}_*} - \left[\Pi_{\Omega}(z^*)\right]_{\mathcal{G}_*} , x_{\mathcal{G}_*} -\left[\Pi_{\Omega}(z^*)\right]_{\mathcal{G}_*} \right\rangle  +\left\langle (\tau - 1)\nabla_{\mathcal{G}_*}f(x^*), (x - x^*)_{\mathcal{G}_*} \right\rangle \\
\le & \left\langle (\tau - 1)\nabla_{\mathcal{G}_*}f(x^*), (x - x^*)_{\mathcal{G}_*} \right\rangle,
\end{align*}
which gives \eqref{eq-2-6-2}.
\end{proof}

\begin{proposition} For the problem \eqref{PBL0}, the following results hold.
\begin{enumerate}
    \item[(i)] A global minimizer $x^*$ is also a $\tau$-stationary point for any $0 < \tau < \frac{1}{L}$ if $f$ is strongly smooth with $L > 0$. Moreover,
    \[
    x^* = {\rm Prox}_{\tau \lambda p(\cdot )} \left( x^* - \tau \nabla f(x^*) \right).
    \]
    \item[(ii)] A $\tau$-stationary point with $\tau > 0$ is a local minimizer if $f$ is convex.

    \item[(iii)] A $\tau$-stationary point with $\tau > 1/\ell$ is also a (unique) global minimizer if $f$ is strongly convex with $\ell > 0$.
\end{enumerate}
\end{proposition}
\begin{proof} (i) Denote $\mathbb{P} := {\rm Prox}_{\tau \lambda p(\cdot) } (x^* - \tau \nabla f(x^*))$ and $ \mu := L - 1/\tau < 0 $ due to $ 0 < \tau < \frac{1}{L} $. Let $ x^* $ be a global minimizer on $\Omega$ and consider any point $ z \in \mathbb{P} $. Because $ f $ is strongly smooth and $L = \mu + 1/\tau$, it holds that
\begin{align}\label{pro-2-1}
&\quad \  2f(z) + 2\lambda p(z) \nonumber\\
&\le 2f(x^*) + 2\left\langle \nabla f(x^*), z - x^* \right\rangle + L\|z - x^*\|^2 + 2\lambda p(z) \\
&= 2f(x^*) + 2\left\langle \nabla f(x^*), z - x^* \right\rangle + (1/\tau) \|z - x^*\|^2 + \mu \|z - x^*\|^2 + 2\lambda p(z) \nonumber\\
&= 2f(x^*) + (1/\tau) \|z - (x^* - \tau \nabla f(x^*))\|^2 - \tau \|\nabla f(x^*)\|^2 + 2\lambda p(z) + \mu \|z - x^*\|^2 .\nonumber
\end{align}
By $ z \in \mathbb{P} $ , we obtain that
\begin{align}\label{pro-2-2}
&\quad \  (1/\tau) \|z - (x^* - \tau \nabla f(x^*))\|^2  + 2\lambda p(z)  \nonumber \\
&\le (1/\tau) \|x^* - (x^* - \tau \nabla f(x^*))\|^2 + 2\lambda p(x^*) = \tau \| \nabla f(x^*))\|^2 + 2\lambda p(x^*).
\end{align}
By the fact that $ x^* $ is the global minimizer of \eqref{PBL0}, we obtain that
\begin{align}\label{pro-2-3}
2f(x^*) + 2\lambda p(x^*)  \le 2f(z) + 2\lambda p(z).
\end{align}
From \eqref{pro-2-1}, \eqref{pro-2-2} and \eqref{pro-2-3}, it holds that
\begin{align*}
2f(z) + 2\lambda p(z) \le 2f(z) + 2\lambda p(z) + \mu \|z - x^*\|^2.
\end{align*}
This, together with $ \mu < 0 $, leads to $0 \leq \frac{\mu}{2} \|z - x^*\|^2 \leq 0$, yielding $ z = x^* $. Therefore, $ x^* $ is a $ \tau $-stationary point of \eqref{PBL0}. Since $ z $ is arbitrary in $ \mathbb{P} $ and $ z = x^* $, it holds that $ \mathbb{P} $ is a singleton only containing $ x^* $.

(ii) Now, we consider the case that $f$ is convex. Let $x^*$ be a $\tau$-stationary point with $\tau > 0$ and $\mathcal{I}_* := \text{supp}(x^*)$, it is obvious that $x^* \in \Omega$. Consider a neighborhood region of $x^*$ as $N(x^*) = \{ x \in \Omega \mid \| x - x^* \| < \epsilon_*\}$, where
\begin{align}\label{eq-epsilon}
\epsilon_* :=
\begin{cases}
\min\left\{ \min\limits_{i\in \mathcal{I}_*}\ |x_i^*|,\ \sqrt{{\tau \lambda}/{2n}} \right\} ,\ \text{if}\ x^* \neq 0, \\
\sqrt{{\tau \lambda}/{2n}},\ \text{if}\ x^*=0.
\end{cases}
\end{align}
For any point $x \in N(x^*)$, we conclude that $\mathcal{I}_* \subseteq \text{supp}(x)$. In fact, this is true when $x^* = 0$. When $x^* \neq 0$, if there is a $j$ such that $j \in \mathcal{I}^*$ but $j \notin \text{supp}(x)$, then we derive a contradiction as follows
$$
\epsilon_* \le \min\limits_{i \in \mathcal{I}_*} |x^*_i| \le |x^*_j| = |x^*_j - x_j| \le \|x - x^*\| < \epsilon_*.
$$
Therefore, we have $\mathcal{I}_* \subseteq \text{supp}(x)$.
Because $\tau > 0$, $f$ is convex and \eqref{eq-2-6} , it holds that
\begin{align}\label{eq-2-6-3-0}
f(x) - f(x^*) &\ \ \ge \ \  \left\langle \nabla f(x^*), x - x^* \right\rangle  \nonumber\\
&\ \  = \ \ \left\langle \nabla_{\mathcal{I}_*}f(x^*), (x - x^*)_{\mathcal{I}_*} \right\rangle + \left\langle \nabla_{\overline{\mathcal{I}}_*} f(x^*),(x - x^*)_{\overline{\mathcal{I}}_*} \right\rangle .
\end{align}
By \eqref{eq-2-3}, $\nabla_{\Theta_*} f(x^*) = 0$ which implies that
$$\left\langle \nabla_{\mathcal{I}_*}f(x^*), (x - x^*)_{\mathcal{I}_*} \right\rangle = \left\langle \nabla_{\Gamma_*}f(x^*), (x - x^*)_{\Gamma_*} \right\rangle .$$
Together with \eqref{eq-2-6-2} and \eqref{eq-2-6-3-0}, we obtain that
\begin{align}\label{eq-2-6-3}
f(x) - f(x^*) &\ge \left\langle \nabla_{\Gamma_*}f(x^*), (x - x^*)_{\Gamma_*} \right\rangle + \left\langle \nabla_{\overline{\mathcal{I}}_*} f(x^*),(x - x^*)_{\overline{\mathcal{I}}_*} \right\rangle \nonumber\\
&\ge \left\langle \nabla_{\overline{\mathcal{I}}_*} f(x^*),(x - x^*)_{\overline{\mathcal{I}}_*} \right\rangle =: \phi.
\end{align}
If $\mathcal{I}_* = \text{supp}(x)$, then $\phi = 0$ due to $x_{\overline{\mathcal{I}}_*} = 0$ , $\|x^*\|_0 = \|x\|_0$ and $\delta_{\Omega} (x^*) = \delta_{\Omega} (x)$. It is obvious that $p(x^*) = p(x)$. These allow us to derive that
\begin{align*}
f(x) + \lambda p(x) \stackrel{\eqref{eq-2-6-3}}{\ge} f(x^*) + \phi + \lambda p(x) = f(x^*) + \lambda p(x^*).
\end{align*}
If $\mathcal{I}_* \subset \text{supp}(x)$, then $\|x \|_0 - 1 \ge \|x^*\|_0$ so that $p(x) - 1 \ge p(x^*)$. In addition,
\begin{align*}
\phi &\, \;= \, \; \left\langle \nabla_{\overline{\mathcal{I}}_*}f(x^*), x_{\overline{\mathcal{I}}_*} - x^*_{\overline{\mathcal{I}}_*} \right\rangle \ge
- \| \nabla_{\overline{\mathcal{I}}_*}f(x^*)\| \| x_{\overline{\mathcal{I}}_*} - x^*_{\overline{\mathcal{I}}_*} \|.
\end{align*}
Combine with \eqref{eq-epsilon} and $\|\nabla_{\overline{\mathcal{I}}_*}f(x^*)\| \le \sqrt{ |\overline{\mathcal{I}}_*|2\lambda/\tau} $ from \eqref{eq-2-3}, we obtain that
\begin{align*}
\phi &\ge -\sqrt{ |\overline{\mathcal{I}}_*|2\lambda/\tau} \| x_{\overline{\mathcal{I}}_*} - x_{\overline{\mathcal{I}}_*}^* \| \ge - \sqrt{n2\lambda/\tau }\epsilon_* \stackrel{\eqref{eq-epsilon}}{>} -\lambda .
\end{align*}
Together with \eqref{eq-2-6-3}, we obtain that
\begin{align*}
f(x) + \lambda p(x) &\ge f(x^*) + \phi + \lambda p(x) \\
& > f(x^*) + \lambda p(x) - \lambda \\
&\ge f(x^*) + \lambda p(x^*).
\end{align*}
Both cases show the local optimality of $x^*$ in the region $N(x^*)$.

(iii) Now, we consider the case that $f$ is strongly convex. Denote $z^* = x^* - \tau \nabla f(x^*)$. Again, it follows from $x^*$ being a $\tau$-stationary point with $\tau > 0$ for any $x \in \Omega$, by \eqref{eq-2-1} it holds that
\begin{align}\label{eq-straight-1}
1/2\| x - z^*\|^2 + \tau \lambda p(x) \ge 1/2\| x^* - z^*\|^2 + \tau \lambda p(x^*)
\end{align}
It can be derived through straightforward calculations that
\begin{align*}
\| x - z^*\|^2 &= \| (x^* - \tau \nabla f(x^*)) - x  \|^2 \\
&= \|x^* -x \|^2 - 2\tau \left\langle \nabla f(x^*), x^* - x \right\rangle + \|\tau \nabla f(x^*) \|^2 ,\\
\| x^* - z^*\|^2 &= \| x^* - (x^* - \tau \nabla f(x^*))\|^2 =  \|\tau \nabla f(x^*) \|^2 ,
\end{align*}
which suffices to obtain that
\begin{align*}
\frac{1}{2}\| x - z^*\|^2 -\frac{1}{2}\| x^* - z^*\|^2 =\frac{1}{2}\|x^* -x \|^2-\tau\left\langle \nabla f(x^*), x^* - x \right\rangle.
\end{align*}
Together with \eqref{eq-straight-1}, we obtain that
\begin{align*}
\frac{1}{2}\|x^* -x \|^2-\tau\left\langle \nabla f(x^*), x^* - x \right\rangle + \tau\lambda p(x) \ge \tau \lambda p(x^*),
\end{align*}
which implies that
\begin{align}\label{eq-2-6-4}
\left\langle \nabla f(x^*), x - x^* \right\rangle + \lambda p(x) \ge -(1/{2\tau}) \|x - x^*\|^2 + \lambda p(x^*).
\end{align}
Since $f$ is strongly convex which defined in \eqref{eq-2-5} together with \eqref{eq-2-6-4}, for any $x \neq x^*$, we have
\begin{align*}
f(x) + \lambda p(x) &\ge  f(x^*) + \left\langle \nabla f(x^*), x - x^* \right\rangle + (\ell/2)\|x - x^*\|^2 + \lambda p(x) \\
&\ge f(x^*) + \frac{\ell - 1/\tau}{2}\|x - x^*\|^2 + \lambda p(x^*) \\
&\ge  f(x^*) + \lambda p(x^*),
\end{align*}
where the last inequality is based on $\tau \ge 1/\ell$. Clearly, if $\tau > 1/\ell$, then the last inequality holds strictly, implying that $x^*$ is a unique global minimizer.
\end{proof}

\section{Subspace Newton's method}
In this section, we propose a subspace Newton's method for solving problem \eqref{PBL0}, which is an extension of the algorithm in \cite{zhou2021newton}. Our method utilize the relationship between \eqref{eq-2-1} and \eqref{eq-2-7} then employs Newton's method to solve a series of stationary equations \eqref{eq-2-7}. That is, finding the $\tau$-stationary point can transform to finding a solution of nonlinear equation \eqref{eq-2-7}.

At iteration $k$, define $\Theta_k$, $\Gamma_k$, $\mathcal{I}_k$ as follows:
\begin{align}\label{eq-index}
{\Theta}_k = \Theta_{\tau}(x^k),\ {\Gamma}_k = \Gamma_{\tau}(x^k),\ {\Gamma}^u_k = \Gamma^u_{\tau}(x^k),\ {\Gamma}^l_k = \Gamma^l_{\tau}(x^k),\ {\mathcal{I}}_k = \mathcal{I}_{\tau}(x^k).
\end{align}
Denote
\begin{align*}
H^k :=
\begin{bmatrix}
\nabla_{\Theta_k,\Theta_k}^2 f(x^k)        & \nabla_{\Theta_k,\Gamma_k}^{2}f(x^k) & \nabla_{\Theta_k,\overline{\mathcal{I}}_{k}}^{2}f(x^k)\\
0 & I_{\Gamma_k}  & 0 \\
0 & 0  & I_{\overline{\mathcal{I}}_{k}} \\
\end{bmatrix}.
\end{align*}
We apply the Newton's method on \eqref{eq-2-7} to obtaining a direction $d^k$ that is
$$H^{k}d^k = -F_{\tau}(x;\Theta_k ; \Gamma_k ; \overline{\mathcal{I}}_k).$$
Specifically, $d^k$ satisfies
\begin{align}\label{direction}
&\nabla_{\Theta_k,\Theta_k}^{2}f(x^k)d_{\Theta_k}^k +\nabla_{\Theta_k,\Gamma_k}^{2}f(x^k)d_{\Gamma_k}^k +\nabla_{\Theta_k,\overline{\mathcal{I}}_k}^{2}f(x^k)d_{\overline{\mathcal{I}}_k}^k = - \nabla_{\Theta_k}f(x^k), \nonumber\\
&d_{\Gamma_k} = \left[\Pi_{\Omega}( x - \nabla f(x))\right]_{\Gamma_k} - x_{\Gamma_k} ,\\
&d_{\overline{\mathcal{I}}_k}^k = -x_{\overline{\mathcal{I}}_k}^k .\nonumber
\end{align}
For notational convenience, let
\begin{align*}
g^k := \nabla f(x^k),\ J_k := \mathcal{I}_{k-1}\backslash \mathcal{I}_{k}.
\end{align*}
To ensure that Newton's steps is descent direction and also a feasible direction, we need
compute the following criterion
\begin{align}\label{zerodecent}
\begin{cases}
& \left\langle g_{\mathcal{I}_{k}}^{k}, d_{\mathcal{I}_{k}}^{k} \right\rangle \leq -\delta \|d^k\|^2 + \|x^k_{\overline{\mathcal{I}}_{k}} \|^2/(4\tau)  ,\\
&|\mathcal{I}_k| \le \|x^k\|_0 ,\\
&x^k + d^k \in \Omega ,\\
&\Theta_k \setminus \mathcal{I}_{k-1} \neq \emptyset \ \mbox{or}\ \mathcal{I}_k =\mathcal{I}_{k-1} .
\end{cases}
\end{align}
Based on $d^k$ we use the modified Amijio line search to guarantee $f(\tilde{x}^k(\alpha)) < f(x^{k})$, where
\begin{align}\label{2.4}
\tilde{x}^k(\alpha):= \begin{bmatrix}
x^k_{\Theta_k} + \alpha d^k_{\Theta_k} \\
x^k_{\Gamma_k} + d^k_{\Gamma_k} \\
x^k_{\overline{\mathcal{I}}_k} + d^k_{\overline{\mathcal{I}}_k}
\end{bmatrix}
=\begin{bmatrix}
x^k_{\Theta_k} + \alpha d^k_{\Theta_k} \\
x^k_{\Gamma_k} + d^k_{\Gamma_k} \\
0
\end{bmatrix}
=\begin{bmatrix}
x^k_{\Theta_k} + \alpha d^k_{\Theta_k} \\
\left[\Pi_{\Omega}( x^k - g^k)\right]_{\Gamma_k } \\
0
\end{bmatrix}.
\end{align}
Note that if $x^k + d^k \in \Omega$ holds, for any $\alpha \in [0,1]$, there is $-l_{\Theta_k}\le x^k_{\Theta_k} + \alpha d^k_{\Theta_k} \le u_{\Theta_k}$. Hence, $\tilde{x}^k(\alpha) \subseteq \Omega$, for simply denote $\tilde{x}^k := \tilde{x}^k(\alpha) $.

If \eqref{zerodecent} fail, we take the projection gradient method to obtain a direction $d^k$ that is
\begin{align}\label{alg-2-3}
d_{\mathcal{I}_k}^k = \left[\Pi_{\Omega}( x^k - \tau g^k)\right]_{\mathcal{I}_k} - x_{\mathcal{I}_k}^k, d^{k}_{\overline{\mathcal{I}}_{k}} = -x^{k}_{\overline{\mathcal{I}}_{k}}.
\end{align}
For projection gradient step, we define
\begin{align}\label{eq-gradienxk}
\bar{x}^k := x^k + d^k .
\end{align}
It is obvious that, regardless of whether $x^{k+1}=\tilde{x}^k$ or $x^{k+1}=\bar{x}^k$, for the sequence $\{x^k\}$ the following conclusions hold true.
\begin{align}\label{support}
\text{supp} (x^{k+1}) \subset \mathcal{I}_k.
\end{align}
By \eqref{support}, it holds that
$$x^{k+1}_{\overline{\mathcal{I}}_{k}} = 0,\ x^{k}_{\overline{\mathcal{I}}_{k-1}} = 0.$$
It is obvious that
\begin{align}\label{2-1}
-d_{\overline{\mathcal{I}}_k}^k \equiv x_{\overline{\mathcal{I}}_k}^k =
\begin{bmatrix}
x_{\mathcal{I}_{k-1}\cap \overline{\mathcal{I}}_{k}}^k\\
x_{\overline{\mathcal{I}}_{k-1}\cap \overline{\mathcal{I}}_{k}}^k
\end{bmatrix}
=
\begin{bmatrix}
x_{\mathcal{I}_{k-1}\setminus \mathcal{I}_{k}}^k\\
0
\end{bmatrix}
=
\begin{bmatrix}
x_{J_k}^k\\
0
\end{bmatrix}
.
\end{align}
Note that, if $\mathcal{I}_{k} = \mathcal{I}_{k-1}$ there is $J_k= \emptyset$, hence $d_{\overline{\mathcal{I}}_k}^k = 0$. We can see that
\begin{align}\label{eq-bound}
\begin{cases}
x^{k+1}_{\Gamma^u_k} = \left[\Pi_{\Omega}(x^k - \tau \nabla f(x^k)) \right]_{\Gamma^u_k} = u_{\Gamma^u_k} ,\\
x^{k+1}_{\Gamma^l_k} = \left[\Pi_{\Omega}(x^k - \tau \nabla f(x^k)) \right]_{\Gamma^l_k} = -l_{\Gamma^l_k} .
\end{cases}
\end{align}
From above, it not hard to see that $x^k \in \Omega$ for any $k>0$.
The framework of the method is in Algorithm \ref{algorithm-NL0R}.
\begin{algorithm} \label{algorithm-NL0R}
\caption{Subspace Newton's method for the box constraint $\ell_0$-regularized optimization (BNL0R)}
\textbf{S0.} $\delta > 0$, $\lambda \in (0,\underline{\lambda})$, $\tau > 0$ satisfies \eqref{eq-tau}, $\sigma \in (0,1/2)$, $\beta \in (0,1)$, $x^0$, $\mathcal{I}_{-1} = \emptyset$, $k = 0$.
If $k < $ MaxIter or the termination conditions are met then stop, otherwise

\textbf{S1.} Compute $\mathcal{I}_{k}$, ${\Theta}_{k}$ and ${\Gamma}_{k}$ by \eqref{eq-index}. Set $\tilde{S}_{k}:= {\Theta}_{k}\setminus \mathcal{I}_{k-1}$. Go to \textbf{S2}.

\textbf{S2.} Update $d^k$ by the following two cases.
\begin{description}
\item{(\textbf{NM})} If \eqref{direction} is solvable and \eqref{zerodecent} is satisfies, update $d^k$ by \eqref{direction}, go to \textbf{S3}. \\
\item{(\textbf{PGM})} If \eqref{zerodecent} is fail, update $d^k$ by \eqref{alg-2-3}. Set $x^{k+1} = \bar{x}^k $ by \eqref{eq-gradienxk}, go to \textbf{S4}.
\end{description}

\textbf{S3} Find the smallest non-negative integer $m_k$ such that
$$ f(x^{k}(\beta^{m_k})) \le f(x^k) + \sigma \beta^{m_k}\left\langle g^k, d^k \right\rangle . $$
Set $\alpha_k = \beta^{m_k}$, $x^{k+1} = \tilde{x}^k (\alpha_k )$ by \eqref{2.4}, go to \textbf{S4}.

\textbf{S4.} Set $k = k+1$, go to \textbf{S1}.
\end{algorithm}

Note that, with the help of condition \eqref{zerodecent}, $x^{k+1}$ is feasible. And this condition also guarantee that function $\phi(x^k)$ in \eqref{BL0} is decreased. However, \eqref{zerodecent} may holds not always. If \eqref{zerodecent} fail, the projection gradient direction in \eqref{alg-2-3} is used and the feasibility as well as the descent property are still preserved which will be explained later.

The following Proposition demonstrates the descent property of the projection gradient direction.
\begin{proposition}
For any index set $\mathcal{F} \subseteq \mathbb{N}_{n}$, $x \in \Omega$ and $\mathbf{g} \in \mathbb{R}^n$. Define $d_{\mathcal{F}} = \left[\Pi_{\Omega}( x - \tau \mathbf{g})\right]_{\mathcal{F}} - x_{\mathcal{F}}$ , it holds that
\begin{align}\label{gamma-1}
\left\langle \mathbf{g}_{\mathcal{F}}, d_{\mathcal{F}} \right\rangle \le - \frac{1}{\tau} \| d_{\mathcal{F}} \|^2.
\end{align}
\end{proposition}
\begin{proof}
Denote $z = x - \tau \mathbf{g}$, by Proposition \ref{pro-projection} and $x \in \Omega$ it holds that
\begin{align*}
\left\langle  (x_{\mathcal{F}} - \tau \mathbf{g}_{\mathcal{F}}) - \left[\Pi_{\Omega}(z)\right]_{\mathcal{F}}, x_{\mathcal{F}} -  \left[\Pi_{\Omega}(z)\right]_{\mathcal{F}} \right\rangle \le 0 .
\end{align*}
Hence, we obtain that
\begin{align*}
\left\langle - \tau \mathbf{g}_{\mathcal{F}}, x_{\mathcal{F}} - \left[\Pi_{\Omega}(z)\right]_{\mathcal{F}} \right\rangle \le - \left\| x_{\mathcal{F}} - \left[\Pi_{\Omega}(z)\right]_{\mathcal{F}} \right\|^2 .
\end{align*}
The proof is complete.
\end{proof}

The next Lemma implies the descent property of the Newton step.
\begin{lemma}\label{lem-3-1}
For $d^k$ defined by \eqref{direction} and $x^k \in \Omega$, the following holds
\begin{align}\label{eq-dkHIJ}
\langle d^k, H^k d^k \rangle = \left\langle d_{\mathcal{I}_k \cup J_k}^k , H^k_{\mathcal{I}_k\cup J_k} d_{\mathcal{I}_k \cup J_k}^k \right\rangle .
\end{align}
Moreover, there is
\begin{align}\label{eq-3-1}
\left\langle d^k_{\mathcal{I}_k}, g^k_{\mathcal{I}_k} \right\rangle \le - \left\langle d_{\mathcal{I}_k \cup J_k}^k , H^k_{\mathcal{I}_k \cup J_k} d_{\mathcal{I}_k \cup J_k}^k \right\rangle + (1 - \frac{1}{\tau}) \|d_{\Gamma_k}^k\|^2 + \|d_{J_k}^k\|^2 .
\end{align}
\end{lemma}
\begin{proof}
From \eqref{2-1}, there is
\begin{align*}
d^k =\left[ \left(d_{\mathcal{I}_k}^k \right)^{\top}, \left(d_{\overline{\mathcal{I}}_k}^k\right)^{\top} \right]^{\top} = \left[ \left(d_{\mathcal{I}_k} ^k \right)^{\top}, \left(x_{J_k}^k\right)^{\top},0 \right]^{\top} =\left[ \left(d_{\mathcal{I}_k \cup {J_k}} ^k \right)^{\top},0 \right]^{\top} .
\end{align*}
Together with \eqref{direction}, it holds that
\begin{align}\label{eq-dHd}
&\left\langle d^k , H^k d^k \right\rangle \nonumber\\
=& \begin{bmatrix}
d_{\Theta_k}^k\\
d_{\Gamma_k}^k \\
d_{\overline{\mathcal{I}}_k}^k
\end{bmatrix}^{\top}
\begin{bmatrix}
\nabla_{\Theta_k,\Theta_k}^{2}f(x^k) d_{\Theta_k}^k   +  \nabla_{\Theta_k,\Gamma_{k}}^{2}f(x^k) d_{\Gamma_k}^k  +  \nabla_{\Theta_k,\overline{\mathcal{I}}_k}^{2}f(x^k) d_{\overline{\mathcal{I}}_k}^k\\
d_{\Gamma_k}^k    \\
d_{\overline{\mathcal{I}}_k}^k
\end{bmatrix}\nonumber\\
=& \begin{bmatrix}
d_{\Theta_k}^k\\
d_{\Gamma_k}^k \\
x_{J_k}^k \\
0
\end{bmatrix}^{\top}
\begin{bmatrix}
\nabla_{\Theta_k,\Theta_k}^{2}f(x^k) d_{\Theta_k}^k   +  \nabla_{\Theta_k,\Gamma_{k}}^{2}f(x^k) d_{\Gamma_k}^k  +  \nabla_{\Theta_k,J_k}^{2}f(x^k) d_{J_k}^k\\
d_{\Gamma_k}^k    \\
x_{J_k}^k \\
0
\end{bmatrix}\nonumber\\
= &\left\langle d_{\mathcal{I}_k \cup J_k}^k , H^k_{\mathcal{I}_k\cup J_k} d_{\mathcal{I}_k \cup J_k}^k \right\rangle ,
\end{align}
where
\begin{align}\label{eq-HIJ}
H_{\mathcal{I}_k \cup J_k}^k
=
\begin{bmatrix}
\nabla_{\Theta_k, \Theta_k}^{2}f(x^k)   &  \nabla_{\Theta_k,\Gamma_k}^{2}f(x^k)  &  \nabla_{\Theta_k,J_{k}}^{2}f(x^k) \\
0   &    I_{\Gamma_k}     &0\\
0   &    0     &I_{J_K}
\end{bmatrix}
.
\end{align}
Hence, the proof of \eqref{eq-dkHIJ} is complete. By \eqref{direction}, we obtain that
\begin{align*}
-g_{\Theta_k}^k = \nabla_{\Theta_k,\Theta_k}^{2}f(x^k) d_{\Theta_k}^k   +  \nabla_{\Theta_k,\Gamma_{k}}^{2}f(x^k) d_{\Gamma_k}^k  +  \nabla_{\Theta_k,\overline{\mathcal{I}}_k}^{2}f(x^k) d_{\overline{\mathcal{I}}_k}^k .
\end{align*}
Together with \eqref{eq-dHd}, we obtain that
\begin{align*}
\left\langle d_{\mathcal{I}_k \cup J_k}^k , H^k_{\mathcal{I}_k\cup J_k} d_{\mathcal{I}_k \cup J_k}^k \right\rangle = - \left\langle d_{\Theta_k}^k, g_{\Theta_k}^k \right\rangle + \| d_{\Gamma_k}^k \|^2 + \| d_{J_k}^k \|^2 ,
\end{align*}
thus
\begin{align}\label{theta-1}
\left\langle d_{\Theta_k}^k, g_{\Theta_k}^k \right\rangle = - \left\langle d_{\mathcal{I}_k \cup J_k}^k , H^k_{\mathcal{I}_k\cup J_k} d_{\mathcal{I}_k \cup J_k}^k \right\rangle + \| d_{\Gamma_k}^k \|^2 + \| d_{J_k}^k \|^2 .
\end{align}
Combining \eqref{theta-1} and \eqref{gamma-1}, we obtain that
\begin{align*}
\left\langle d^k_{\mathcal{I}_k}, g^k_{\mathcal{I}_k} \right\rangle &= \left\langle d_{\Theta_k}^k, g_{\Theta_k}^k \right\rangle + \left\langle g_{\Gamma_k}^k, d_{\Gamma_k}^k \right\rangle\\
&\le - \left\langle d_{\mathcal{I}_k \cup J_k}^k , H^k_{\mathcal{I}_k\cup J_k} d_{\mathcal{I}_k \cup J_k}^k \right\rangle +  \|d_{\Gamma_k}^k\|^2 + \|d_{J_k}^k\|^2  - \frac{1}{\tau} \|d_{\Gamma_k}^k\|^2 \\
&= - \left\langle d_{\mathcal{I}_k \cup J_k}^k , H^k_{\mathcal{I}_k\cup J_k} d_{\mathcal{I}_k \cup J_k}^k \right\rangle + (1-\frac{1}{\tau}) \|d_{\Gamma_k}^k\|^2 + \|d_{J_k}^k\|^2  .
\end{align*}
The proof is complete.
\end{proof}
Lemma \ref{lem-3-1} indicates that if $H^k_{\mathcal{I}_k\cup J_k}$ has positive lower and upper bounds, then $d^k$ is a decent direction under some properly chosen $\delta$ and $\tau$. It is obvious that $H^k_{\mathcal{I}_k\cup J_k}$ being bounded from below can be guaranteed by some assumptions, such as the strong convexity of $f$.

\section{Convergence analysis}
Next, we analyze the global convergence of the algorithm. First, we will prove the descent property of the Newton's step and the finite termination of the line search. Finally, we will prove the global convergence of algorithm \ref{algorithm-NL0R} in a Theorem. We need to define following parameters
\begin{align}\label{2-1-1}
&\overline{\alpha} := \min \left\{\frac{1-2\sigma}{L/\delta-\sigma}, \frac{2(1-\sigma)\delta}{L},1 \right\} ,\nonumber \\
&\overline{\tau} := \min \left\{\frac{2\overline{\alpha}\delta \beta}{nL^2}, \frac{\overline{\alpha} \beta}{n},\frac{1}{4L},  \frac{2(1-\sigma)}{L }, \frac{a}{2\lambda}, \frac{a}{2\max\limits_{i\in \mathbb{N}_{n}} |\nabla_i f(0)|} \right\} \nonumber ,\\
&\rho := \frac{2\delta - n\tau L^2}{2}.
\end{align}
The first Lemma demonstrates that the iteration direction $ d_k $ obtained from \eqref{direction} is a descent direction.
\begin{lemma}\label{lem-2-1}
{\rm(Descent property)} Let $f$ be strongly smooth with $L>0$, $\overline{\tau},\rho$ be defined by \eqref{2-1-1}. If $d^k$ is defined by \eqref{direction}. Then, for any $\tau  \in(0,\overline{\tau})$, it holds that $\rho >0$ and
\begin{align}\label{lem-2-1-0}
\left\langle g^{k},d^{k}\right\rangle \le -\rho\|d^{k}\|^{2} .
\end{align}
\end{lemma}
\begin{proof}
It follows from \eqref{2-1-1} that $\overline{\alpha} \le 1$ and thus $\overline{\alpha} \beta < 1$ due to $\beta \in (0, 1)$. Hence $\overline{\tau} \le 2\delta/(nL^2)$, immediately showing that $\rho > 0$ if $\tau \in (0,\overline{\tau})$ . By \eqref{direction} and \eqref{2-1}, it holds that
\begin{align}\label{lem-2-1-1}
\|g_{\Theta_k}^k\| &= \|\nabla_{\Theta_k,\Theta_k}^{2}f(x^k)d_{\Theta_k}^k  + \nabla_{\Theta_k,\Gamma_k}^{2}f(x^k)d_{\Gamma_k}^k + \nabla_{\Theta_k, {J_k}}f(x^k)  d_{J_k}^k\|   \nonumber\\
&= \| \nabla^2_{\Theta_k, \mathcal{I}_k \cup {J_k}}f(x^k)  d_{ \mathcal{I}_k \cup {J_k}}^k\|  \nonumber\\
&\le L\| d_{\mathcal{I}_k \cup {J_k}}^k \|\nonumber\\
&= L\|d^k\|  .
\end{align}
By replacing set $T_k$ and $S_k$ with $ \mathcal{I}_k$ and $\tilde{S}_k$ in \cite[Lemma 3]{zhou2021newton}, the subsequent proof becomes analogous to that of \cite[Lemma 3]{zhou2021newton}. Thus, this part of the proof is omitted for the sake of brevity.
\end{proof}
Next Lemma shows that $\alpha_k$ exists and has a lower bound away from zero. This implies that our step size will not be too small, ensuring that it is well defined.
\begin{lemma}\label{lem-2-2} Let $f$ be strongly smooth with respect to the constant $L > 0$, and set $\overline{\tau}$, $\overline{\alpha}$ as above. If $d^k$ is defined by \eqref{direction}. Then, for any $k \ge 0$ and $0 < \tau \le \min \left\{\frac{\alpha \delta}{nL^2}, \frac{\alpha}{n}, \frac{1}{4L}, \frac{2(1-\sigma)}{L } \right\}$,
$0 < \alpha \le \overline{\alpha}$,  $0 < \delta \le \min \{1, 2L\}$ , $\sigma \in (0, 1/2)$,
the following inequality holds
\begin{align}\label{lem-2-2-1}
f(x^k(\alpha)) \le f(x^k) + \sigma \alpha \left\langle g^k, d^k \right\rangle.
\end{align}
Moreover, for any $\tau \in (0, \overline{\tau})$, we have
\begin{align}\label{lem-2-2-4}
\inf\limits_{k \ge 0} \{\alpha_k\} \ge \beta \overline{\alpha} > 0.
\end{align}
\end{lemma}
\begin{proof}
If $0 < \alpha \le \overline{\alpha}$ and $0 < \delta \le \min \{1, 2L\}$. By \eqref{2-1-1}, it holds that
\begin{align}\label{lem-2-2-2}
\alpha \le \frac{2(1-\sigma)\delta}{L}, \quad \alpha \le \frac{1-2\sigma}{L/\delta - \sigma} \le \frac{1-2\sigma}{\max\{0, L - \sigma\}}.
\end{align}
By the strong smoothness of $f$ in \eqref{eq-2-4}, it holds that
\begin{align*}
&\quad \ 2f(x^k(\alpha)) - 2f(x^k) - 2\alpha \sigma \left\langle g^k, d^k \right\rangle \\
&\le 2\left\langle g^k, x^k(\alpha) - x^k \right\rangle + L\|x^k(\alpha) - x^k\|^2 - 2\alpha \sigma \left\langle g^k, d^k \right\rangle .
\end{align*}
By \eqref{2.4}, one has
\begin{align*}
2\left\langle g^k, x^k(\alpha) - x^k \right\rangle &= 2\alpha \left\langle g^k_{\Theta_k}, d^k_{\Theta_k} \right\rangle + 2\left\langle g^k_{\Gamma_k}, d^k_{\Gamma_k} \right\rangle -  2\left\langle g^k_{\overline{\mathcal{I}}_k}, x^k_{\overline{\mathcal{I}}_k} \right\rangle ,\\
L\|x^k(\alpha) - x^k\|^2 &= L\alpha^2\|d^k_{\Theta_k}\|^2 + L\|d^k_{\Gamma_k}\|^2 + L \|x^k_{\overline{\mathcal{I}}_k}\|^2.
\end{align*}
Then, we get
\begin{align*}
&\quad \  2\left\langle g^k, x^k(\alpha) - x^k \right\rangle + L\|x^k(\alpha) - x^k\|^2 - 2\alpha \sigma \left\langle g^k, d^k \right\rangle \\
&= 2\alpha \left\langle g^k_{\Theta_k}, d^k_{\Theta_k} \right\rangle + 2\left\langle g^k_{\Gamma_k}, d^k_{\Gamma_k} \right\rangle  - 2\left\langle g^k_{\overline{\mathcal{I}}_k}, x^k_{\overline{\mathcal{I}}_k} \right\rangle - 2\alpha \sigma \left\langle g^k, d^k \right\rangle \\
&\quad \,+ L\left[ \alpha^2\|d^k_{\Theta_k}\|^2 + \|d^k_{\Gamma_k}\|^2 + \|x^k_{\overline{\mathcal{I}}_k}\|^2 \right] .
\end{align*}
Through simple calculations, by \eqref{direction} and \eqref{gamma-1}, it can be obtained that
\begin{align*}
&\quad \ 2\alpha \left\langle g^k_{\Theta_k}, d^k_{\Theta_k} \right\rangle + 2\left\langle g^k_{\Gamma_k}, d^k_{\Gamma_k} \right\rangle  - 2\left\langle g^k_{\overline{\mathcal{I}}_k}, x^k_{\overline{\mathcal{I}}_k} \right\rangle - 2\alpha \sigma \left\langle g^k, d^k \right\rangle \\
&= 2\alpha (1-\sigma)\left\langle g^k_{\Theta_k}, d^k_{\Theta_k} \right\rangle + 2(1-\alpha \sigma)\left\langle g^k_{\Gamma_k}, d^k_{\Gamma_k} \right\rangle  - 2(1-\alpha \sigma)\left\langle g^k_{\overline{\mathcal{I}}_k}, x^k_{\overline{\mathcal{I}}_k} \right\rangle \\
&\le 2\alpha (1-\sigma)\left\langle g^k_{\Theta_k}, d^k_{\Theta_k} \right\rangle + \frac{2(1-\alpha \sigma)}{\tau}\| d^k_{\Gamma_k}\|^2  - 2(1-\alpha \sigma)\left\langle g^k_{\overline{\mathcal{I}}_k}, x^k_{\overline{\mathcal{I}}_k} \right\rangle .
\end{align*}
From above together with \eqref{2-1} and $\frac{2(1-\alpha \sigma)}{L } \ge \frac{2(1- \sigma)}{L } \ge \tau$, we obtain that
\begin{align*}
&\quad \ 2f(x^k(\alpha)) - 2f(x^k) - 2\alpha \sigma \left\langle g^k, d^k \right\rangle \\
&\le 2\alpha (1-\sigma)\left\langle g^k_{\Theta_k}, d^k_{\Theta_k} \right\rangle + 2(1-\alpha \sigma)\left\langle g^k_{\Gamma_k}, d^k_{\Gamma_k} \right\rangle  - 2(1-\alpha \sigma)\left\langle g^k_{\overline{\mathcal{I}}_k}, x^k_{\overline{\mathcal{I}}_k} \right\rangle \\
&\quad + L\left[ \alpha^2\|d^k_{\Theta_k}\|^2 + \|d^k_{\Gamma_k}\|^2 + \|x^k_{\overline{\mathcal{I}}_k}\|^2 \right] \\
&\le 2\alpha (1-\sigma)\left\langle g^k_{\mathcal{I}_k}, d^k_{\mathcal{I}_k} \right\rangle - 2(1-\alpha \sigma)\left\langle g^k_{\overline{\mathcal{I}}_k}, x^k_{\overline{\mathcal{I}}_k} \right\rangle \\
&\quad + L\left[ \alpha^2\|d^k_{\Theta_k}\|^2 + \left(1 - \frac{2(1-\alpha \sigma)}{L \tau} \right) \|d^k_{\Gamma_k}\|^2 + \|x^k_{\overline{\mathcal{I}}_k}\|^2  \right]\\
&\le 2\alpha (1-\sigma)\left\langle g^k_{\mathcal{I}_k}, d^k_{\mathcal{I}_k} \right\rangle - 2(1-\alpha \sigma)\left\langle g^k_{\overline{\mathcal{I}}_k}, x^k_{\overline{\mathcal{I}}_k} \right\rangle  + L\left[ \alpha^2\|d^k_{\Theta_k}\|^2 + \|x^k_{\overline{\mathcal{I}}_k}\|^2  \right] \\
&= 2\alpha (1-\sigma)\left\langle g^k_{\mathcal{I}_k}, d^k_{\mathcal{I}_k} \right\rangle - 2(1-\alpha \sigma)\left\langle g^k_{J_k}, x^k_{J_k} \right\rangle  + L\left[ \alpha^2\|d^k_{\Theta_k}\|^2 + \|x^k_{\overline{\mathcal{I}}_k}\|^2  \right]
=: \psi.
\end{align*}
Hence, the function $\psi$ is the same as it in \cite[Lemma 4]{zhou2021newton}. By employing a proof analogous to that in \cite[Lemma 4]{zhou2021newton}, we can derive the conclusion of this Theorem. Therefore, the proof is omitted here for brevity.
\end{proof}
The following lemma demonstrates that the sequence $\{x^k\}$ does not diverge.
\begin{lemma}\label{lem-2-3}
Let $f$ be strongly smooth with the constant $L > 0$, and let $\overline{\tau}$ be defined as above. The sequence $\{x^k\}$ is generated by the Algorithm \ref{algorithm-NL0R} with $\tau \in (0, \overline{\tau})$ and $\delta \in (0, \min \{1, 2L\})$, $\sigma \in (0, 1/2)$. Then, the sequence $\{\phi(x^k)\}$ is strictly decreasing and
\begin{align}\label{lem-2-3-1}
\lim\limits_{k\to\infty} \max \{\|F_{\tau}(x^k;\Theta_k; \Gamma_k;\overline{\mathcal{I}}_k)\|,\  \|x^{k+1}-x^{k}\|,\  \|g_{\Theta_k}^k\|,\ \|\tilde{x}^k - \bar{x}^k\|,\  \|{d}^k\| \} =0.
\end{align}
\end{lemma}
\begin{proof}
It is obvious that $ \min\limits_{y\in \mathbb{R}^n} \frac{1}{2}\|y - (x^k-\tau g^k)\|^2 + \tau \lambda \|y\|_0 + \delta_{\Omega}(y) $ is equivalent to the following problem
\begin{align*}
\min\limits_{y\in \Omega} f(x^k) + \langle g^k,y-x^k \rangle + \frac{1}{2\tau}\|y-x^k\|^2 + \lambda \|y\|_0 .
\end{align*}
If $d^k$ is update from \eqref{alg-2-3}, $x^{k+1} = \bar{x}^k= x^k + d^k$, then it holds that
\begin{align}\label{eq-2-3-min}
x^{k+1} \in \arg \min\limits_{y\in \Omega} f(x^k) + \langle g^k,y-x^k \rangle + \frac{1}{2\tau}\|y-x^k\|^2 + \lambda \|y\|_0 .
\end{align}
By $\tau \in (0,\overline{\tau})$ and \eqref{2-1-1}, there is $1/\tau \ge 4L$. Together with \eqref{eq-2-4} it holds that
\begin{align}
\phi(x^{k+1}) &= f(x^{k+1}) + \lambda \|x^{k+1}\|_0 \nonumber\\
&\le \stackrel{a}{\overbrace{f(x^k) + \langle g^k,x^{k+1}-x^k \rangle + \frac{L}{2}\|x^{k+1}-x^k\|^2 + \lambda \|x^{k+1}\|_0 }} \label{eq-ihtdcrease1}\\
&\le {\underbrace{f(x^k) + \langle g^k,x^{k+1}-x^k \rangle + \frac{1}{2\tau}\|x^{k+1}-x^k\|^2 + \lambda \|x^{k+1}\|_0 }_b} \label{eq-ihtdcrease2}\\
&\le f(x^{k}) + \lambda \|x^{k}\|_0 = \phi(x^k), \nonumber
\end{align}
where the last inequality is by letting $y = x^k$ in \eqref{eq-2-3-min}. Moreover, by $b-a$ from \eqref{eq-ihtdcrease2} and \eqref{eq-ihtdcrease1}, it holds that
\begin{align}\label{eq-2-3-decent}
\phi(x^k) - \phi(x^{k+1}) \ge \frac{1/\tau - L}{2}\|x^{k+1}-x^k\|^2 \ge \frac{3L}{2}\|x^{k+1}-x^k\|^2.
\end{align}
If $d^k$ is update from \eqref{direction} then $x^{k+1} = \tilde{x}^k= \tilde{x}^k(\alpha_k)$ from \eqref{2.4}. Combine with \eqref{lem-2-1-0} and Lemma \ref{lem-2-2}, denoting $c_0:=\sigma \beta \overline{\alpha} \rho$, we obtain that
\begin{align}\label{newtonfun}
f(x^{k+1}) - f(x^k)\le \sigma \alpha_k \left\langle g^k, d^k \right\rangle \le - \sigma \alpha_k \rho \|d^k\|^2 \le -c_0\|d^k\|^2.
\end{align}
By \eqref{zerodecent} and \eqref{support}, there is $\|x^{k+1}\|_0 \le |\mathcal{I}_k| \le \|x^{k}\|_0 $. Together with \eqref{newtonfun}, we obtain that
\begin{align*}
f(x^{k+1}) - f(x^k) + \lambda(\|x^{k+1}\|_0  - \|x^{k}\|_0) \le -c_0\|d^k\|^2.
\end{align*}
Hence, there is
\begin{align}\label{newtonphi}
\phi(x^{k+1} )-\phi(x^{k} ) \leq - c_0 \|d^{k} \|^{2} .
\end{align}
Denote $\eta = \min (c_0, 3L/2)$, by \eqref{eq-2-3-decent} and \eqref{newtonphi}, we obtain that
\begin{align*}
\phi(x^{k+1} )-\phi(x^{k} ) \leq - \eta \|d^{k} \|^{2} .
\end{align*}
This implies that $\{\phi(x^k)\}$ is non-increasing and moreover it holds that
\begin{align*}
\sum_{k=0}^{\infty} \eta \|d^{k} \|^{2}  \leq \sum_{k=0}^{\infty}\left[\phi(x^{k} )-\phi(x^{k+1} )\right]  =\left[\phi(x^{0} )-\lim _{k \to \infty} \phi(x^{k} )\right] < +\infty ,
\end{align*}
where the last inequality is due to $f$ being bounded from below. Hence $ \|d^{k} \| \to 0$,  which suffices to obtain that $ \left\|x^{k+1}-x^{k} \right\| \rightarrow 0$ because of
$$ \left\|x^{k+1}-x^{k} \right\|^2 = \alpha_k \left\|d^k_{\Theta_k}\right\|^2 + \left\|d^k_{\Gamma_k}\right\|^2 + \left\|d^k_{\overline{\mathcal{I}}_k}\right\|^2\le \left\|d^k\right\|^2 .$$
The above relation also indicates that $\left\|d_{\Gamma_{k}}^{k} \right\|^{2} \rightarrow 0$ and $\left\|x_{\overline{\mathcal{I}}_{k}}^{k} \right\|^{2} \rightarrow 0$. In addition, if $d^{k}$ is taken from \eqref{direction}, then $ \|g_{\Theta_{k}}^{k} \| \leq L \|d^{k} \| \rightarrow 0$ by \eqref{lem-2-1-1}. If it is taken from \eqref{alg-2-3} then $ \left\|\tau g_{\Theta_{k}}^{k} \right\|= \left\|\left[\Pi_{\Omega}(x^k-\tau g^k)\right]_{\Theta_k} - x^k_{\Theta_k}\right\| = \left\|d_{\Theta_{k}}^{k} \right\| \to 0$. These results together with \eqref{eq-2-7} imply that $\left\|F_{\tau}\left(x^{k} ; \Theta_{k}; \Gamma_k; \overline{\mathcal{I}}_k \right)\right\|^{2}= \left\|g_{\Theta_{k}}^{k} \right\|^{2} + \left\|d_{\Gamma_{k}}^{k} \right\|^{2}+ \left\|x_{\overline{\mathcal{I}}_{k}}^{k} \right\|^{2} \rightarrow 0$.

Moreover, if the iteration number of $d^{k}$ from \eqref{direction} is infinite, then $\|\tilde{x}^k - x^k\| \to 0$ by Algorithm \ref{algorithm-NL0R}. Similarly, if the iteration number of $d^{k}$ from \eqref{alg-2-3} is infinite, then $\|\bar{x}^k - x^k\| \to 0$. Hence, we obtain that
\begin{align*}
\|\tilde{x}^k - \bar{x}^k\| \le \|\tilde{x}^k - x^k\| + \|\bar{x}^k - x^k\| \to 0 .
\end{align*}
The proof is complete.
\end{proof}
\begin{remark}
In order to make the most use of the Newton step, we can relax the condition in \eqref{zerodecent} to that of the following
\begin{align}\label{zerodecent2}
\begin{cases}
&2 \left\langle g_{\mathcal{I}_{k}}^{k}, d_{\mathcal{I}_{k}}^{k} \right\rangle \leq -\delta \|d^k\|^2 + \|x^k_{\overline{\mathcal{I}}_{k}}\|^2/(4\tau) , \\
&\sigma \beta \overline{\alpha} \left\langle g^k, d^k \right\rangle + \lambda (|\mathcal{I}_k| - \|x^k\|_0 ) \le \frac{1}{2}\sigma \beta \overline{\alpha} \left\langle g^k, d^k \right\rangle ,\\
&x^k + d^k \in \Omega ,\\
&\Theta_k \setminus \mathcal{I}_{k-1} \neq \emptyset \ \mbox{or}\ \mathcal{I}_k =\mathcal{I}_{k-1}.
\end{cases}
\end{align}
Clearly, when \eqref{zerodecent} holds and $\left\langle g^k, d^k \right\rangle \le 0$ then condition \eqref{zerodecent2} holds. Moreover, by Lemma \ref{lem-2-2-1}, we obtain that as long as condition \eqref{zerodecent2} is met, \eqref{newtonphi} remains valid in this case.
Therefore this relaxed condition is well-defined.
\end{remark}
The following lemma shows that for sequence $\{x^k\}$, the magnitude of nonzero component $x^k_i$ cannot be too small and $\mathcal{I}_k$ changes only finitely often for sufficiently large $k$.
\begin{lemma}\label{lem-constantindex}
Let $\{x^k\}$ be generated by the Algorithm \ref{algorithm-NL0R} with $\tau \in (0, \overline{\tau})$ and $\delta \in (0, \min \{1, 2L\})$, $\sigma \in (0, 1/2)$. Then, the following holds.
\begin{description}
\item{(i)} For all $k\ge 0$,
\begin{align}\label{eq-barxk}
|\bar{x}^{k+1}_{i}| \ge \sqrt{2\tau\lambda},\ \mbox{if}\ i\in \mathcal{I}_k.
\end{align}
Moreover, for sufficiently large $k$, there exist $\sqrt{2\tau\lambda} \ge \eta > 0$ such that
\begin{align}\label{eq-tildexk}
|\tilde{x}^{k+1}_{i}| \ge \eta,\ \mbox{if}\ i\in \mathcal{I}_k.
\end{align}

\item{(ii)} For sufficiently large $k$, there is
    \begin{align}\label{eq-notsmall}
    \|x^{k+1} - x^k\| \ge \eta, \ \mbox{if}\ \mathcal{I}_k \neq \mathcal{I}_{k-1}.
    \end{align}

\item{(iii)} For sufficiently large $k$, $\mathcal{I}_k \equiv \mathcal{I}_{k-1} \equiv \mathcal{I}_{\infty}$ and ${\rm{supp}} (x^k) = \mathcal{I}_{\infty}$.
\end{description}
\end{lemma}
\begin{proof}
(i) By the definition of $\mathcal{I}_k$ in \eqref{eq-index}, for any $i \in \mathcal{I}_k$ and $k> 0$ it holds that  $|\bar{x}^{k}_i| = \left|\left[\Pi_{\Omega}(x^k - \tau g^k)\right]_i \right| \ge \sqrt{2\tau\lambda}$ by \eqref{eq-2-2-2}. Hence, \eqref{eq-barxk} is true. By \eqref{lem-2-3-1} there is $\|\bar{x}^k - \tilde{x}^k\| \to 0$. Thus, for any $\epsilon > 0$, there exist $K> 0$ such that for any $k>K$ there is $\|\bar{x}^k - \tilde{x}^k\| \le \epsilon$. Hence, there is
\begin{align*}
|\bar{x}^k_i| - |\tilde{x}^k_i| \le \|\bar{x}^k_i - \tilde{x}^k_i\| \le \|\bar{x}^k - \tilde{x}^k\| \le \epsilon .
\end{align*}
Which implies that $|\tilde{x}^k_i| \ge |\bar{x}^k_i| - \epsilon$. Together with \eqref{eq-barxk}, there is $|\tilde{x}^k_i| \ge \sqrt{2\tau\lambda} - \epsilon$. Hence, \eqref{eq-tildexk} holds true by letting $\epsilon < \sqrt{2\tau\lambda}$ and $\eta = \sqrt{2\tau\lambda} - \epsilon$.

(ii) From claim (i) together with Algorithm \ref{algorithm-NL0R}, we know that for sufficiently large $k$ and $i\in \mathcal{I}_{k-1}$ or $i\in \mathcal{I}_k$ there is $|x^{k}_i| \ge \eta$ or $|x^{k+1}_i| \ge \eta$. Hence, suppose that $\mathcal{I}_k \neq \mathcal{I}_{k-1}$, we obtain that
\begin{align*}
\|x^{k+1} - x^k\| \ge |x^{k+1}_i - |x^k_i | = \max (|x^{k+1}_i, |x^k_i |) \ge \eta.
\end{align*}

(iii) By \eqref{lem-2-3-1} there is $\|x^{k+1} - {x}^k\| \to 0$ which together with \eqref{eq-notsmall} implies that $\mathcal{I}_k$ does not change when $k$ is sufficient large. That is $\mathcal{I}_k = \mathcal{I}_{k-1} = \mathcal{I}_{\infty}$. From Algorithm \ref{algorithm-NL0R}, we have $x^{k+1} = \bar{x}^k$ or $x^{k+1} = \tilde{x}^k$. Together with claim (i), for sufficiently large $k$, it holds that ${\rm{supp}} (x^k) = \mathcal{I}_{k} = \mathcal{I}_{\infty}$.
Hence, the proof is complete.
\end{proof}
We hypothesize that the index set $\mathcal{I}_{k}$ will stabilize after a finite number of iterations. Prior to this, we investigate a particular type of box-constrained optimization problems. We present some technical results about a projection gradient method for this type of optimization problems that will be subsequently used in this paper. Denote $\mathcal{I}' \subseteq [1,\cdots,n]$, consider the following problem
\begin{align}\label{BL02}
\min\limits_{x\in \mathbb{R}^n} &f(x)  \nonumber \\
\mbox{s.t.} & -l_i \le x_i \le u_i,\ i\in \mathcal{I'}; \nonumber \\
& x_i = 0,\ i\in \overline{\mathcal{I'}}.
\end{align}
By substituting the vector function of the variational inequality in \cite{Pang1990} with the gradient $\nabla f(x)$, we can obtain the following Theorem.
\begin{proposition}{\rm \cite[Proposition 2.3]{Pang1990}}\label{pro-station}
Let $x\in \mathbb{X}$ and $\mathbb{X} \subseteq \mathbb{R}^n$ is a nonempty closed convex set. Then, $\| x^{*} - \Pi_{\mathbb{X}}(x^{*} - \tau \nabla f(x^{*})) \| =0$ if and only if $x^*$ solves the following problem
\begin{align}\label{VIP}
\left\langle \nabla f(x^*), x - x^* \right\rangle \ge 0,\ \text{for any}\ x\in \mathbb{X}.
\end{align}
\end{proposition}
\begin{definition}
Let $f: \mathbb{R}^{n} \to \mathbb{R}$. We call $x^*$ is an stationary point if \eqref{VIP} holds for
\begin{align*}
\min\limits_{x\in \mathbb{R}^n} f(x)\ {\rm{s.t.}}\ x\in \mathbb{X}.
\end{align*}
\end{definition}
By \cite[Proposition 1.2.3]{facchinei2003finite}, define $ \mathbb{X} = \left\{ x\in \Omega \, \mid \, x_i = 0,\, i\in \overline{\mathcal{I'}} \right\} $, we know that if \eqref{VIP} holds then $x^*$ is an stationary point of \eqref{BL02}. The following Lemma serves as a prelude to demonstrate that index set $\Theta_k$ and $\Gamma_k$ can be finally identified.
\begin{lemma}\label{lem-kkt}
Consider problem \eqref{BL02}. Let $x^*$ satisfies $\| x^{*} - \Pi_{\mathbb{X}}(x^{*} - \tau \nabla f(x^{*})) \| =0$. If strict complementarity and Abadie's CQ holds at $x^*$, then we have
\begin{align*}
|x^*_i -\tau \nabla_i f(x^*) | \neq l_i\ \text{or}\ u_i,\ \text{for any}\ i\in \mathcal{I}'.
\end{align*}
\end{lemma}
\begin{proof}
The Lagrange function of \eqref{BL02} is
$$ \mathcal{L}(x) = f(x) + \sum_{i=1}^{|\mathcal{I'}|}\xi_i (x_i - u_i) -  \sum_{i=1}^{|\mathcal{I'}|}\beta_i (x_i + l_i) + \sum_{i=1}^{|\overline{\mathcal{I'}}|}\nu_i x_i .$$
The derivatives of $\mathcal{L}(x)$ is
$$ \nabla_{x} \mathcal{L}(x) = \nabla f(x) + \xi - \beta + \nu. $$
By $\| x^{*} - \Pi_{\mathbb{X}}(x^{*} - \tau \nabla f(x^{*})) \| =0$ and Proposition \ref{pro-station}, we obtain that $x^*$ is a stationary point of \eqref{BL02}. Together with Abadie's CQ and \cite[Proposition 1.3.4]{facchinei2003finite} we obtain that \eqref{VIP} is the KKT conditions of \eqref{BL02}. Therefor, it holds that
$$ \nabla_{x^*} \mathcal{L}(x^*) = \nabla f(x^*) + \xi - \beta + \nu = 0. $$
By the strict complementarity on $x^*$, we obtain that
\begin{align*}
\begin{cases}
\xi_i > 0,\ \beta_i = 0,\ \nu_i = 0, \ &\text{if}\ x^*_i - u_i = 0,\ i\in \mathcal{I}' ;\\
\beta_i > 0,\ \xi_i = 0,\ \nu_i = 0, \ &\text{if}\ x^*_i + l_i = 0,\ i\in \mathcal{I}'  ; \\
\xi_i = \beta_i = \nu_i = 0, \ &\text{if}\ x^*_i\in (-l_i, u_i),\ i\in \mathcal{I}' ;\\
\nu_i \in \mathbb{R},\ \beta_i = 0=\xi_i = 0, \ &\text{if}\ x^*= 0,\ i\in \overline{\mathcal{I}'}  .
\end{cases}
\end{align*}
By $\nabla_{x} \mathcal{L}(x^*) = 0$, it holds that
\begin{align*}
\begin{cases}
\nabla_i f(x^*) = -\xi_i < 0, \ &\text{if}\ x^*_i - u_i = 0,\ i\in \mathcal{I}'  ;\\
\nabla_i f(x^*) = \beta_i > 0, \ &\text{if}\ x^*_i + l_i = 0,\ i\in \mathcal{I}'  ; \\
\nabla_i f(x^*) = 0, \ &\text{if}\ x^*_i\in (-l_i, u_i),\ i\in \mathcal{I}' ;\\
\nabla_i f(x^*) = -\nu_i , \ &\text{if}\ x^*= 0,\ i\in \overline{\mathcal{I}'}  .
\end{cases}
\end{align*}
We suppose that there exist $i\in \mathcal{I}'$ such that $x^*_i - \tau \nabla_i f(x^*) = u_i$. If $x^*_i = u_i$ then $\tau \nabla_i f(x^*) = x^*_i -  u_i = 0$ which contradict with $\nabla_i f(x^*) < 0$. If $x^*_i \in (-l_i, u_i)$ then $\tau \nabla_i f(x^*) = x^*_i -  u_i \neq 0$ which contradicts with $\nabla_i f(x^*) = 0$. Hence, $x^*_i - \tau \nabla_i f(x^*) \neq u_i$ and similarly $x^*_i - \tau \nabla_i f(x^*) \neq -l_i$.
\end{proof}
We are ready to conclude from above that the index set $\mathcal{I}_{k}$ can be identified within finite steps and the sequence converges to a $\tau$-stationary point or a local minimizer globally that are presented by the following theorem.
\begin{theorem}\label{the-2-1}
{\rm (Convergence and identification of $\mathcal{I}_{k}$ )} Let $f$ be strongly smooth with respect to the constant $L > 0$, and let $\overline{\tau}$ be defined as above. The sequence $\{x^k\}$ is generated by the algorithm, with $\tau \in (0, \overline{\tau})$ and $\delta \in (0, \min \{1, 2L\})$. The strict complementarity and Abadie's CQ holds at the limit point $x^*$.  Then
\begin{description}
\item{(i)} There exist a subsequence $\{x^{k_t}\}$ of $\{x^k\}$ such that $\Theta_{k_t} \equiv \Theta_{\infty}$ and $\Gamma_{k_t} \equiv \Gamma_{\infty}$ for sufficiently large $k_t$.
\item{(ii)} Any limit point $x^{*}$ satisfies
\begin{align}\label{the-2-1-1}
\nabla_{\Theta_{\infty}}f(x^*) = 0, \ d^*_{\Gamma_{\infty}} = 0,\ x_{\overline{\mathcal{I}}_{\infty}}^* = 0,\ {\rm{supp}}(x^*) \subset \mathcal{I}_{\infty}.
\end{align}
Where $\mathcal{I}_{\infty}$ defined in Lemma \ref{lem-constantindex}. For $\Gamma^u_{\infty} := \Gamma^u_{k_t}$ and $\Gamma^l_{\infty} := \Gamma^l_{k_t}$, there is
\begin{align}\label{eq-Gammainfty}
x^*_{\Gamma^{u}_{\infty}} = u_{\Gamma^{u}_{\infty}},\ x^*_{\Gamma^{l}_{\infty}} = -l_{\Gamma^{u}_{\infty}}.
\end{align}
And $x^{*} \neq 0$, $x^* \in \Omega$. Furthermore, if $\tau_{*}$ satisfies
\begin{align}\label{the-2-1-2}
0 < \tau_{*} < \min \left\{ \overline{\tau},\ \min\limits_{i \in {\rm{supp}}(x^*)} |x_{i}^{*}|^2/(2\lambda),\ 2\lambda/(b^2) \right\},
\end{align}
where
\begin{align}\label{eq-b}
b := \max \left\{\max\limits_{i\in \overline{{\rm{supp}}(x^*)}} |g_i^*|, 1 \right\},
\end{align}
then $x^{*}$ is a $\tau_{*}$-stationary point .
\item{(iii)} If $x^{*}$ is an isolated point, then $x^k \to x^{*}$ .
\end{description}
\end{theorem}
\begin{proof}(i) From Lemma \ref{lem-constantindex} we know that for sufficiently large $k$, $\mathcal{I}_k \equiv \mathcal{I}_{k-1} \equiv \mathcal{I}_{\infty}$. By \eqref{lem-2-3-1}, there is $\lim\limits_{k\to\infty} \|g_{\Theta_k}\| = 0$ and $\lim\limits_{k\to\infty} \|d_{\Gamma_k}\| = 0$. Since $$x^k_{\Theta_k} - \left[ \Pi_{\Omega}(x^{k} - \tau \nabla f(x^{k})) \right]_{\Theta_k} = \tau \nabla_{\Theta_k} f(x^{k}).$$
It holds that $ \lim\limits_{k\to\infty} \| x^{k}_{\mathcal{I}_k} - \left[ \Pi_{\Omega}(x^{k} - \tau \nabla f(x^{k})) \right]_{\mathcal{I}_k} \| =0 $. Let $\{x^{k_t}\}$ be the convergent subsequence of $\{x^k\}$ that converges to $x^*$. Therefore we have
$$ \| x^{*}_{\mathcal{I}_{\infty}} - \left[ \Pi_{\Omega}(x^{*} - \tau \nabla f(x^{*})) \right]_{\mathcal{I}_{\infty}}  \| =  \lim\limits_{k_t \to\infty} \| x^{k_t}_{\mathcal{I}_{\infty}} - \left[ \Pi_{\Omega}(x^{k_t} - \tau \nabla f(x^{k_t})) \right]_{\mathcal{I}_{\infty}}  \|  =0, $$
and $ x^{*}_{\overline{ \mathcal{I} }_{\infty}} = x^{k_t}_{\overline{ \mathcal{I} }_{\infty}} =0 $ for sufficiently large $k_t$. With $\mathcal{I}'$ replaced by $\mathcal{I} _{\infty}$ in problem \eqref{BL02}, it follows immediately that $x^*$ is a stationary point of problem \eqref{BL02}. Denote
\begin{align*}
&\Theta_* := \{ i \in \mathbb{N}^n\ |\ a_i > |x_i^* - \tau \nabla_i f(x^*)| \ge \sqrt{2\tau \lambda} \} ,\\
&\Gamma_*^{u} := \{ i \in \mathbb{N}^n\ |\  x_i^* - \tau \nabla_i f(x^*) \ge u_i  \} ,\\
&\Gamma_*^{l} := \{ i \in \mathbb{N}^n\ |\ x_i^* - \tau \nabla_i f(x^*) \le -l_i  \} ,\\
&\Gamma_* := \Gamma_*^{u}\cup \Gamma_*^{l},\\
&\mathcal{I}_* := [\Theta_*, \Gamma_*] =  \{ i \in \mathbb{N}^n\ |\  |x_i^* - \tau \nabla_i f(x^*)| \ge \sqrt{2\tau \lambda} \} .
\end{align*}
For simplicity, we will only consider the case of $\Gamma_*^{u}$. From Lemma \ref{lem-kkt}, we obtain that for any $i\in \mathcal{I}_{\infty} $ there is $|x^*_i -\tau \nabla_i f(x^*) | \neq u_i$ or $l_i$. Hence, for any $i\in \Gamma^{u}_*$ there is $x_i^* - \tau \nabla_i f(x^*) > u_i$, by the continuity of $| \cdot |$ and $f$ together with $x^{k_t} \to x^*$, it holds that $x^{k_t}_i - \tau \nabla_i f(x^{k_t}) \ge u_i $ for sufficiently large $k_t$. By the definition \eqref{eq-2-6}, we obtain $i \in \Gamma^{u}_{k_t}$ which is defined in \eqref{eq-index}. Conversely, if $i \in \Gamma^{u}_{k_t}$ for sufficiently large $k_t$. By the continuity, it holds that $ x_i^* - \tau \nabla_i f(x^*) \ge u_i $ then $i\in \Gamma^{u}_*$. $\Gamma_*^{l}$ is similar to $\Gamma_*^{u}$.  Therefore
\begin{align}\label{eq-Gammaequiv}
\Gamma_{k_t} \equiv \Gamma_{\infty} := \Gamma_*
\end{align}
 holds for sufficiently large $k$. Because $\mathcal{I}_{k_t} \equiv \mathcal{I}_{\infty}$, then $\Theta_{k_t} \equiv \Theta_{\infty} := \mathcal{I}_{\infty} \setminus \Gamma_*$ holds for sufficiently large $k_t$.

(ii) By the continuity and \eqref{lem-2-3-1}, for sufficiently large $k_t$, it holds that
\begin{align}\label{eq-3.30}
\begin{cases}
\nabla_{\Theta_{\infty}}f(x^*) = \nabla_{\Theta_{k_t}}f(x^*) = \lim\limits_{k_t \to \infty}\nabla_{\Theta_{k_t}}f(x^{k_t}) = 0 , \\
\| d_{\Gamma_{\infty}}^* \|:= \| \left[\Pi_{\Omega}(x^* - \tau g^*)\right]_{\Gamma_{\infty}} - x^*_{\Gamma_{\infty}} \| = \|d_{\Gamma_{k_t}}^{*}\| = \lim\limits_{k_t \to \infty}\|d_{\Gamma_{k_t}}^{k_t} \| = 0 .
\end{cases}
\end{align}
It is obviously that $x_{\overline{\mathcal{I}}_{\infty}}^* = 0$, hence \eqref{the-2-1-1} is true. Because $\left[ \Pi_{\Omega}(x^* - \tau g^*)\right]_{\Gamma^{u}_{\infty}}  = u_{\Gamma^{u}_{\infty}}$, we obtain that $x^*_{\Gamma^{u}_{\infty}}  = u_{\Gamma^{u}_{\infty}}$. Similarly, there is $x^*_{\Gamma^{l}_{\infty}}  = -l_{\Gamma^{l}_{\infty}}$. Hence, \eqref{eq-Gammainfty} is true.

Next, we claim that $x^* \neq 0$. Suppose $ x^* = 0 $. Algorithm \ref{algorithm-NL0R} runs infinite steps only when $ \nabla f(0) \neq 0 $. For $i\in \{i\in \mathbb{N}_{n}\mid \nabla_i f(0) \neq 0\}$. By $ x^{k_t} \to x^* = 0 $ and the continuous of $\nabla f(\cdot)$, there is a sufficiently small $\epsilon > 0$, such that
\begin{align}\label{eq-xf0leepsilon}
\min (a/2, \sqrt{2\tau \underline{\lambda}}) > \epsilon  \ \mbox{and}\ |x^{k_t}_i| + \tau |\nabla_{i} f(0) - \nabla_i f(x^{k_t})| \le \epsilon,
\end{align}
for sufficiently large $k$, where $a$ defined in \eqref{eq-tau}.
By \eqref{eq-2.13} and $ \lambda \in (0, \underline{\lambda}) $, we have $\tau |\nabla_i f(0)| \ge \sqrt{2\tau \underline{\lambda}}$. Together with \eqref{eq-xf0leepsilon}, it holds that
\begin{align}\label{eq-3.31}
|x^{k_t}_i - \tau g^{k_t}_i| & \ge \tau |\nabla_i f(0)| - |x^{k_t}_i| - \tau |\nabla_i f(0) - \nabla_i f(x^{k_t})| \nonumber \\
&\ge \sqrt{2\tau \underline{\lambda}} - \epsilon \ge \sqrt{2\tau \lambda} .
\end{align}
By $\tau \le \bar{\tau}$ and \eqref{2-1-1}, we have
\begin{align*}
\tau |\nabla_i f(0)| \le \tau \max\limits_{i} |\nabla_i f(0)| \le a/2.
\end{align*}
Together with \eqref{eq-xf0leepsilon}, it holds that
\begin{align}\label{eq-3.31-2}
|x^{k_t}_i - \tau g^{k_t}_i| & \le \tau |\nabla_i f(0)| + |x^{k_t}_i| + \tau |\nabla_i f(0) - \nabla_i f(x^{k_t})| \nonumber \\
&\le a/2 + \epsilon < a .
\end{align}
Combine with \eqref{eq-xf0leepsilon}, \eqref{eq-3.31} and \eqref{eq-3.31-2}, we know that
\begin{align*}
a > a/2 + \epsilon \ge |x^{k_t}_i - \tau g^{k_t}_i| \ge \sqrt{2\tau \lambda} .
\end{align*}
By the definition of $a$ and $\Theta_{k_t}$ we obtain that $i\in \Theta_{k_t} $. By $\Theta_{k_t} \equiv \Theta_{\infty}$ and \eqref{eq-3.30}, for $i \in \Theta_{k_t} \equiv \Theta_{\infty}$ there is $ g_{k_t}^i \to 0 $. Combine with $ x_{k_t}^i \to 0 $, we obtain that $|x^{k_t}_i - \tau g^{k_t}_i|\to 0$ which contradicting \eqref{eq-3.31}. Thus, $ x^* \neq 0 $.

Now we need to prove that $x^{*}$ is a $\tau_{*}$-stationary point, and for this, we need to first prove that $ \text{supp} (x^*) = \mathcal{I}_*$, where
\begin{align*}
\mathcal{I}_* = \{ i \in \mathbb{N}^n \, \mid \, |x^*_i - \tau_* \nabla_i f(x^*)|  \ge \sqrt{2\tau_* \lambda} \}.
\end{align*}
It is obvious that, $\text{supp}(x^*) \subseteq \text{supp} (x^{k_t + 1}) = \mathcal{I}_{\infty} $. First, given $i \in \text{supp} (x^*)$ we prepare to prove $i\in \mathcal{I}_*$. There are two possibilities: either $i\in \Theta_{\infty}$ or $i \in \Gamma_{\infty}$. Since $\Gamma_{\infty} = \Gamma_{*}$ holds by its definition in \eqref{eq-Gammaequiv}. If $i\in  \Gamma_{\infty}$, there is $i\in \Gamma_* \subseteq \mathcal{I}_*$. Hence, we only need to consider $i\in \Theta_{\infty}$. Then, we have $\nabla_{i} f(x^*) = 0$ by $\nabla_{\Theta_{\infty}} f(x^*) = 0$. We proceed by contradiction, assume that $i\notin \mathcal{I}_*$ that is $|x^*_i - \tau_* \nabla_i f(x^*)| < \sqrt{2\tau_* \lambda}$. From \eqref{the-2-1-2}, there is $2\tau_*\lambda \le \min\limits_{i \in \text{supp}(x^*)} |x_{i}^{*}|^2$ , moreover it holds that
$$|x^*_i| = |x^*_i - \tau_* \nabla_i f(x^*)| < \sqrt{2\tau_* \lambda} \le |x^*_i|,$$
that is $|x^*_i| < |x^*_i|$, which results in a contradiction, the hypothesis is rejected. The contradiction argument is now complete. That is $i\in \mathcal{I}_*$. Hence, $\text{supp} (x^*)\subseteq \mathcal{I}_{*}$.

Second, proving $\mathcal{I}_* \subseteq \text{supp} (x^*)$ is equivalent to proving $\overline{ \text{supp} (x^*)} \subseteq \overline{ \mathcal{I}}_* $. Consider $i\in \overline{ \text{supp} (x^*)}$, there is $x^{*}_i = 0$. By \eqref{eq-b}, it holds that
\begin{align*}
|\tau_* \nabla_i f(x^*)|^2 \le \tau_* \max\limits_{i\in \overline{\text{supp}(x^*)}} |\nabla_i f(x^*)| < |\tau_* b|^2 .
\end{align*}
Combine with $|\tau_* b|^2 = \tau_*(\tau_* b^2) \le \tau_*(2\lambda)$ from \eqref{the-2-1-2},
we obtain that
$$ |x^*_i - \tau_* \nabla_i f(x^*)|^2 = |\tau_* \nabla_i f(x^*)|^2 < |\tau_* b|^2 \le 2\tau_* \lambda ,$$
which contradict with $ |x^*_i - \tau_* \nabla_i f(x^*)| \ge \sqrt{2\tau_* \lambda}$ by the definition of $\mathcal{I}_*$.

From above, we obtain that $ \text{supp} (x^*) = \mathcal{I}_* = [\Theta_*, \Gamma_* ] $. This, together with $ \nabla_{\Theta_*} f(x^*) = 0 $, $d_{\Gamma_*}^* = 0$ and $ x^*_{\overline{\mathcal{I}}_*} = 0 $ from \eqref{the-2-1-1}, suffices to obtain that $F_{\tau_*}(x^*;\Theta_*;\Gamma_*;\overline{\mathcal{I}}_*) = 0$. Finally, Proposition \ref{pro-3} allows us to claim that $ x^* $ is a $ \tau_* $-stationary point.

(iii) By \cite[Lemma 4.10]{more1983computing} and $\|x^{k+1} - x^k\| \to 0$ from Lemma \ref{lem-2-3}, then, if $x^*$ is isolated  , the whole sequence converges to $x^*$.
\end{proof}

\begin{theorem}\label{the-2-2}
{\rm (Global and quadratic convergence)} Let $\left\{x^{k}\right\}$ be the sequence generated by Algorithm \ref{algorithm-NL0R} with $\tau \in(0, \overline{\tau})$ and $\delta \in\left(0, \min \left\{1, \ell_{*}\right\}\right)$ and $x^{*}$ be one of its accumulation points. Assume that strict complementarity and Abadie's CQ holds at  $x^*$. Suppose $f$ is strongly smooth with constant $L>0$ and locally strongly convex with $\ell_{*}>0$ around $x^{*}$. Then, the following results hold.
\begin{description}
\item{(i)} The whole sequence converges to $x^{*}$, a strictly local minimizer of \eqref{PBL0}.

\item{(ii)} The Newton direction is always accepted for sufficiently large $k$.

\item{(iii)} If the Hessian of $f$ is locally Lipschitz continuous around $x^*$ with constant $M_* \ge 0$, then for sufficiently large $k$, it holds that
    \begin{align}
    \|x^{k+1} - x^* \| \le M_{*}/(2\tilde{\ell}_{*}) \|x^k - x^* \|^2.
    \end{align}
where, $\tilde{\ell}_{*} = \min (1,\ell_*)$.
\end{description}
\end{theorem}
\begin{proof}
(i) Denote $\mathcal{H}_{*}:=\text{supp} (x^{*} )$, $\mathcal{F}_* := \Theta_{\infty} \cap \mathcal{H}_{*}$, $\mathcal{G}_* := \Gamma_{\infty} \cap \mathcal{H}_{*}$, it holds that $ \mathcal{F}_{*}\cup \mathcal{G}_{*} = \mathcal{H}_{*} \cap \mathcal{I}_{\infty} =\mathcal{H}_{*} $ by \eqref{the-2-1-1}. Theorem \ref{the-2-1} shows that \begin{align}\label{eq-the7-1}
\nabla_{\mathcal{F}_{*}} f(x^{*} )=0,\ x^*_{\mathcal{G}_{*}} = \left[\Pi_{\Omega}(x^* - \tau \nabla f(x^*))\right]_{\mathcal{G}_{*}},\ \mbox{and}\ x^{*} \neq 0.
\end{align}
Consider a local region $N(x^{*} ):=\left\{x \in \Omega: \|x-x^{*} \|<\epsilon_{*}\right\}$, where
\begin{align}\label{epsilon7}
\epsilon_{*}:=\min \left\{\lambda /\left(2 \|\nabla_{\overline{\mathcal{H}}_{*}} f(x^{*} ) \|\right), \min _{i \in \mathcal{H}_{*}} |x_{i}^{*} |\right\}.
\end{align}
For any $x (\neq x^{*} ) \in N (x^{*} )$, we have $\mathcal{H}_{*} \subseteq \text{supp}(x)$. In fact if there is a $j$ such that $j \in \mathcal{H}_{*}$ but $j \notin \text{supp}(x)$, then we derive a contradiction:
$$
\epsilon_{*} \leq \min_{i \in \mathcal{H}_{*}} |x_{i}^{*} | \leq |x_{j}^{*} |= |x_{j}^{*}-x_{j} | \leq \|x-x^{*} \|_{2} < \epsilon_{*} .
$$
By $f$ is locally strongly convex with $\ell_{*}>0$ around $x^{*}$ which is defined in \eqref{eq-2-5}, for any $x\left(\neq x^{*}\right) \in N\left(x^{*}\right)$ together with \eqref{eq-the7-1}, it is true that
\begin{align*}
&\quad\  f(x)+\lambda p(x)-f(x^{*} )-\lambda p(x^*) \\
&\ge \left\langle \nabla f(x^*), x - x^* \right\rangle +  (\ell_{*} / 2 ) \| x- x^{*} \|^{2}+\lambda p(x)-\lambda p(x^*) \nonumber\\
&=\left\langle \nabla_{\mathcal{F}_*}f(x^*), (x - x^*)_{\mathcal{F}_*} \right\rangle + \left\langle \nabla_{\mathcal{G}_*}f(x^*), (x - x^*)_{\mathcal{G}_*} \right\rangle + \left\langle \nabla_{\overline{\mathcal{H}}_*} f(x^*),(x - x^*)_{\overline{\mathcal{H}}_*} \right\rangle \nonumber \\
&\quad  +  (\ell_{*} / 2 ) \| x- x^{*} \|^{2}+\lambda p(x)-\lambda p(x^*) \nonumber\\
&= \left\langle \nabla_{\mathcal{G}_*}f(x^*), (x - x^*)_{\mathcal{G}_*} \right\rangle + \left\langle \nabla_{\overline{\mathcal{H}}_*} f(x^*),(x - x^*)_{\overline{\mathcal{H}}_*} \right\rangle \nonumber \\
&\quad +  (\ell_{*} / 2 ) \| x- x^{*} \|^{2}+\lambda p(x)-\lambda p(x^*)
\end{align*}
By \eqref{eq-the7-1} and \eqref{eq-2-6-2}, there is $\left\langle \nabla_{\mathcal{G}_*} f(x^*),(x - x^*)_{\mathcal{G}_*} \right\rangle \ge 0$. Hence, it holds that
\begin{align*}
&\quad \  f(x)+\lambda p(x)-f(x^{*} )-\lambda p(x^*) \\
&\ge \left\langle \nabla_{\overline{\mathcal{H}}_*} f(x^*),(x - x^*)_{\overline{\mathcal{H}}_*} \right\rangle +  (\ell_{*} / 2) \| x- x^{*} \|^{2}+\lambda p(x)-\lambda p(x^*) =: \phi.
\end{align*}
Clearly, if $\mathcal{H}_{*}=\text{supp}(x)$, then $x_{\overline{\mathcal{H}}_{*}}=0, \|x\|_0 =\|x^{*}\|_{0}$ and $p(x) = p(x^*)$  hence $\phi= (\ell_{*} / 2 ) \|x-x^{*} \|^{2}>0$. If $\mathcal{H}_{*} \subset \text{supp}(x)$, then $p(x) \geq p(x^*)+1$ together with $\|\nabla_{\overline{\mathcal{H}}_{*}} f (x^{*} ) \| \|x-x^{*} \| \le \|\nabla_{\overline{\mathcal{H}}_{*}} f (x^{*} ) \|\epsilon_* $ from $x^* \in N(x^*)$, we obtain that
\begin{align*}
\phi & \geq-  \|\nabla_{\overline{\mathcal{H}}_{*}} f (x^{*} ) \| \|(x - x^*)_{\overline{\mathcal{H}}_{*}} \|+\left(\ell_{*} / 2\right) \|x-x^{*} \|^{2}+ \lambda \\
&\ge  - \|\nabla_{\overline{\mathcal{H}}_{*}} f (x^{*} ) \| \|x-x^{*} \|+\left(\ell_{*} / 2\right) \|x-x^{*} \|^{2}+\lambda \\
&\ge -\|\nabla_{\overline{\mathcal{H}}_{*}} f (x^{*} ) \|\epsilon_* +\left(\ell_{*} / 2\right) \|x- x^{*} \|^{2}+\lambda>0.
\end{align*}
By \eqref{epsilon7}, it holds that
$\|\nabla_{\overline{\mathcal{H}}_{*}} f (x^{*} ) \|\epsilon_* \le \lambda /2.$ Hence, we obtain that
\begin{align*}
\phi \ge -\lambda /2 +\left(\ell_{*} / 2\right) \|x- x^{*} \|^{2}+\lambda>0.
\end{align*}

Both cases show that $x^{*}$ is a strictly local minimizer of \eqref{PBL0} and is unique in $N (x^{*} )$, namely, $x^{*}$ is isolated local minimizer in $N (x^{*} )$. Therefore, the whole sequence tends to $x^{*}$ by Theorem \ref{the-2-1}.

(ii) Now, we ready to prove that condition \eqref{newtonfun} always holds for sufficiently large $k$, that is Algorithm \ref{algorithm-NL0R} will ultimately utilize the Newton step. From now on, we denote $x^{k+1} = \tilde{x}^k$ which is defined in \eqref{eq-tildexk}. First, we verify that $H^{k}$ is nonsingular when $k$ is sufficiently large and for $d^k$ defined by \eqref{direction}, it satisfies
$$
\left\langle g_{\mathcal{I}_{k}}^{k}, d_{\mathcal{I}_{k}}^{k} \right\rangle \leq -\delta \|d^k\|^2 + \|x^k_{\overline{\mathcal{I}}_{k}}\|^2/(4\tau)  .
$$
Since $f$ is strongly smooth with $L$ and locally strongly convex with $\ell_{*}$ around $x^{*}$ and by the properties of the block upper triangular structure of matrix
$H^k$, it holds that
\begin{align}\label{the-2-2-2}
\begin{cases}
&\ell_{*}  \leq  \lambda_{i}\left( \nabla^2_{\Theta_k, \Theta_k} f(x^k) \right) \leq L ,\\
&\tilde{\ell_{*}} := \min (1, \ell_{*}) \leq  \lambda_{i}\left(H^{k}\right) \leq \tilde{L} := \max (1, L),
\end{cases}
\end{align}
where $\lambda_{i}(A)$ is the $i$ th largest eigenvalue of $A$. By \eqref{eq-3-1}, \eqref{the-2-2-2} and \eqref{2-1}, we obtain that
\begin{align*}
\left\langle d^k_{\mathcal{I}_k}, g^k_{\mathcal{I}_k} \right\rangle & \le  - \left\langle d_{\mathcal{I}_k \cup J_k}^k , H^k_{\mathcal{I}_k\cup J_k} d_{\mathcal{I}_k \cup J_k}^k \right\rangle + (1 - \frac{1}{\tau}) \|d_{\Gamma_k}^k\|^2 + \|d_{J_k}^k\|^2 \\
&\le -\tilde{\ell_{*}} \|d_{\mathcal{I}_k \cup J_k}^k\|^{2} + (1 - \frac{1}{\tau}) \|d_{\Gamma_k}^k\|^2 + \|d_{J_k}^k\|^2  \\
&\le -\tilde{\ell_{*}} \| d_{\mathcal{I}_k \cup J_k}^k \|^{2}  + \|d_{J_k}^k\|^2  \quad \text{by} \  (0 < \tau < \overline{\tau} <1)\\
&= - \tilde{\ell_{*}} \|d^{k} \|^{2} + \|x_{\overline{\mathcal{I}}_k}^k\|^2\\
&\le - \delta \|d^{k} \|^{2}+ \|x_{\overline{\mathcal{I}}_{k}}^{k} \|^{2} /(2 \tau) ,
\end{align*}
where the last inequality is due to $\delta \leq \tilde{\ell_{*}}$ and $\tau <\overline{\tau} \le 1/n \le 1/2 $. From Theorem \ref{the-2-1} we know that $\text{supp} (x^{k+1}) = \mathcal{I}_k$ and $\mathcal{I}_k \equiv \mathcal{I}_{\infty}$ holds for sufficiently large $k$, then it holds that
$$|\mathcal{I}_k| \le \|x^{k}\|_0, \ \mathcal{I}_k = \mathcal{I}_{k-1} .$$
Now we need to prove $x^k + d^k \in \Omega$ for sufficiently large $k$, where $d^k$ from \eqref{direction}. By \eqref{direction}, there is
\begin{align*}
-l_{\Gamma_k} \le x^{k+1}_{\Gamma_k} \le u_{\Gamma_k},\ -l_{\overline{\mathcal{I}}_k} \le x^{k+1}_{\overline{\mathcal{I}}_k} \le u_{\overline{\mathcal{I}}_k}.
\end{align*}
Hence, we only need to prove $-l_{\Theta_k} \le x^k_{\Theta_k} + d^k_{\Theta_k} \le u_{\Theta_k}$. By \eqref{lem-2-3-1}, we know that $g^k_{\Theta_k} \to 0$, $d_{\Gamma_k}  \to 0$ and $x_{\overline{\mathcal{I}}_k}^k  \to 0$. Hence, there is $d_{\Theta_k}^k \to 0$ and also $d^k\to 0$ by \eqref{direction}. From claim (i), there is $x^k \to x^*$ together with Theorem \ref{the-2-1}, we obtain that $\Theta_{k} \equiv \Theta_{\infty}$ holds for sufficiently large $k$ and $\Theta_{*} \subseteq \Theta_{\infty}$. It is obvious that for any $i\in \Theta_{*}$ there is $-l_i < x^*_i = x^*_i - \tau g^*_i <u_i$. For simplicity, we only consider the case of $x^*_i < u_i$. Denote
\begin{align}\label{eq-xiepsilon}
\xi = u_i - x^*_i >0\ \mbox{and}\ \epsilon^k_i = x^k_i - x^*_i.
\end{align}
By $x^k\to x^*$ there is $\epsilon^k_i \to 0$, together with $\|d^k\|\to 0$ we obtain that $ \epsilon^k_i + d^k_i \le \xi $ for sufficiently large $k$. For $i\in  \Theta_{*} \subseteq \Theta_{k}$ and sufficiently large $k$, by \eqref{eq-xiepsilon}, it holds that
$$ x^k_i + d^k_i = x^*_i + \epsilon^k_i +d^k_i {\le} x^*_i + \xi \le u_i .$$
Similarly, we can prove that $-l_i \le x^k_i + d^k_i$ holds for sufficiently large $k$. Hence, we obtain that $-l_{\Theta_k} \le x^k_{\Theta_k} + d^k_{\Theta_k} \le u_{\Theta_k}$, thus there is $x^k + d^k \in \Omega$. This proves that $d^{k}$ from \eqref{direction} is always admitted for sufficiently large $k$.

(iii) By Theorem \ref{the-2-1}, for sufficiently large $k$, we have \eqref{eq-Gammainfty}, namely
\begin{align}\label{the-2-2-3}
g_{\Theta_{k}}^* = 0, \quad x_{\Gamma^u_k}^*  = u_{\Gamma^u_k},\quad x_{\Gamma^l_k}^*  = -l_{\Gamma^l_k},\quad x_{\overline{\mathcal{I}}_{k}}^* = 0.
\end{align}
For any $0 \leq t \leq 1$, by letting $x(t):=x^{*}+t\left(x^{k}-x^{*}\right)$. the Hessian of $f$ being locally Lipschitz continuous at $x^{*}$ derives
\begin{align}\label{the-2-2-4}
\left\|\nabla_{\mathcal{I}_{k}:}^{2} f\left(x^{k}\right)-\nabla_{\mathcal{I}_{k}:}^{2} f(x(t))\right\|_{2} \leq M_{*}\left\|x^{k}-x(t)\right\|=(1-t) M_{*}\left\|x^{k}-x^{*}\right\| .
\end{align}
Moreover, by Taylor expansion, we obtain
\begin{align}\label{the-2-2-5}
\nabla f(x^k) - \nabla f(x^* ) =
\int_{0}^{1} \nabla^2 f(x(t) )(x^k - x^*)dt .
\end{align}
By \eqref{support} and \eqref{the-2-2-3}, there is $\| x^{k+1}_{\overline{\mathcal{I}}_k} - x^*_{\overline{\mathcal{I}}_k} \|=0$. Together with \eqref{2.4}, it holds that
\begin{align}\label{eq-xk1alpha}
\| x^{k+1} - x^* \|^2 &= \| x^{k+1}_{\mathcal{I}_k} - x^*_{\mathcal{I}_k} \|^2 + \| x^{k+1}_{\overline{\mathcal{I}}_k} - x^*_{\overline{\mathcal{I}}_k} \|^2 \nonumber \\
&= \| x^{k}_{\Theta_k} - x^*_{\Theta_k} + \alpha_k d_{\Theta_k}^k \|^2 +  \| x^{k}_{\Gamma_k} - x^*_{\Gamma_k} + d_{\Gamma_k}^k \|^2 .
\end{align}
By the definition of $\Gamma^u_k$ and $\Gamma^l_k$ in \eqref{eq-index} together with \eqref{eq-bound} and \eqref{the-2-2-3}. There is
\begin{align}\label{eq-gammaulcon}
\| u^{k}_{\Gamma^u_k} - x^*_{\Gamma^u_k}\|= 0\ \mbox{and}\ \| l^{k}_{\Gamma^l_k} - x^*_{\Gamma^l_k} \|=0.
\end{align}
Combine with \eqref{eq-xk1alpha} and \eqref{eq-gammaulcon}, we obtain that
\begin{align}\label{eq-gammakstar}
\| x^{k}_{\Gamma_k} - x^*_{\Gamma_k} + d_{\Gamma_k}^k \|^2 &\le \| x^{k}_{\Gamma^u_k} - x^*_{\Gamma^u_k} + d_{\Gamma^u_k}^k \|^2 +  \| x^{k}_{\Gamma^l_k} - x^*_{\Gamma^l_k} + d_{\Gamma^l_k}^k \|^2 \nonumber\\
&= \| u^{k}_{\Gamma^u_k} - x^*_{\Gamma^u_k}  \|^2 +  \| l^{k}_{\Gamma^l_k} - x^*_{\Gamma^l_k} \|^2 {=} 0.
\end{align}
Combine with \eqref{eq-xk1alpha} and \eqref{eq-gammakstar}, we have the following chain of inequalities
\begin{align}\label{eq-3.37}
\| x^{k+1} - x^* \|^2 &\ \, = \ \, \| x^{k}_{\Theta_k} - x^*_{\Theta_k} + \alpha_k d_{\Theta_k}^k \|^2 \nonumber \\
&\ \, = \ \, \|(1 - \alpha_k)(x^{k}_{\Theta_k} - x^*_{\Theta_k}) + \alpha_k( x^{k}_{\Theta_k} - x^*_{\Theta_k} +  d_{\Theta_k}^k ) \|^2 \nonumber \\
&\ \, \le \ \, (1 - \alpha_k)\| x^{k}_{\Theta_k} - x^*_{\Theta_k} \|^2 + \alpha_k \| x^{k}_{\Theta_k} - x^*_{\Theta_k} +  d_{\Theta_k}^k \|^2 .
\end{align}
where the last inequality is due to $\| \cdot \|^2$ being a convex function.
Combine with \eqref{eq-3.37} and \eqref{lem-2-2-4}, we obtain that
\begin{align}\label{eq-3.38}
\| x^{k+1} - x^* \|^2 \le (1 - \beta \overline{\alpha})\| x^{k} - x^* \|^2 + \overline{\alpha} \| x^{k}_{\Theta_k} - x^*_{\Theta_k} +  d_{\Theta_k}^k \|^2.
\end{align}
From claim (ii), $d^k$ will be updated eventually by \eqref{direction}. Hence, by \eqref{direction} and \eqref{the-2-2-2}, for sufficiently large $k$, it holds that
\begin{align}\label{eq-lxk-1}
&\quad \ {\ell_*}  \|x_{ \Theta_k}^{k}-x_{ \Theta_k}^{*}+d_{ \Theta_k}^{k} \| \nonumber\\
&=\ell_*  \| \nabla_{\Theta_k,\Theta_k}^2 f(x^k)^{-1}[-\nabla_{\Theta_k,\Gamma_k}^{2}f(x^k)d_{\Gamma_k}^k +\nabla_{\Theta_k,\overline{\mathcal{I}}_k}^{2}f(x^k)x_{\overline{\mathcal{I}}_k}^k - g_{\Theta_k}^k] + x_{\Theta_k}^k - x_{\Theta_k}^*  \| \nonumber\\
&\le \|-\nabla_{\Theta_k,\Gamma_k}^{2}f(x^k)d_{\Gamma_k}^k +\nabla_{\Theta_k,\overline{\mathcal{I}}_k}^{2}f(x^k)x_{\overline{\mathcal{I}}_k}^k - g_{\Theta_k}^k + \nabla_{\Theta_k}^2f(x^k)(x_{\Theta_k}^k - x_{\Theta_k}^*) \| .
\end{align}
Denote $z^k = x^k - \tau \nabla f(x^k)$, by $g_{\Theta_{k}}^* = 0$ and $x_{\overline{\mathcal{I}}_{k}}^* = 0$ from  \eqref{the-2-2-3}, it holds that
\begin{align}\label{eq-lxk-2}
&\quad \, -\nabla_{\Theta_k,\Gamma_k}^{2}f(x^k)d_{\Gamma_k}^k +\nabla_{\Theta_k,\overline{\mathcal{I}}_k}^{2}f(x^k)x_{\overline{\mathcal{I}}_k}^k - g_{\Theta_k}^k + \nabla_{\Theta_k, \Theta_k}^2f(x^k)(x_{\Theta_k}^k - x_{\Theta_k}^*)  \nonumber\\
&= -\nabla_{\Theta_k,\Gamma_k}^{2}f(x^k)\left(\left[\Pi_{\Omega}(z^k)\right]_{\Gamma_k} - x_{\Gamma_k}^k \right) +\nabla_{\Theta_k,\overline{\mathcal{I}}_k}^{2}f(x^k)x_{\overline{\mathcal{I}}_k}^k - g_{\Theta_k}^k \nonumber\\
&\quad + \nabla_{\Theta_k, \Theta_k}^2f(x^k)(x_{\Theta_k}^k - x_{\Theta_k}^*)  \nonumber\\
&= \nabla_{\Theta_k :}^2f(x^k)x^k - \nabla_{\Theta_k, \Theta_k}^2f(x^k) x_{\Theta_k}^* - g_{\Theta_k}^k -\nabla_{\Theta_k,\Gamma_k}^{2}f(x^k)\left[\Pi_{\Omega}(z^k)\right]_{\Gamma_k}\nonumber\\
&= \nabla_{\Theta_k :}^2f(x^k)x^k - \nabla_{\Theta_k :}^2f(x^k) x^* - g_{\Theta_k}^k -\nabla_{\Theta_k,\Gamma_k}^{2}f(x^k)\left[\Pi_{\Omega}(z^k)\right]_{\Gamma_k} \nonumber\\
&\quad +\nabla_{\Theta_k, \Gamma_k}^2f(x^k) x^*_{\Gamma_k} + \nabla_{\Theta_k, \overline{\mathcal{I}}_k}^2f(x^k) x^*_{\overline{\mathcal{I}}_k} \nonumber\\
&= \nabla_{\Theta_k :}^2f(x^k)x^k - \nabla_{\Theta_k :}^2f(x^k) x^* - g_{\Theta_k}^k + g_{\Theta_k}^* \ (\mbox{by}\ \eqref{the-2-2-3})\nonumber\\
&\quad -\nabla_{\Theta_k,\Gamma_k}^{2}f(x^k)\left[\Pi_{\Omega}(z^k)\right]_{\Gamma_k} +\nabla_{\Theta_k, \Gamma_k}^2f(x^k) x^*_{\Gamma_k}
.
\end{align}
Combine with \eqref{eq-lxk-1} and \eqref{eq-lxk-2}, we obtain that
\begin{align}\label{eq-xdtheta}
&\quad\ {\ell_*}  \|x_{ \Theta_k}^{k}-x_{ \Theta_k}^{*}+d_{ \Theta_k}^{k} \| \nonumber\\
&\le \| \nabla_{\Theta_k :}^{2}f(x^k)x^k - g_{\Theta_k}^k  - \nabla_{\Theta_k :}^{2}f(x^k)x^* + g_{\Theta_k}^* \| \nonumber\\
&\quad  +  \left\| \nabla_{\Theta_k, \Gamma_k}^2f(x^k) \right\| \left\|\left[\Pi_{\Omega}(z^k)\right]_{\Gamma_k} -  x^*_{\Gamma_k}  \right\|.
\end{align}
By \eqref{eq-gammaulcon}, we obtain that $\left\| \left[\Pi_{\Omega}(z^k) \right]_{\Gamma_k} -  x^*_{\Gamma_k} \right\| = 0$. Therefore, by \eqref{eq-xdtheta} and \eqref{the-2-2-5}, it holds that
\begin{align*}
&\quad\  {\ell_*} \left\|x_{ \Theta_k}^{k}-x_{ \Theta_k}^{*}+d_{ \Theta_k}^{k}\right\| \\
&\le \left\| \nabla_{\Theta_k :}^{2}f(x^k)x^k - g_{\Theta_k}^k  - \nabla_{\Theta_k :}^{2}f(x^k)x^* + g_{\Theta_k}^* \right\| \\
&\le \left\|  \int_{0}^{1}[ \nabla_{\Theta_k :}^2 f(x^k) - \nabla_{\Theta_k :}^2 f(x(t))] (x^k - x^*) dt \right\| \\
&\le \int_{0}^{1}\|  \nabla_{\Theta_k,:}^2 f(x^k) - \nabla_{\Theta_k,:}^2 f(x(t))\| \|(x^k - x^*) \|dt \\
&\le M_{*}\| x^k - x^*\|^2 \int_{0}^{1}(1-t)dt \\
&\le 0.5 M_{*}\| x^k - x^*\|^2 .
\end{align*}
It follows from $d_{\overline{\mathcal{I}}_k}^k = - x_{\overline{\mathcal{I}}_k}^k$ , $\left\|\left[\Pi_{\Omega}(z^k)\right]_{\Gamma_k} -  x^*_{\Gamma_k} \right\| = 0$ and \eqref{the-2-2-3}, then there is $\|x^k + d^k - x^*\| = \|x_{ \Theta_k}^{k}-x_{ \Theta_k}^{*}+d_{ \Theta_k}^{k}\|$, leading to the following fact
\begin{align}\label{eq-3.41}
\frac{\|x^k + d^k - x^*\|}{\|x^k - x^*\|} = \frac{\|x_{ \Theta_k}^{k}-x_{ \Theta_k}^{*}+d_{ \Theta_k}^{k}\|}{\|x^k - x^*\|} \le \frac{M_* \|x^k - x^*\|^2}{2 {\ell_*} \|x^k - x^*\|} \le \frac{M_*}{2\ell_*}\|x^k - x^*\|.
\end{align}
Now, we have three facts: \eqref{eq-3.41}, $ x^k \to x^* $ from (1), and $ \left\langle \nabla f(x^k), d^k \right\rangle \le \rho \|d^k \|^2 $ from Lemma \ref{lem-2-1}, which together with \cite[Theorem 3.3]{facchinei1995minimization} allow us to claim that eventually the stepsize $ \alpha_k $ determined by the Armijo rule is 1, namely $ \alpha_k = 1 $. Then, the sequence converges quadratically, completing the proof.
\end{proof}

\section{Numerical results}
In this section, we report the numerical results of the novel subspace Newton's method (BNL0R) proposed in this paper. To demonstrate the advantages of our algorithm, we conducted a comparative analysis with projection gradient algorithm (PGA \cite{zhang2017projected}), proximal iterative hard thresholding methods (PIHT \cite{lu2014iterative}). The code of these methods is implemented in MATLAB R2023a and computed on a laptop with 12th Gen Intel(R) Core(TM) i7-12800HX 2.00 GHz CPU and 128 GB memory.

PGA is an algorithm for solving a relaxation problem of \eqref{BL0} where $\ell_0$ norm is replaced by $\ell_1$ norm. We implement the proximal gradient algorithm (PGA) which used the same line search as \cite{CHENG2023179}.

PIHT is a method for $\ell_0$-regularized convex cone programming. We implement this algorithm and take the same parameters setting as those in \cite{lu2014iterative}.

We set parameters in BL0R as the same as those in NL0R \cite{zhou2021newton} except for termination conditions. In order to make the most of the Newton step, we replace the condition \eqref{zerodecent} to \eqref{zerodecent2}. For these three methods, we set the halting conditions as the following
\begin{align}\label{eq-stop}
\frac{\|x^k - x^{k-1}\|}{\max(1,\|x^k\|) } \le 10^{-6},\ f(x^k) \le \epsilon,\ {\rm{iter}} > 2000.
\end{align}
The setting of $\epsilon$ is described in each respective example. For PGA, PIHT and BNL0R, the adjustment of the penalty parameter $\lambda$ is implemented using the same strategy outlined in \cite{zhou2021newton} . For all experiments, we set start point $x^0=[0,\cdots,0]^{\top}$.

The testing problems are from Compressed sensing which taken as the following form.
\begin{align*}
\min &\frac{1}{2}\|Ax - b\|^2 + \lambda \|x\|_0, \nonumber \\
\mbox{s.t.} & -l \le x \le u,
\end{align*}
where $A\in \mathbb{R}^{m\times n}$ is a data matrix, $b\in \mathbb{R}^m$ is an observation vector and $l, u\in \mathbb{R}_{+}^n$ is a boundary vector. Set $f(x) = \|Ax - b\|^2/2$. We will report the following results: dimension $n$, number of iterations, computation time (in seconds), and res. Here, res = $\|x^k - x^*\|$, where $x^*$ is the groundtruth solution.

\subsection{Noise-free signal recovery }
\textbf{E1} We consider the exact recovery $b = Ax^*$. The matrix $A$ is from \cite{zhang2018acceleratedproximaliterativehard}. We set the number of non-zero components $s = 0.001n$. $A \in \mathbb{R}^{m\times n}$ are random Gaussian matrix and the columns of $A$ are normalized to have $\ell_2$ norm of 1. $x^*$ is a random vector that follows a normal distribution with values between 0.1 and 3, which by the following codes.
\begin{align*}
x^* = {\rm{zeros}}(n,1),\ w = {\rm{randperm}}(n),\ x^*(w(1:s)) = 0.1 + (3-0.1){\rm{rand}}(s,1).
\end{align*}
In this case, we set the boundary vector $l=u = 3\times {\rm{ones}}(n,1)$ . We test different dimensions from $n = 5000$ to $n = 30000$, and different row numbers with $m =0.25n$ and $m = 0.15n$.  We run 20 trials for each example and take the average of the results and round the number of iterations iter to the nearest integer. In particular, for BNL0R, in one of the 20 trials, we examined the variation of the number of non-zero components and the objective function value during the iteration process. For \textbf{E1}, we set $\epsilon = 10^{-20}$ in \eqref{eq-stop}.

\begin{figure}[htbp]
	\centering
	\includegraphics[width=6cm]{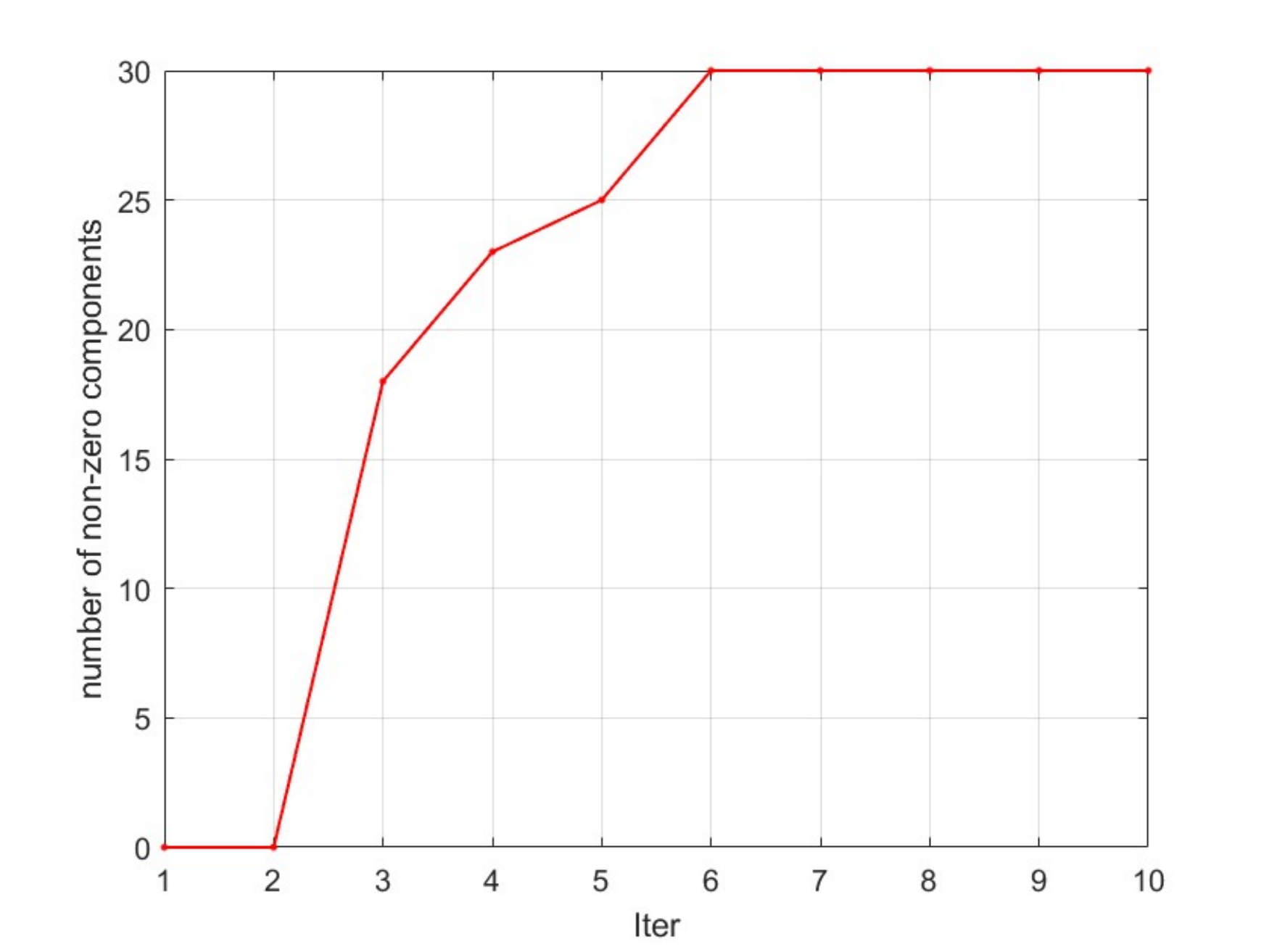
}
    \includegraphics[width=6cm]{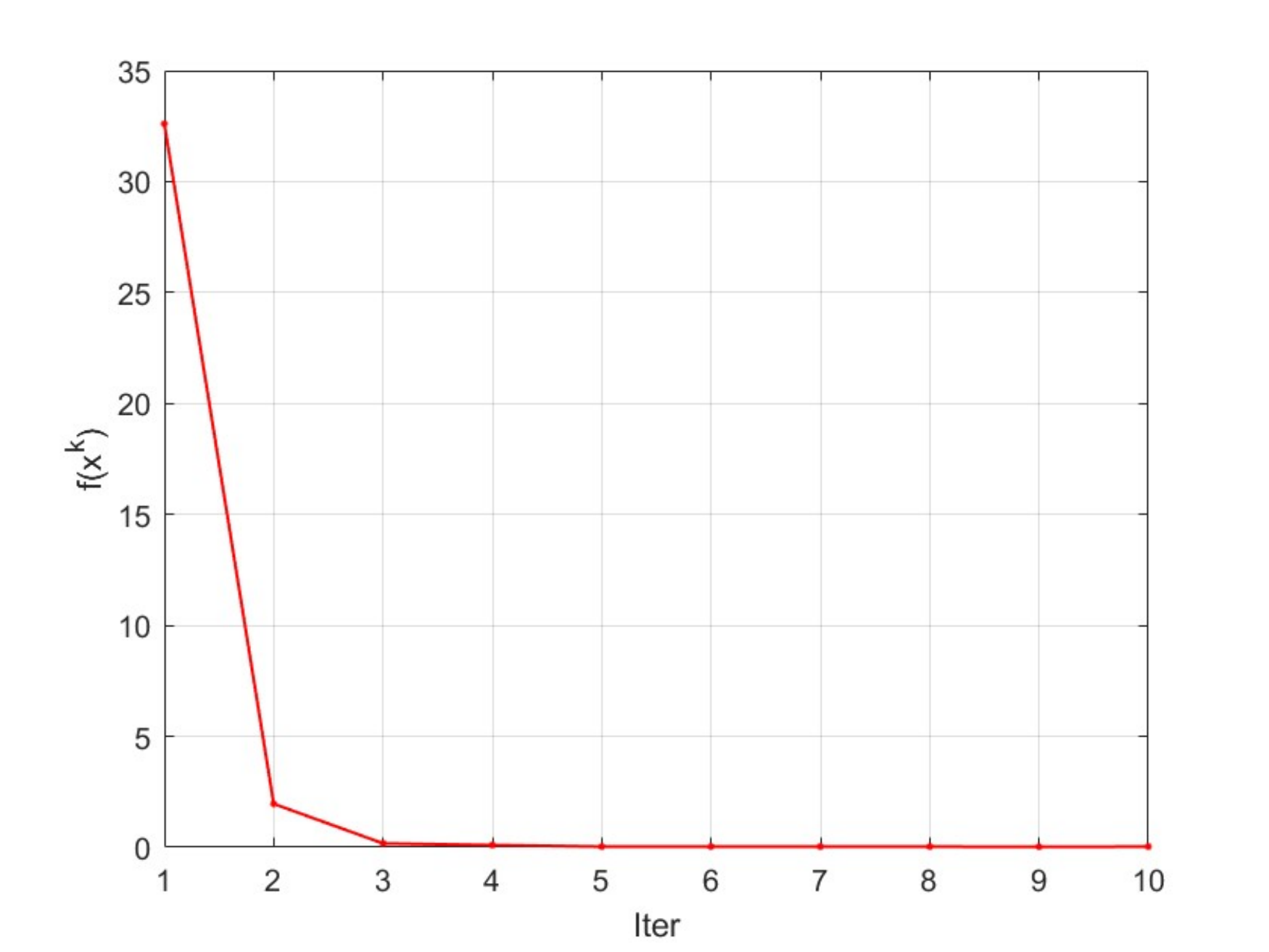
}
	\caption{The variation of the number of non-zero components and the objective function value during the iteration process of BNL0R for \textbf{E1} with $m = 0.25n$}
	\label{fig0}
\end{figure}
In Figures \ref{fig0}, we observe that the number of non-zero components of the iteration sequence $\{x^k\}$ is not very large, so the subspace Newton's method only needs to calculate a relatively small Hessian matrix. Which reduces the CPU time of Newton's method.
\begin{table*}[htbp]
\scriptsize
\centering
\caption{\text{Results on noise-free signal recovery with $m = 0.25n$}}\label{Tab01}
\begin{tabular*}{\textwidth}{@{\extracolsep\fill}ccccrcccr}
\toprule
\textbf{Algorithm}   &{n} &  iter     &  time   &   res
&{n}  &  iter      &  time   &   res \\
 \midrule
\text{PGA} & 5000  & 10  & 0.09  & 9.52e-02 & 10000  & 10  & 0.34  & 1.41e-01 \\
\text{PIHT} & 5000  & 10  & 0.02  & 1.68e-10 & 10000  & 10  & 0.11  & 1.89e-09 \\
\text{BNL0R} & 5000  & 04  & 0.01  & 8.12e-17 & 10000  & 04  & 0.04  & 3.65e-17 \\
 \midrule
\text{PGA} & 15000  & 10  & 0.83  & 1.20e-01 & 20000  & 10  & 1.62  & 1.24e-01 \\
\text{PIHT} & 15000  & 10  & 0.26  & 3.07e-09 & 20000  & 10  & 0.59  & 4.54e-09 \\
\text{BNL0R} & 15000  & 05  & 0.08  & 2.32e-17 & 20000  & 05  & 0.19  & 1.94e-17 \\
 \midrule
\text{PGA} & 25000  & 10  & 2.58  & 8.06e-02 & 30000  & 10  & 3.90  & 6.66e-02 \\
\text{PIHT} & 25000  & 10  & 0.91  & 7.98e-09 & 30000  & 10  & 1.30  & 6.48e-09 \\
\text{BNL0R} & 25000  & 05  & 0.30  & 1.79e-17 & 30000  & 06  & 0.44  & 1.60e-17 \\
\bottomrule
\end{tabular*}
\end{table*}

\begin{table*}[htbp]
\scriptsize
\centering
\caption{\text{Results on noise-free signal recovery with $m = 0.15n$}}\label{Tab02}
\begin{tabular*}{\textwidth}{@{\extracolsep\fill}ccccrcccr}
\toprule
\textbf{Algorithm}   &{n} &  iter     &  time   &   res
&{n}  &  iter      &  time   &   res \\
 \midrule
\text{PGA} & 5000  & 10  & 0.04  & 1.75e-01 & 10000  & 10  & 0.18  & 1.30e-01 \\
\text{PIHT} & 5000  & 10  & 0.01  & 9.84e-09 & 10000  & 11  & 0.07  & 1.32e-08 \\
\text{BNL0R} & 5000  & 04  & 0.01  & 6.82e-17 & 10000  & 05  & 0.03  & 1.15e-17 \\
 \midrule
\text{PGA} & 15000  & 10  & 0.39  & 2.06e-01 & 20000  & 10  & 0.65  & 1.27e-01 \\
\text{PIHT} & 15000  & 10  & 0.15  & 1.51e-08 & 20000  & 10  & 0.25  & 1.49e-08 \\
\text{BNL0R} & 15000  & 05  & 0.05  & 9.21e-18 & 20000  & 05  & 0.09  & 1.77e-17 \\
 \midrule
\text{PGA} & 25000  & 10  & 1.49  & 5.46e-01 & 30000  & 10  & 2.08  & 2.78e-01 \\
\text{PIHT} & 25000  & 10  & 0.56  & 1.98e-08 & 30000  & 10  & 0.76  & 3.24e-08 \\
\text{BNL0R} & 25000  & 05  & 0.18  & 2.23e-17 & 30000  & 06  & 0.26  & 2.23e-17 \\
\bottomrule
\end{tabular*}
\end{table*}

\begin{figure}[htbp]
	\centering
	\includegraphics[width=6cm]{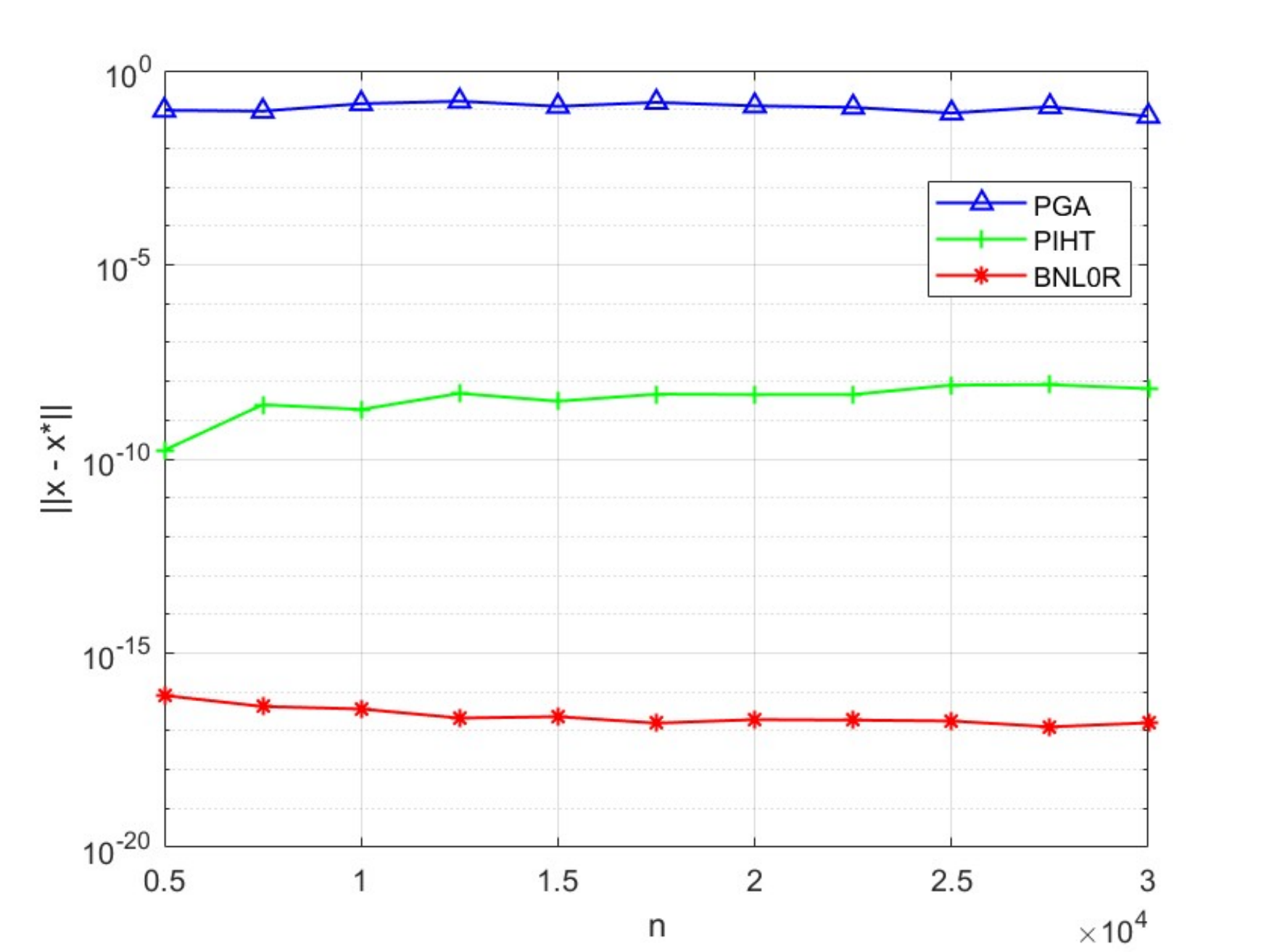
}
    \includegraphics[width=6cm]{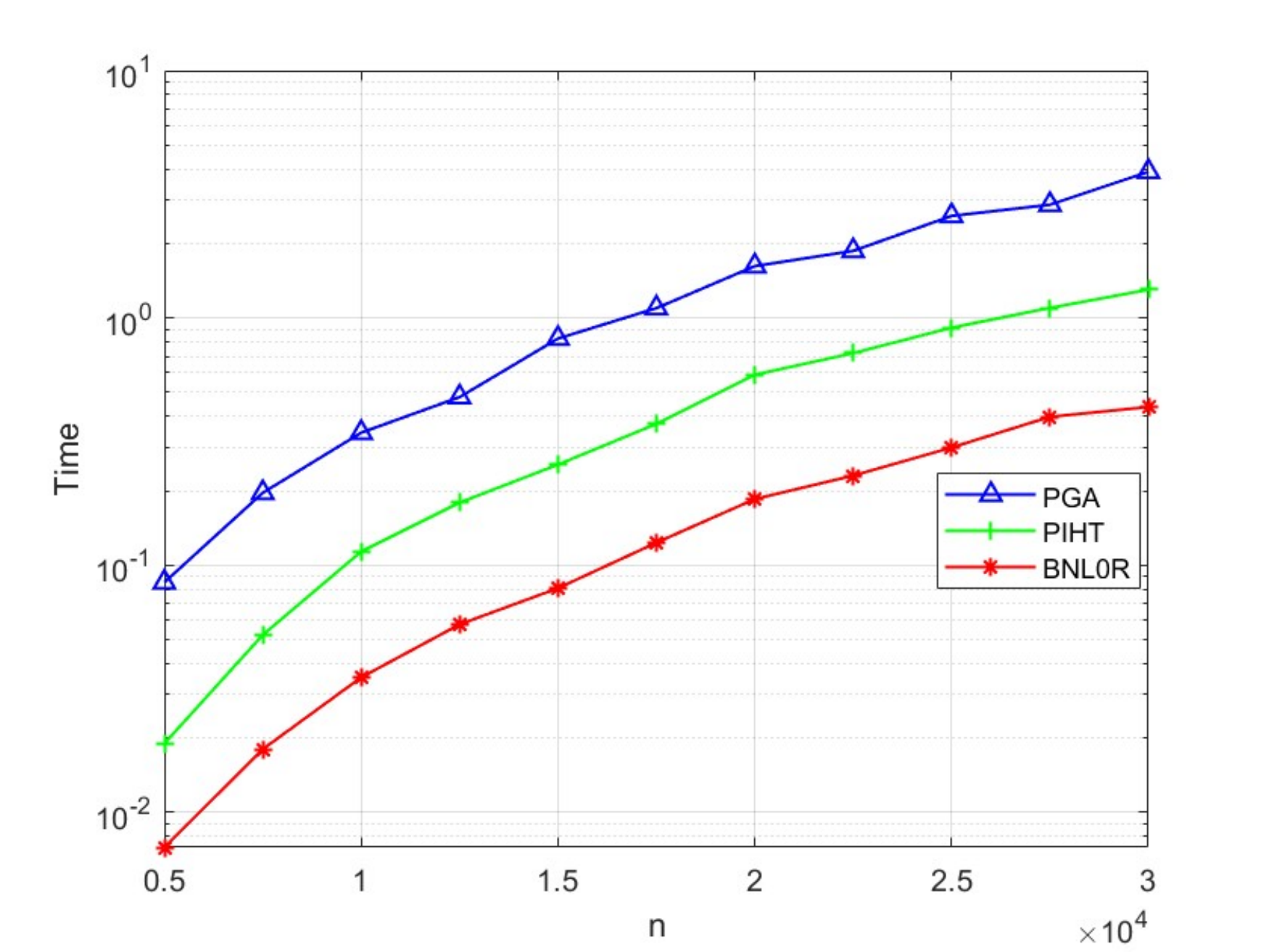
}
	\caption{Average recovery res and time for \textbf{E1} with $m = 0.25n$}
	\label{fig1}
\end{figure}

\begin{figure}[htbp]
	\centering
	\includegraphics[width=6cm]{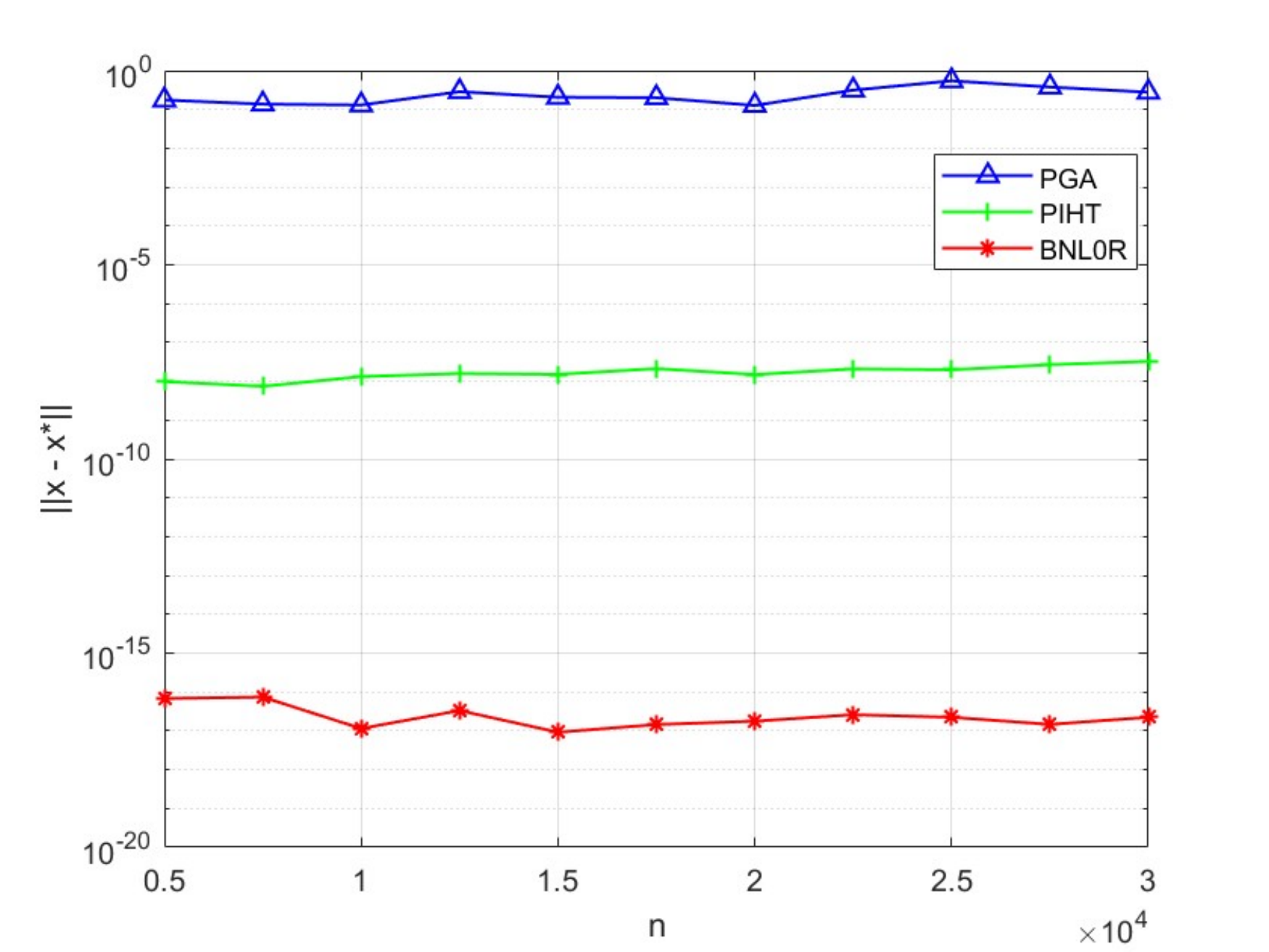
}
    \includegraphics[width=6cm]{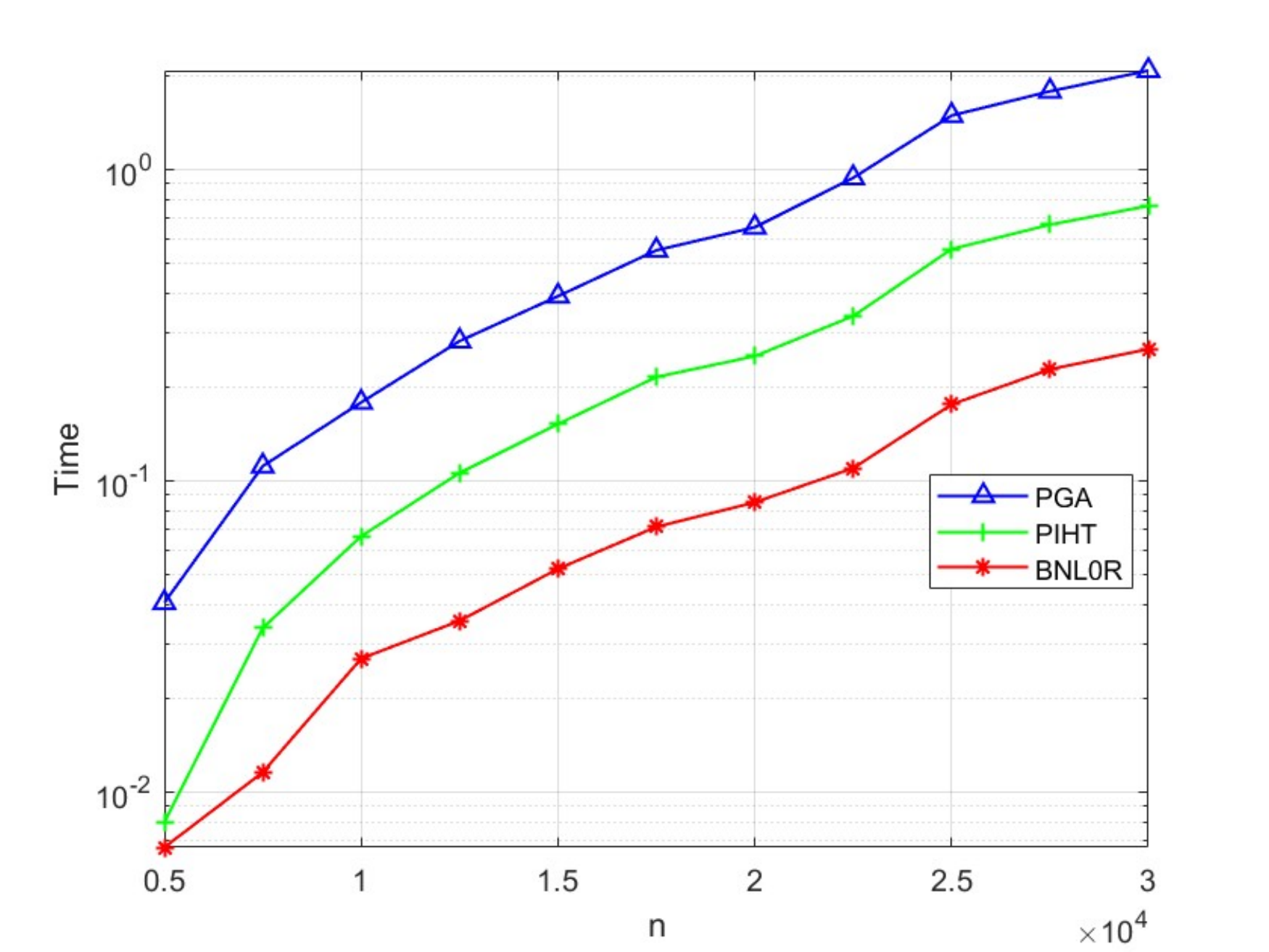
}
	\caption{Average recovery res and time for \textbf{E1} with $m = 0.15n$}
	\label{fig2}
\end{figure}
In Tables \ref{Tab01} and \ref{Tab02} as well as Figures \ref{fig1} and \ref{fig2}, we observe that our algorithm demonstrates better performance in terms of iteration count, CPU time and accuracy in \textbf{E1} compared to PGA and PIHT. This highlights the advantages of second-order algorithms over first-order algorithms. From Tables \ref{Tab01} and \ref{Tab02}, we can see that when the number of rows $ m $ decreases from $ m = 0.25n $ to $ m = 0.15n $, the accuracy of the PGA algorithm significantly declines. This is may attributed to an increase in matrix singularity.

\subsection{Noised signal recovery}
\textbf{E2} We consider the exact recovery $b = Ax^* + \xi$. The matrix $A$ is from \cite{zhang2018acceleratedproximaliterativehard}. In this case, we set the matrix $A$, the groundtruth solution and boundary vector are the same as those in \textbf{E1}. We set $\xi$ be a white Gaussian noise with the signal-to-noise ratio (SNR) by 30 dB. We test different dimensions from $n = 5000$ to $n = 30000$, and different row numbers with $m =0.25n$ and $m = 0.15n$.  We run 20 trials for each example and take the average of the results and round the number of iterations iter to the nearest integer. For \textbf{E2}, we set $\epsilon = 10^{-6}$ in \eqref{eq-stop}.

\begin{table*}[htbp]
\scriptsize
\centering
\caption{\text{Results on noised signal recovery with $m = 0.25n$}}\label{Tab03}
\begin{tabular*}{\textwidth}{@{\extracolsep\fill}ccccrcccr}
\toprule
\textbf{Algorithm}   &{n} &  iter     &  time   &   res
&{n}  &  iter      &  time   &   res \\
  \midrule
\text{PGA} & 5000  & 10  & 0.08  & 1.06e-01 & 10000  & 10  & 0.33  & 8.72e-02 \\
\text{PIHT} & 5000  & 12  & 0.03  & 5.83e-03 & 10000  & 12  & 0.21  & 1.98e-02 \\
\text{BNL0R} & 5000  & 10  & 0.03  & 3.79e-03 & 10000  & 10  & 0.11  & 2.24e-03 \\
 \midrule
\text{PGA} & 15000  & 10  & 0.65  & 9.38e-02 & 20000  & 10  & 1.42  & 7.27e-02 \\
\text{PIHT} & 15000  & 12  & 0.41  & 1.99e-02 & 20000  & 12  & 0.90  & 1.93e-02 \\
\text{BNL0R} & 15000  & 10  & 0.17  & 2.26e-03 & 20000  & 10  & 0.41  & 2.01e-03 \\
 \midrule
\text{PGA} & 25000  & 10  & 2.72  & 1.39e-01 & 30000  & 10  & 3.58  & 1.27e-01 \\
\text{PIHT} & 25000  & 12  & 1.68  & 2.02e-02 & 30000  & 12  & 2.55  & 2.01e-02 \\
\text{BNL0R} & 25000  & 10  & 0.71  & 1.94e-03 & 30000  & 10  & 1.08  & 2.04e-03 \\
\bottomrule
\end{tabular*}
\end{table*}

\begin{table*}[htbp]
\scriptsize
\centering
\caption{\text{Results on noised signal recovery with $m = 0.15n$}}\label{Tab04}
\begin{tabular*}{\textwidth}{@{\extracolsep\fill}ccccrcccr}
\toprule
\textbf{Algorithm}   &{n} &  iter     &  time   &   res
&{n}  &  iter      &  time   &   res \\
 \midrule
\text{PGA} & 5000  & 10  & 0.04  & 2.26e-01 & 10000  & 10  & 0.20  & 1.22e-01 \\
\text{PIHT} & 5000  & 12  & 0.02  & 1.11e-02 & 10000  & 12  & 0.13  & 1.50e-02 \\
\text{BNL0R} & 5000  & 10  & 0.01  & 3.25e-03 & 10000  & 10  & 0.07  & 4.92e-03 \\
 \midrule
\text{PGA} & 15000  & 10  & 0.43  & 2.22e-01 & 20000  & 10  & 0.74  & 4.18e-01 \\
\text{PIHT} & 15000  & 12  & 0.34  & 1.74e-02 & 20000  & 12  & 0.68  & 1.69e-02 \\
\text{BNL0R} & 15000  & 10  & 0.19  & 4.95e-03 & 20000  & 10  & 0.32  & 4.51e-03 \\
 \midrule
\text{PGA} & 25000  & 10  & 1.06  & 2.44e-01 & 30000  & 10  & 1.51  & 2.92e-01 \\
\text{PIHT} & 25000  & 13  & 1.02  & 1.73e-02 & 30000  & 12  & 1.50  & 1.59e-02 \\
\text{BNL0R} & 25000  & 10  & 0.45  & 4.46e-03 & 30000  & 10  & 0.66  & 3.45e-03 \\
\bottomrule
\end{tabular*}
\end{table*}

\begin{figure}[htbp]
	\centering
	\includegraphics[width=6cm]{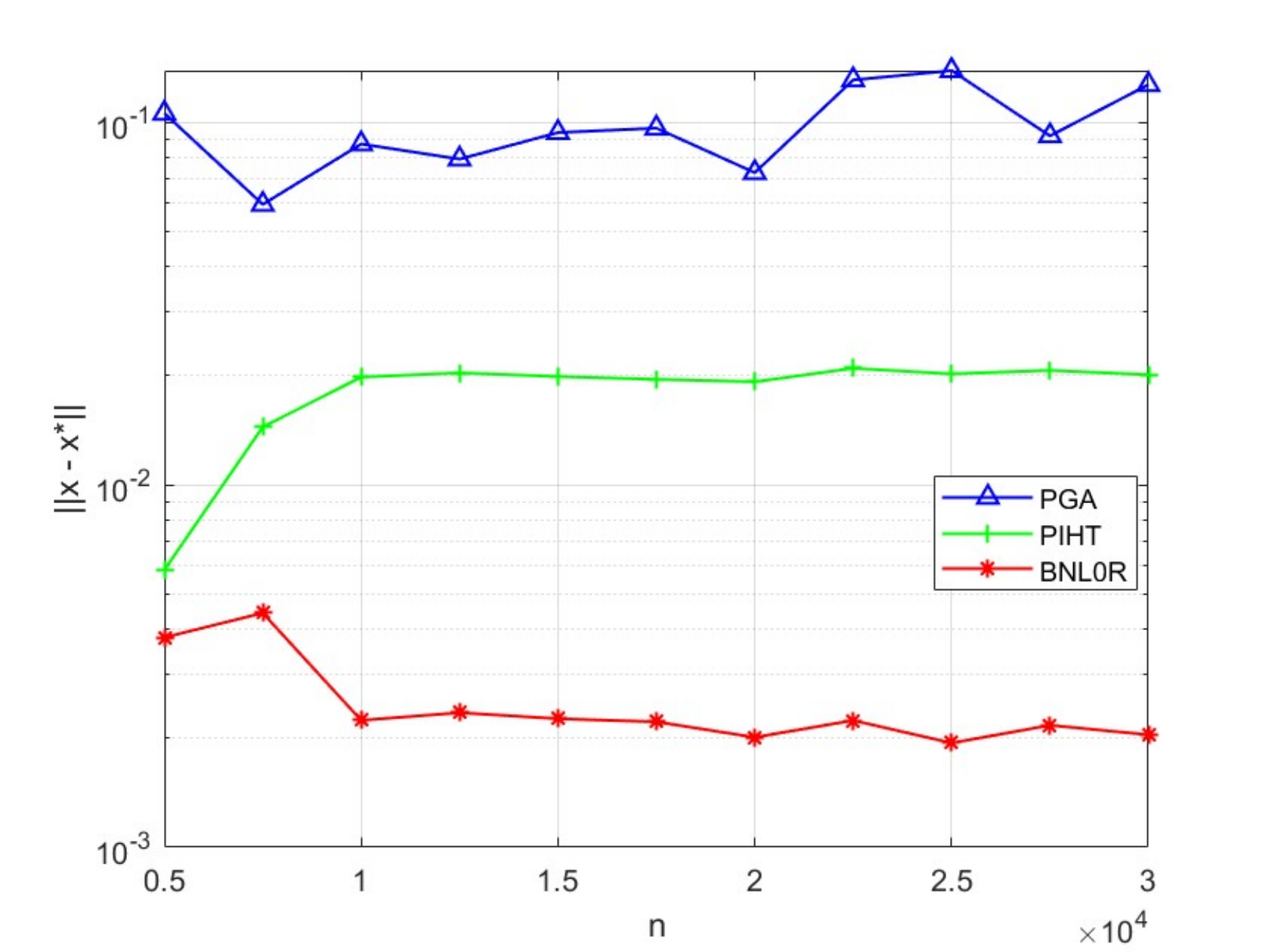
}
    \includegraphics[width=6cm]{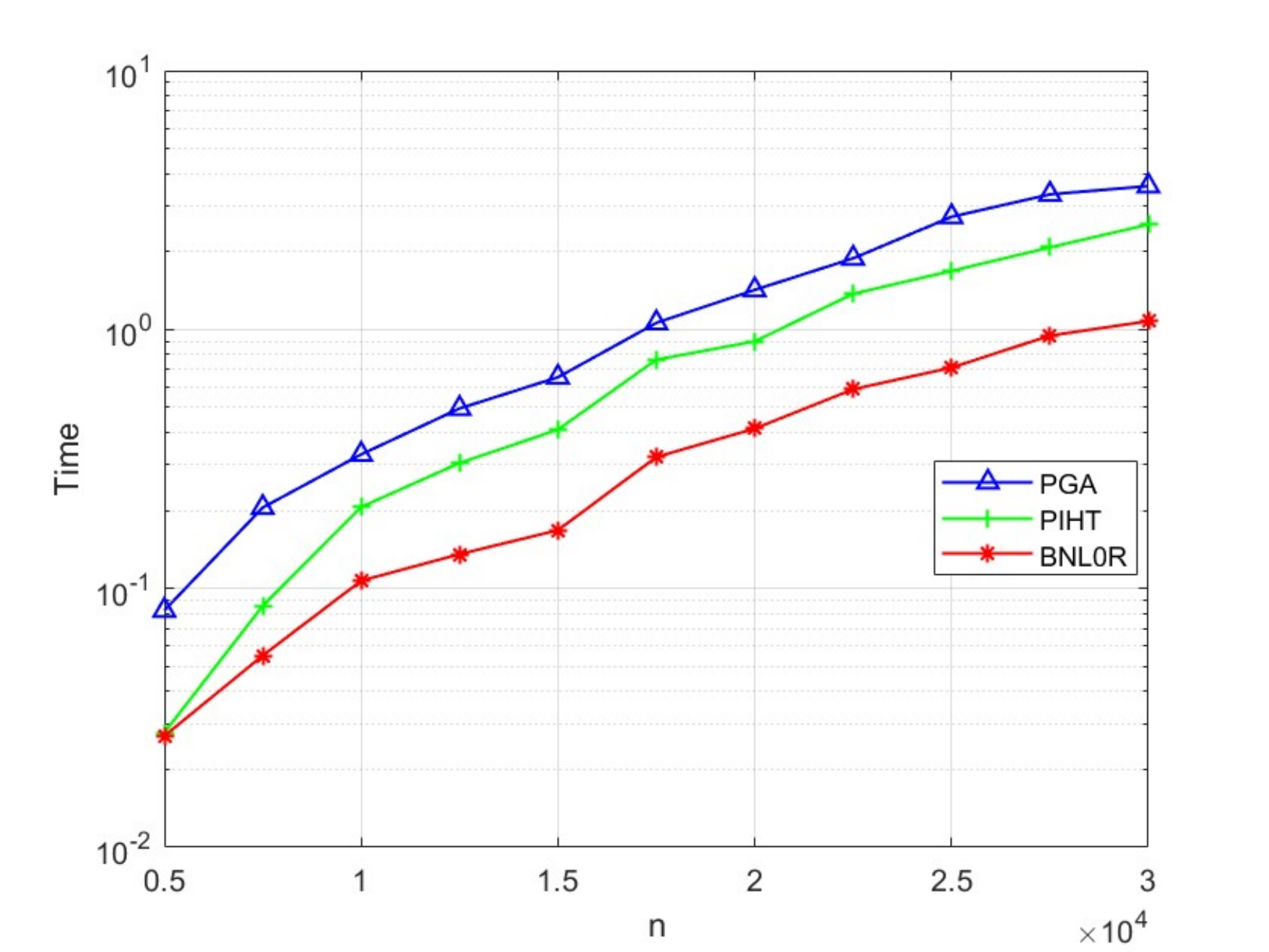
}
	\caption{Average recovery res and time for \textbf{E2} with $m = 0.25n$}
	\label{fig3}
\end{figure}

\begin{figure}[htbp]
	\centering
	\includegraphics[width=6cm]{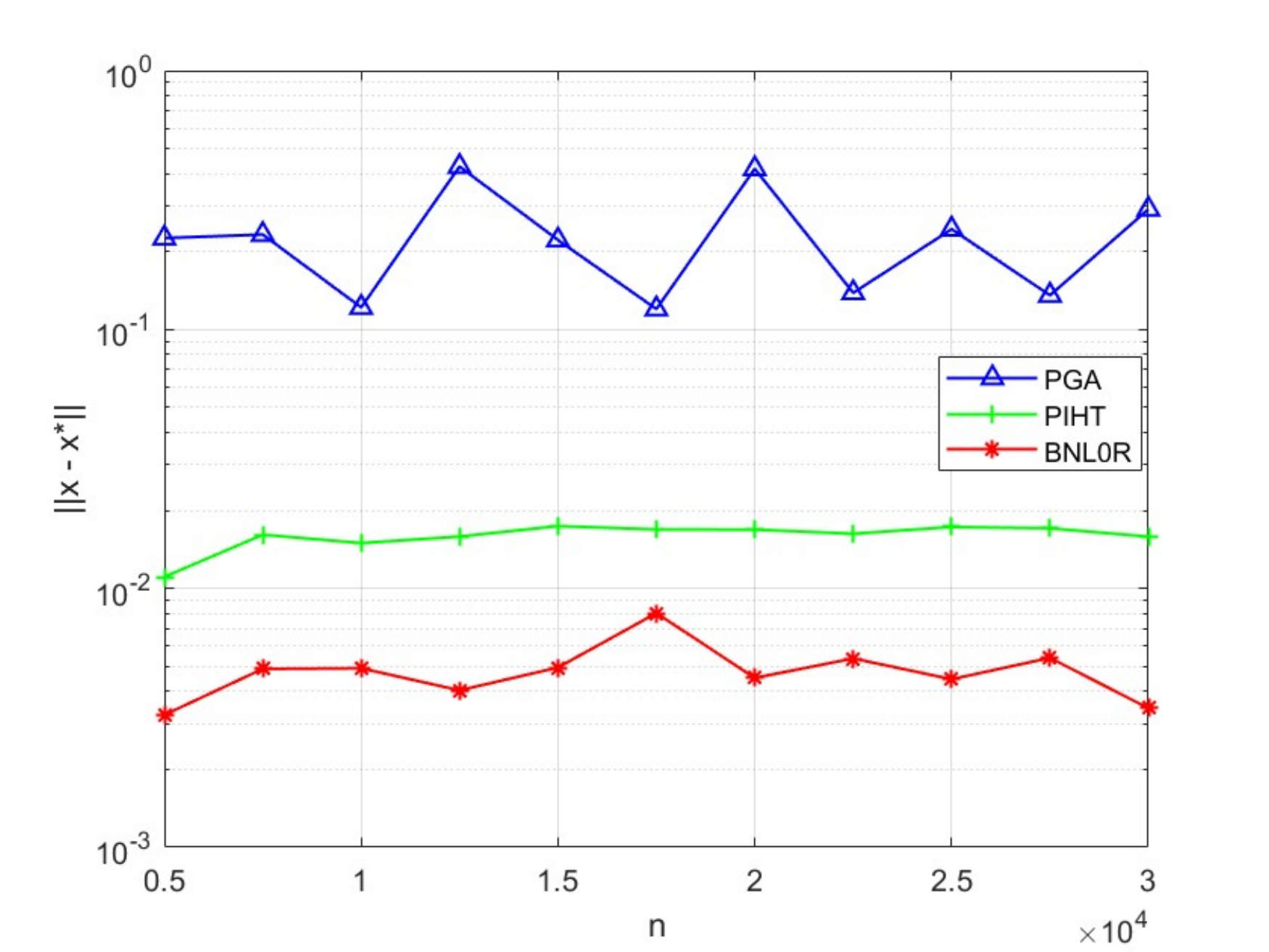
}
    \includegraphics[width=6cm]{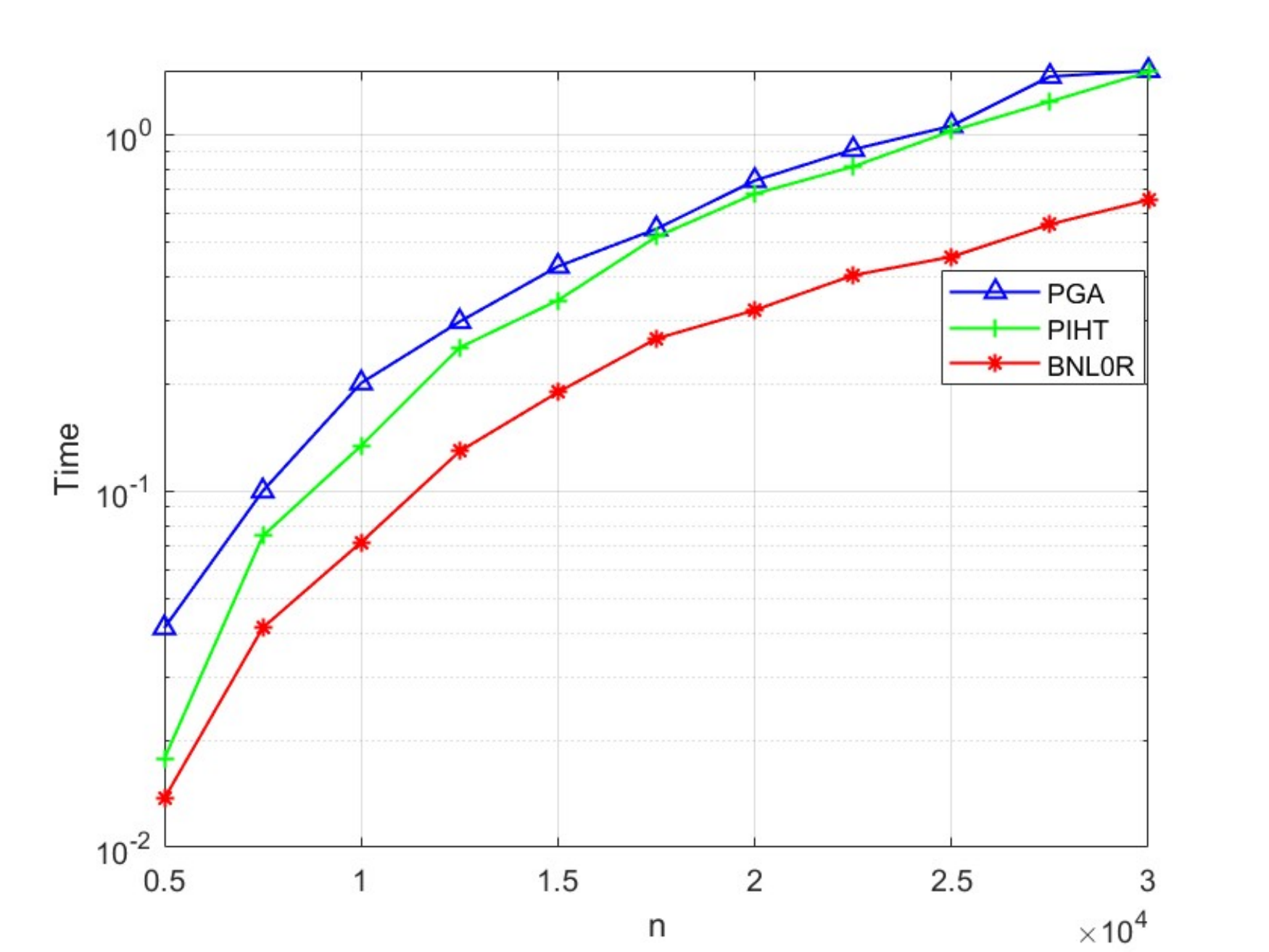
}
	\caption{Average recovery res and time for \textbf{E2} with $m = 0.15n$}
	\label{fig4}
\end{figure}

In Tables \ref{Tab03} and \ref{Tab04}, as well as Figures \ref{fig3} and \ref{fig4}, we also observe that our algorithm posses better performance in terms of iteration numbers, CPU time and accuracy in \textbf{E2} compared to PGA and PIHT. From \ref{Tab03} and \ref{Tab04}, we observe that BNL0R achieves the highest computational accuracy, followed by PIHT, while PGA exhibits the lowest accuracy. This indicates two points: first, it shows the advantages of the second-order algorithms; second, it highlights the superiority of the $\ell_0$ regularization model compared to the $\ell_1$ regularization model.

\subsection{Noised signal recovery with different box constraints }
\textbf{E3} We consider the exact recovery $b = Ax^* + \xi$. We set the number of non-zero components $s = 100$. $A \in \mathbb{R}^{m\times n}$ are random Gaussian matrix . $\xi$ is white Gaussian noise with the signal-to-noise ratio (SNR) by 30 dB. We test different dimensions from $n = 10000$ to $n = 30000$ while $m = 0.25n$. For each dimension case, we uniformly divide the groundtruth solution $x^*$ into four intervals, randomly selecting 10 positions within each interval and assigning random non-zero values. Similarly, we divide the boundary $l$, $u$ into four intervals. That is
\begin{align*}
w^1 &= {\rm{randperm}}(n/4),\ \mathcal{I}^2 = w^1(1:25) , \\
w^2 &= {\rm{randperm}}(n/4) + (2-1)\frac{n}{4},\ \mathcal{I}^2 = w^2(1:25).
\end{align*}
While $w^3$, $w^4$, $\mathcal{I}^3$ and $\mathcal{I}^4$ are similar setting. The groundtruth solution is set as follows
\begin{align*}
x^*(\mathcal{I}^1) &= {\rm{rand}}(25, 1),\ x^*(\mathcal{I}^2) = 2\times {\rm{rand}}(25, 1),\\
x^*(\mathcal{I}^3) &= 3\times {\rm{rand}}(25, 1),\ x^*(\mathcal{I}^4) = 4\times {\rm{rand}}(25, 1).
\end{align*}
The box constraint $[-l, u]$ taken as
\begin{align*}
[-2,2]^{n/4}\times [-3,3]^{n/4}\times  [-4,4]^{n/4}\times  [-5,5]^{n/4}.
\end{align*}
We run 20 trials for each example and take the average of the results and round the number of iterations iter to the nearest integer. For \textbf{E3}, we set $\epsilon = 10^{-6}$ in \eqref{eq-stop}.

\begin{table*}[htbp]
\scriptsize
\centering
\caption{\text{Results on E3 with $m = n$}}\label{Tab05}
\begin{tabular*}{\textwidth}{@{\extracolsep\fill}ccccrcccr}
\toprule
\textbf{Algorithm}   &{n} &  iter     &  time   &   res
&{n}  &  iter      &  time   &   res \\
 \midrule
\text{PGA} & 12000  & 10  &  0.45  & 1.74e-01 & 14000  & 10  &  0.63  & 9.47e-02 \\
\text{PIHT} & 12000  & 13  & 0.41  & 2.18e-02 & 14000  & 13  & 0.58  & 2.12e-02 \\
\text{BNL0R} & 12000  & 10  & 0.30  & 7.71e-03 & 14000  & 10  & 0.37  & 6.55e-03 \\
\midrule
\text{PGA} & 16000  & 10  &  0.90  & 1.47e-01 & 18000  & 10  &  0.98  & 1.33e-01 \\
\text{PIHT} & 16000  & 13  & 0.77  & 2.23e-02 & 18000  & 13  & 0.93  & 2.03e-02 \\
\text{BNL0R} & 16000  & 10  & 0.38  & 5.78e-03 & 18000  & 10  & 0.45  & 5.25e-03 \\
 \midrule
\text{PGA} & 20000  & 10  &  1.37  & 3.62e-01 & 22000  & 10  &  1.36  & 2.63e-01 \\
\text{PIHT} & 20000  & 13  & 1.10  & 2.03e-02 & 22000  & 13  & 1.21  & 1.93e-02 \\
\text{BNL0R} & 20000  & 10  & 0.47  & 4.83e-03 & 22000  & 10  & 0.52  & 4.57e-03 \\
 \midrule
\text{PGA} & 24000  & 10  &  1.81  & 1.89e-01 & 26000  & 10  &  2.02  & 2.43e-01 \\
\text{PIHT} & 24000  & 13  & 1.43  & 1.99e-02 & 26000  & 13  & 1.59  & 2.00e-02 \\
\text{BNL0R} & 24000  & 10  & 0.58  & 4.28e-03 & 26000  & 10  & 0.67  & 4.19e-03 \\
 \midrule
\text{PGA} & 28000  & 10  &  2.62  & 1.78e-01 & 30000  & 10  &  3.75  & 9.60e-02 \\
\text{PIHT} & 28000  & 13  & 1.92  & 2.14e-02 & 30000  & 13  & 2.17  & 2.26e-02 \\
\text{BNL0R} & 28000  & 10  & 0.74  & 4.26e-03 & 30000  & 10  & 0.90  & 5.27e-03 \\
\bottomrule
\end{tabular*}
\end{table*}

\begin{figure}[htbp]
	\centering
	\includegraphics[width=6cm]{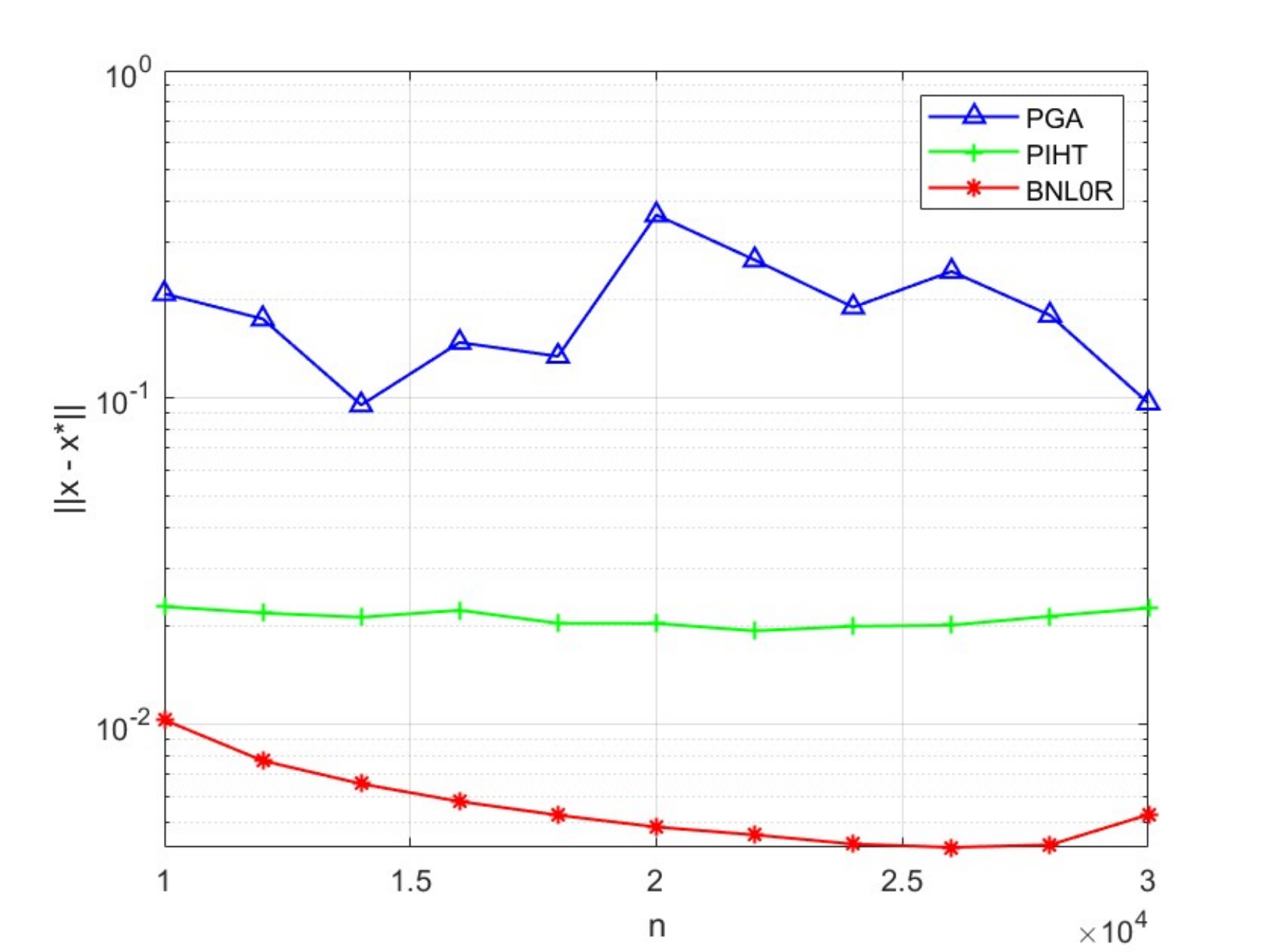
}
    \includegraphics[width=6cm]{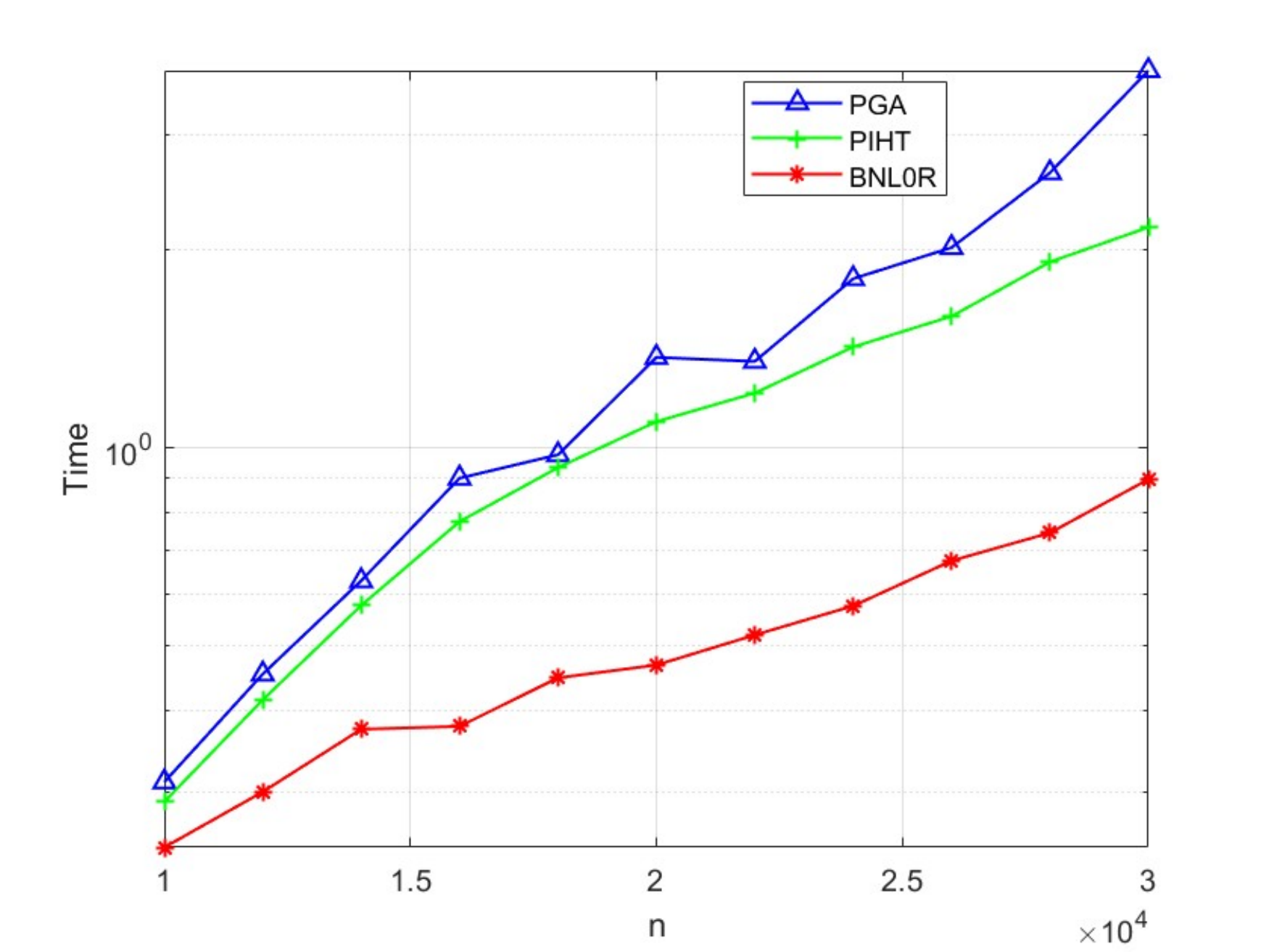
}
	\caption{Average recovery res and time for \textbf{E3} with $m = 0.25n$}
	\label{fig5}
\end{figure}

The computational data for \textbf{E3} is presented in Table \ref{Tab05} and Figure \ref{fig5}. In Figure \ref{fig5}, we observe that BNL0R takes a advantage in terms of computation time and accuracy.

\subsection{2D-image recovery}
\textbf{E4} (2-D image data) \cite{zhou2021newton}
Some images are naturally not sparse themselves but can be sparse under some wavelet transforms. Here, we take advantage of the Daubechies wavelet 1, denoted as $W(\cdot)$. Then the images under this transform (i.e., $x^* := W(\omega)$) are sparse, and $\omega$ is the vectorized intensity of an input image. Therefore, the explicit form of the sampling matrix may not be available. We consider the exact recovery $y = Ax^* + \xi$. We consider a sampling matrix taking the form $A = FW^{-1}$, where $F$ is the partial fast Fourier transform, and $W^{-1}$ is the inverse of $W$. Finally, the added noise $\xi$ has each element $\xi_i \sim {\rm{nf}} \cdot \mathcal{N}$ where $\mathcal{N}$ is the standard normal distribution and $n_f$ is the noise factor. Two typical choices of ${\rm{nf}}$ are considered, namely ${\rm{nf}} \in \{0.001,0.05, 0.1\}$. We set the boundary vector $l=u = 10\times {\rm{ones}}(n,1)$. For this experiment, we compute a gray image (see the original image in Fig. \ref{fig6}) with size $256 \times 256$ (i.e. $n = 256^2 = 65536$) and the sampling sizes $m = 14369$. We compute the peak signal to noise ratio (PSNR) defined by $\text{PSNR} := 10\log_{10}(n\|x - x^*\|^{-2})$ to measure the performance of the method. Note that a larger PSNR implies a better performance. For \textbf{E4}, we set $\epsilon = \|Ax^*-y\|$ in \eqref{eq-stop}.

\begin{table*}[htbp]
\scriptsize
\centering
\caption{\text{Results on 2D-image recovery}}\label{Tab06}
\begin{tabular*}{\textwidth}{@{\extracolsep\fill}cccccccccc}
\toprule
\multirow{2}{*}{\text{Method}} &\multicolumn{3}{c}{${\rm{nf}} = 0.01$} &\multicolumn{3}{c}{${\rm{nf}} = 0.05$} &\multicolumn{3}{c}{${\rm{nf}} = 0.1$} \\
\cmidrule{2-10} &PSNR    &time   & $\|x\|_0$ &PSNR     &time   & $\|x\|_0$ &PSNR    &time   & $\|x\|_0$ \\
 \midrule
\text{PGA} & 21.66  & 1.21e-1 & 64471 & 21.37  & 1.26e-1  & 61431 & 20.36  & 1.31e-1  & 59927  \\
\text{PIHT} & 21.76   & 8.04e-2 & 28115 & 21.54  & 7.09e-2 & 23690 & 20.81  & 6.67e-2  & 22298  \\
\text{BNL0R} & 38.98  & 5.50e-1 & 3939  & 25.72 & 3.96e-1 & 2952 & 22.91  & 1.68e-1  & 3281 \\
\bottomrule
\end{tabular*}
\end{table*}

\begin{figure}[htbp]
	\centering
	\includegraphics[width=12cm]{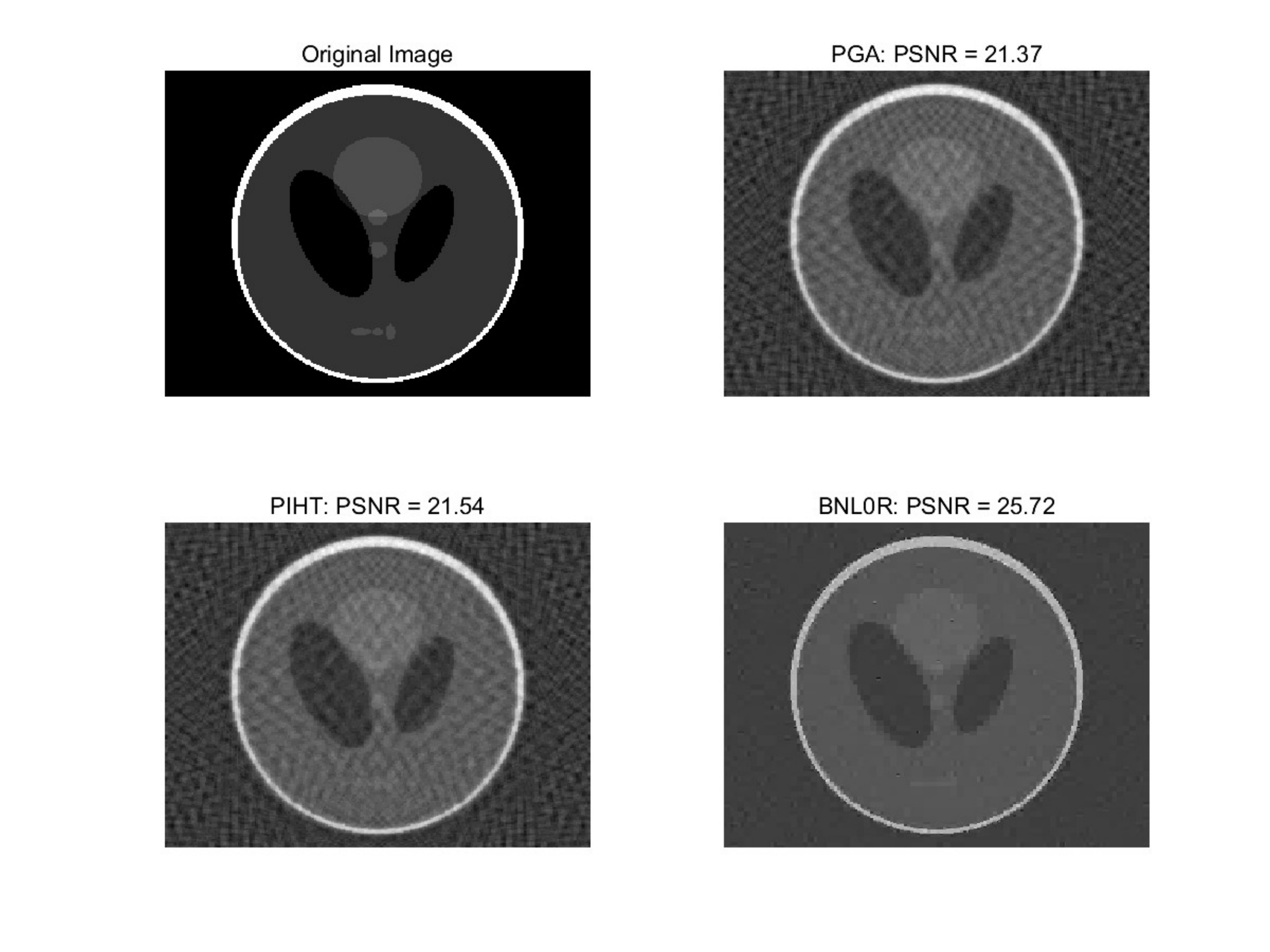
}
	\caption{Recovery results for \textbf{E4} with ${\rm{nf}} = 0.05$}
	\label{fig6}
\end{figure}

For \textbf{E4}, we conducted experiments using the four aforementioned algorithms. The computational results are presented in Table \ref{Tab06} and Figure \ref{fig6}. From Figure \ref{fig6}, we observe that BNL0R achieve the highest PSNR. This indicates that our algorithm maintains its efficiency with some cases of image recovery.

\section{Conclusion}
In this paper, we proposed a subspace Newton's method (BNL0R) for $\ell_0$-regularized optimization with box constraint. We first identify the active and inactive variables of the support set which is defined by  $\tau$-stationary point of \eqref{BL0}, Then we use the Newton's method to update the inactive variables and use the proximal gradient method to update the active variables. Under appropriate conditions, we prove the global convergence and the local quadratic convergence rate. The efficiency of the proposed method is demonstrated by the experimental results.


\bibliography{sn-bibliography}


\begin{thebibliography}{40}
\ifx \bisbn   \undefined \def \bisbn  #1{ISBN #1}\fi
\ifx \binits  \undefined \def \binits#1{#1}\fi
\ifx \bauthor  \undefined \def \bauthor#1{#1}\fi
\ifx \batitle  \undefined \def \batitle#1{#1}\fi
\ifx \bjtitle  \undefined \def \bjtitle#1{#1}\fi
\ifx \bvolume  \undefined \def \bvolume#1{\textbf{#1}}\fi
\ifx \byear  \undefined \def \byear#1{#1}\fi
\ifx \bissue  \undefined \def \bissue#1{#1}\fi
\ifx \bfpage  \undefined \def \bfpage#1{#1}\fi
\ifx \blpage  \undefined \def \blpage #1{#1}\fi
\ifx \burl  \undefined \def \burl#1{\textsf{#1}}\fi
\ifx \doiurl  \undefined \def \doiurl#1{\url{https://doi.org/#1}}\fi
\ifx \betal  \undefined \def \betal{\textit{et al.}}\fi
\ifx \binstitute  \undefined \def \binstitute#1{#1}\fi
\ifx \binstitutionaled  \undefined \def \binstitutionaled#1{#1}\fi
\ifx \bctitle  \undefined \def \bctitle#1{#1}\fi
\ifx \beditor  \undefined \def \beditor#1{#1}\fi
\ifx \bpublisher  \undefined \def \bpublisher#1{#1}\fi
\ifx \bbtitle  \undefined \def \bbtitle#1{#1}\fi
\ifx \bedition  \undefined \def \bedition#1{#1}\fi
\ifx \bseriesno  \undefined \def \bseriesno#1{#1}\fi
\ifx \blocation  \undefined \def \blocation#1{#1}\fi
\ifx \bsertitle  \undefined \def \bsertitle#1{#1}\fi
\ifx \bsnm \undefined \def \bsnm#1{#1}\fi
\ifx \bsuffix \undefined \def \bsuffix#1{#1}\fi
\ifx \bparticle \undefined \def \bparticle#1{#1}\fi
\ifx \barticle \undefined \def \barticle#1{#1}\fi
\bibcommenthead
\ifx \bconfdate \undefined \def \bconfdate #1{#1}\fi
\ifx \botherref \undefined \def \botherref #1{#1}\fi
\ifx \url \undefined \def \url#1{\textsf{#1}}\fi
\ifx \bchapter \undefined \def \bchapter#1{#1}\fi
\ifx \bbook \undefined \def \bbook#1{#1}\fi
\ifx \bcomment \undefined \def \bcomment#1{#1}\fi
\ifx \oauthor \undefined \def \oauthor#1{#1}\fi
\ifx \citeauthoryear \undefined \def \citeauthoryear#1{#1}\fi
\ifx \endbibitem  \undefined \def \endbibitem {}\fi
\ifx \bconflocation  \undefined \def \bconflocation#1{#1}\fi
\ifx \arxivurl  \undefined \def \arxivurl#1{\textsf{#1}}\fi
\csname PreBibitemsHook\endcsname

\bibitem[\protect\citeauthoryear{Cand{\`e}s et~al.}{2006}]{candes2006robust}
\begin{barticle}
\bauthor{\bsnm{Cand{\`e}s}, \binits{E.J.}},
\bauthor{\bsnm{Romberg}, \binits{J.}},
\bauthor{\bsnm{Tao}, \binits{T.}}:
\batitle{Robust uncertainty principles: Exact signal reconstruction from highly
  incomplete frequency information}.
\bjtitle{IEEE Transactions on information theory}
\bvolume{52}(\bissue{2}),
\bfpage{489}--\blpage{509}
(\byear{2006})
\end{barticle}
\endbibitem

\bibitem[\protect\citeauthoryear{Candes and Tao}{2005}]{1542412}
\begin{barticle}
\bauthor{\bsnm{Candes}, \binits{E.J.}},
\bauthor{\bsnm{Tao}, \binits{T.}}:
\batitle{Decoding by linear programming}.
\bjtitle{IEEE Transactions on Information Theory}
\bvolume{51}(\bissue{12}),
\bfpage{4203}--\blpage{4215}
(\byear{2005})
\doiurl{10.1109/TIT.2005.858979}
\end{barticle}
\endbibitem

\bibitem[\protect\citeauthoryear{Donoho}{2006}]{donoho2006compressed}
\begin{barticle}
\bauthor{\bsnm{Donoho}, \binits{D.L.}}:
\batitle{Compressed sensing}.
\bjtitle{IEEE Transactions on information theory}
\bvolume{52}(\bissue{4}),
\bfpage{1289}--\blpage{1306}
(\byear{2006})
\end{barticle}
\endbibitem

\bibitem[\protect\citeauthoryear{Wright et~al.}{2010}]{wright2010sparse}
\begin{barticle}
\bauthor{\bsnm{Wright}, \binits{J.}},
\bauthor{\bsnm{Ma}, \binits{Y.}},
\bauthor{\bsnm{Mairal}, \binits{J.}},
\bauthor{\bsnm{Sapiro}, \binits{G.}},
\bauthor{\bsnm{Huang}, \binits{T.S.}},
\bauthor{\bsnm{Yan}, \binits{S.}}:
\batitle{Sparse representation for computer vision and pattern recognition}.
\bjtitle{Proceedings of the IEEE}
\bvolume{98}(\bissue{6}),
\bfpage{1031}--\blpage{1044}
(\byear{2010})
\end{barticle}
\endbibitem

\bibitem[\protect\citeauthoryear{Yuan et~al.}{2012}]{yuan2012visual}
\begin{barticle}
\bauthor{\bsnm{Yuan}, \binits{X.-T.}},
\bauthor{\bsnm{Liu}, \binits{X.}},
\bauthor{\bsnm{Yan}, \binits{S.}}:
\batitle{Visual classification with multitask joint sparse representation}.
\bjtitle{IEEE Transactions on Image Processing}
\bvolume{21}(\bissue{10}),
\bfpage{4349}--\blpage{4360}
(\byear{2012})
\end{barticle}
\endbibitem

\bibitem[\protect\citeauthoryear{Bian and Chen}{2012}]{bian2012smoothing}
\begin{barticle}
\bauthor{\bsnm{Bian}, \binits{W.}},
\bauthor{\bsnm{Chen}, \binits{X.}}:
\batitle{Smoothing neural network for constrained non-lipschitz optimization
  with applications}.
\bjtitle{IEEE transactions on neural networks and learning systems}
\bvolume{23}(\bissue{3}),
\bfpage{399}--\blpage{411}
(\byear{2012})
\end{barticle}
\endbibitem

\bibitem[\protect\citeauthoryear{Dinh et~al.}{2020}]{dinh2020sparsity}
\begin{bchapter}
\bauthor{\bsnm{Dinh}, \binits{T.}},
\bauthor{\bsnm{Wang}, \binits{B.}},
\bauthor{\bsnm{Bertozzi}, \binits{A.}},
\bauthor{\bsnm{Osher}, \binits{S.}},
\bauthor{\bsnm{Xin}, \binits{J.}}:
\bctitle{Sparsity meets robustness: Channel pruning for the feynman-kac
  formalism principled robust deep neural nets}.
In: \bbtitle{Machine Learning, Optimization, and Data Science: 6th
  International Conference, LOD 2020, Siena, Italy, July 19--23, 2020, Revised
  Selected Papers, Part II 6},
pp. \bfpage{362}--\blpage{381}
(\byear{2020}).
\bcomment{Springer}
\end{bchapter}
\endbibitem

\bibitem[\protect\citeauthoryear{Lin et~al.}{2019}]{lin2019toward}
\begin{barticle}
\bauthor{\bsnm{Lin}, \binits{S.}},
\bauthor{\bsnm{Ji}, \binits{R.}},
\bauthor{\bsnm{Li}, \binits{Y.}},
\bauthor{\bsnm{Deng}, \binits{C.}},
\bauthor{\bsnm{Li}, \binits{X.}}:
\batitle{Toward compact convnets via structure-sparsity regularized filter
  pruning}.
\bjtitle{IEEE transactions on neural networks and learning systems}
\bvolume{31}(\bissue{2}),
\bfpage{574}--\blpage{588}
(\byear{2019})
\end{barticle}
\endbibitem

\bibitem[\protect\citeauthoryear{Bian and Chen}{2015}]{bian2015linearly}
\begin{barticle}
\bauthor{\bsnm{Bian}, \binits{W.}},
\bauthor{\bsnm{Chen}, \binits{X.}}:
\batitle{Linearly constrained non-lipschitz optimization for image
  restoration}.
\bjtitle{SIAM Journal on Imaging Sciences}
\bvolume{8}(\bissue{4}),
\bfpage{2294}--\blpage{2322}
(\byear{2015})
\end{barticle}
\endbibitem

\bibitem[\protect\citeauthoryear{Chen et~al.}{2012}]{chen2012non}
\begin{barticle}
\bauthor{\bsnm{Chen}, \binits{X.}},
\bauthor{\bsnm{Ng}, \binits{M.K.}},
\bauthor{\bsnm{Zhang}, \binits{C.}}:
\batitle{Non-lipschitz $\ell_p$-regularization and box constrained model for
  image restoration}.
\bjtitle{IEEE Transactions on Image Processing}
\bvolume{21}(\bissue{12}),
\bfpage{4709}--\blpage{4721}
(\byear{2012})
\end{barticle}
\endbibitem

\bibitem[\protect\citeauthoryear{Elad}{2010}]{elad2010sparse}
\begin{bbook}
\bauthor{\bsnm{Elad}, \binits{M.}}:
\bbtitle{Sparse and Redundant Representations: from Theory to Applications in
  Signal and Image Processing}.
\bpublisher{Springer}
(\byear{2010})
\end{bbook}
\endbibitem

\bibitem[\protect\citeauthoryear{Elad et~al.}{2010}]{elad2010role}
\begin{barticle}
\bauthor{\bsnm{Elad}, \binits{M.}},
\bauthor{\bsnm{Figueiredo}, \binits{M.A.}},
\bauthor{\bsnm{Ma}, \binits{Y.}}:
\batitle{On the role of sparse and redundant representations in image
  processing}.
\bjtitle{Proceedings of the IEEE}
\bvolume{98}(\bissue{6}),
\bfpage{972}--\blpage{982}
(\byear{2010})
\end{barticle}
\endbibitem

\bibitem[\protect\citeauthoryear{Cand\`{e}s and
  Recht}{2012}]{10.1145/2184319.2184343}
\begin{barticle}
\bauthor{\bsnm{Cand\`{e}s}, \binits{E.}},
\bauthor{\bsnm{Recht}, \binits{B.}}:
\batitle{Exact matrix completion via convex optimization}.
\bjtitle{Commun. ACM}
\bvolume{55}(\bissue{6}),
\bfpage{111}--\blpage{119}
(\byear{2012})
\doiurl{10.1145/2184319.2184343}
\end{barticle}
\endbibitem

\bibitem[\protect\citeauthoryear{Liu and Wu}{2007}]{Liu01122007}
\begin{barticle}
\bauthor{\bsnm{Liu}, \binits{Y.}},
\bauthor{\bsnm{Wu}, \binits{Y.}}:
\batitle{Variable selection via a combination of the $l_0$ and $l_1$
  penalties}.
\bjtitle{Journal of Computational and Graphical Statistics}
\bvolume{16}(\bissue{4}),
\bfpage{782}--\blpage{798}
(\byear{2007})
\end{barticle}
\endbibitem

\bibitem[\protect\citeauthoryear{Blumensath and Davies}{2006}]{1561263}
\begin{barticle}
\bauthor{\bsnm{Blumensath}, \binits{T.}},
\bauthor{\bsnm{Davies}, \binits{M.}}:
\batitle{Sparse and shift-invariant representations of music}.
\bjtitle{IEEE Transactions on Audio, Speech, and Language Processing}
\bvolume{14}(\bissue{1}),
\bfpage{50}--\blpage{57}
(\byear{2006})
\doiurl{10.1109/TSA.2005.860346}
\end{barticle}
\endbibitem

\bibitem[\protect\citeauthoryear{Tibshirani}{2018}]{10.1111/j.2517-6161.1996.tb02080.x}
\begin{barticle}
\bauthor{\bsnm{Tibshirani}, \binits{R.}}:
\batitle{Regression shrinkage and selection via the lasso}.
\bjtitle{Journal of the Royal Statistical Society: Series B (Methodological)}
\bvolume{58}(\bissue{1}),
\bfpage{267}--\blpage{288}
(\byear{2018})
\doiurl{10.1111/j.2517-6161.1996.tb02080.x}
\end{barticle}
\endbibitem

\bibitem[\protect\citeauthoryear{Blumensath and
  Davies}{2008}]{blumensath2008iterative}
\begin{barticle}
\bauthor{\bsnm{Blumensath}, \binits{T.}},
\bauthor{\bsnm{Davies}, \binits{M.E.}}:
\batitle{Iterative thresholding for sparse approximations}.
\bjtitle{Journal of Fourier analysis and Applications}
\bvolume{14},
\bfpage{629}--\blpage{654}
(\byear{2008})
\end{barticle}
\endbibitem

\bibitem[\protect\citeauthoryear{Attouch et~al.}{2013}]{attouch2013convergence}
\begin{barticle}
\bauthor{\bsnm{Attouch}, \binits{H.}},
\bauthor{\bsnm{Bolte}, \binits{J.}},
\bauthor{\bsnm{Svaiter}, \binits{B.F.}}:
\batitle{Convergence of descent methods for semi-algebraic and tame problems:
  proximal algorithms, forward--backward splitting, and regularized
  gauss--seidel methods}.
\bjtitle{Mathematical Programming}
\bvolume{137}(\bissue{1}),
\bfpage{91}--\blpage{129}
(\byear{2013})
\end{barticle}
\endbibitem

\bibitem[\protect\citeauthoryear{Bertsimas et~al.}{2016}]{bertsimas2016best}
\begin{botherref}
\oauthor{\bsnm{Bertsimas}, \binits{D.}},
\oauthor{\bsnm{King}, \binits{A.}},
\oauthor{\bsnm{Mazumder}, \binits{R.}}:
Best subset selection via a modern optimization lens
(2016)
\end{botherref}
\endbibitem

\bibitem[\protect\citeauthoryear{Cheng et~al.}{2020}]{cheng2020active}
\begin{barticle}
\bauthor{\bsnm{Cheng}, \binits{W.}},
\bauthor{\bsnm{Chen}, \binits{Z.}},
\bauthor{\bsnm{Hu}, \binits{Q.}}:
\batitle{An active set barzilar--borwein algorithm for $\ell_0$ regularized
  optimization}.
\bjtitle{Journal of Global Optimization}
\bvolume{76}(\bissue{4}),
\bfpage{769}--\blpage{791}
(\byear{2020})
\end{barticle}
\endbibitem

\bibitem[\protect\citeauthoryear{Ito and Kunisch}{2013}]{Ito_2014}
\begin{barticle}
\bauthor{\bsnm{Ito}, \binits{K.}},
\bauthor{\bsnm{Kunisch}, \binits{K.}}:
\batitle{A variational approach to sparsity optimization based on lagrange
  multiplier theory}.
\bjtitle{Inverse Problems}
\bvolume{30}(\bissue{1}),
\bfpage{015001}
(\byear{2013})
\doiurl{10.1088/0266-5611/30/1/015001}
\end{barticle}
\endbibitem

\bibitem[\protect\citeauthoryear{Huang et~al.}{2018}]{huang2018constructive}
\begin{barticle}
\bauthor{\bsnm{Huang}, \binits{J.}},
\bauthor{\bsnm{Jiao}, \binits{Y.}},
\bauthor{\bsnm{Liu}, \binits{Y.}},
\bauthor{\bsnm{Lu}, \binits{X.}}:
\batitle{A constructive approach to $l_0$ penalized regression}.
\bjtitle{Journal of Machine Learning Research}
\bvolume{19}(\bissue{10}),
\bfpage{1}--\blpage{37}
(\byear{2018})
\end{barticle}
\endbibitem

\bibitem[\protect\citeauthoryear{Zhou et~al.}{2021a}]{zhou2021global}
\begin{barticle}
\bauthor{\bsnm{Zhou}, \binits{S.}},
\bauthor{\bsnm{Xiu}, \binits{N.}},
\bauthor{\bsnm{Qi}, \binits{H.-D.}}:
\batitle{Global and quadratic convergence of newton hard-thresholding pursuit}.
\bjtitle{Journal of Machine Learning Research}
\bvolume{22}(\bissue{12}),
\bfpage{1}--\blpage{45}
(\byear{2021})
\end{barticle}
\endbibitem

\bibitem[\protect\citeauthoryear{Zhou et~al.}{2021b}]{zhou2021quadratic}
\begin{barticle}
\bauthor{\bsnm{Zhou}, \binits{S.}},
\bauthor{\bsnm{Pan}, \binits{L.}},
\bauthor{\bsnm{Xiu}, \binits{N.}},
\bauthor{\bsnm{Qi}, \binits{H.-D.}}:
\batitle{Quadratic convergence of smoothing newton's method for 0/1 loss
  optimization}.
\bjtitle{SIAM Journal on Optimization}
\bvolume{31}(\bissue{4}),
\bfpage{3184}--\blpage{3211}
(\byear{2021})
\end{barticle}
\endbibitem

\bibitem[\protect\citeauthoryear{Zhou et~al.}{2021c}]{zhou2021newton}
\begin{botherref}
\oauthor{\bsnm{Zhou}, \binits{S.}},
\oauthor{\bsnm{Pan}, \binits{L.}},
\oauthor{\bsnm{Xiu}, \binits{N.}}:
Newton method for $\ell_0$-regularized optimization.
Numerical Algorithms,
1--30
(2021)
\end{botherref}
\endbibitem

\bibitem[\protect\citeauthoryear{Liuzzi and Rinaldi}{2015}]{liuzzi2015solving}
\begin{barticle}
\bauthor{\bsnm{Liuzzi}, \binits{G.}},
\bauthor{\bsnm{Rinaldi}, \binits{F.}}:
\batitle{Solving $\ell_0$-penalized problems with simple constraints via the
  frank--wolfe reduced dimension method}.
\bjtitle{Optimization Letters}
\bvolume{9}(\bissue{1}),
\bfpage{57}--\blpage{74}
(\byear{2015})
\end{barticle}
\endbibitem

\bibitem[\protect\citeauthoryear{Wu et~al.}{2021}]{wu2021smoothing}
\begin{botherref}
\oauthor{\bsnm{Wu}, \binits{F.}},
\oauthor{\bsnm{Bian}, \binits{W.}},
\oauthor{\bsnm{Xue}, \binits{X.}}:
Smoothing fast iterative hard thresholding algorithm for $\ell_0 $ regularized
  nonsmooth convex regression problem.
arXiv preprint arXiv:2104.13107
(2021)
\end{botherref}
\endbibitem

\bibitem[\protect\citeauthoryear{Zhao et~al.}{2015}]{zhao2015adaptive}
\begin{barticle}
\bauthor{\bsnm{Zhao}, \binits{Z.}},
\bauthor{\bsnm{Xu}, \binits{F.}},
\bauthor{\bsnm{Li}, \binits{X.}}:
\batitle{Adaptive projected gradient thresholding methods for constrained
  $\ell_0$ problems}.
\bjtitle{Science China Mathematics}
\bvolume{58},
\bfpage{1}--\blpage{20}
(\byear{2015})
\end{barticle}
\endbibitem

\bibitem[\protect\citeauthoryear{Bian and Chen}{2020}]{doi:10.1137/18M1186009}
\begin{barticle}
\bauthor{\bsnm{Bian}, \binits{W.}},
\bauthor{\bsnm{Chen}, \binits{X.}}:
\batitle{A smoothing proximal gradient algorithm for nonsmooth convex
  regression with cardinality penalty}.
\bjtitle{SIAM Journal on Numerical Analysis}
\bvolume{58}(\bissue{1}),
\bfpage{858}--\blpage{883}
(\byear{2020})
\doiurl{10.1137/18M1186009}
{\href{https://arxiv.org/abs/https://doi.org/10.1137/18M1186009}{{https://doi.org/10.1137/18M1186009}}}
\end{barticle}
\endbibitem

\bibitem[\protect\citeauthoryear{Lu}{2014}]{lu2014iterative}
\begin{barticle}
\bauthor{\bsnm{Lu}, \binits{Z.}}:
\batitle{Iterative hard thresholding methods for $\ell_0$ regularized convex
  cone programming}.
\bjtitle{Mathematical Programming}
\bvolume{147}(\bissue{1}),
\bfpage{125}--\blpage{154}
(\byear{2014})
\end{barticle}
\endbibitem

\bibitem[\protect\citeauthoryear{Li and Bian}{2021}]{LI2021678}
\begin{barticle}
\bauthor{\bsnm{Li}, \binits{W.}},
\bauthor{\bsnm{Bian}, \binits{W.}}:
\batitle{Smoothing neural network for $l_0$ regularized optimization problem
  with general convex constraints}.
\bjtitle{Neural Networks}
\bvolume{143},
\bfpage{678}--\blpage{689}
(\byear{2021})
\doiurl{10.1016/j.neunet.2021.08.001}
\end{barticle}
\endbibitem

\bibitem[\protect\citeauthoryear{Zhang and
  Zhang}{2019}]{zhang2018acceleratedproximaliterativehard}
\begin{barticle}
\bauthor{\bsnm{Zhang}, \binits{X.}},
\bauthor{\bsnm{Zhang}, \binits{X.}}:
\batitle{A new proximal iterative hard thresholding method with extrapolation
  for $\ell _0$minimization}.
\bjtitle{Journal of Scientific Computing}
\bvolume{79}(\bissue{2}),
\bfpage{809}--\blpage{826}
(\byear{2019})
\end{barticle}
\endbibitem

\bibitem[\protect\citeauthoryear{Wu and Bian}{2020}]{wu2020accelerated}
\begin{barticle}
\bauthor{\bsnm{Wu}, \binits{F.}},
\bauthor{\bsnm{Bian}, \binits{W.}}:
\batitle{Accelerated iterative hard thresholding algorithm for $\ell_0$
  regularized regression problem}.
\bjtitle{Journal of Global Optimization}
\bvolume{76}(\bissue{4}),
\bfpage{819}--\blpage{840}
(\byear{2020})
\end{barticle}
\endbibitem

\bibitem[\protect\citeauthoryear{Calamai and
  Mor{\'e}}{1987}]{calamai1987projected}
\begin{barticle}
\bauthor{\bsnm{Calamai}, \binits{P.H.}},
\bauthor{\bsnm{Mor{\'e}}, \binits{J.J.}}:
\batitle{Projected gradient methods for linearly constrained problems}.
\bjtitle{Mathematical Programming}
\bvolume{39}(\bissue{1}),
\bfpage{93}--\blpage{116}
(\byear{1987})
\end{barticle}
\endbibitem

\bibitem[\protect\citeauthoryear{Harker and Pang}{1990}]{Pang1990}
\begin{barticle}
\bauthor{\bsnm{Harker}, \binits{P.T.}},
\bauthor{\bsnm{Pang}, \binits{J.-S.}}:
\batitle{Finite-dimensional variational inequality and nonlinear
  complementarity problems: a survey of theory, algorithms and applications}.
\bjtitle{Mathematical Programming}
\bvolume{48}(\bissue{1-3}),
\bfpage{161}--\blpage{220}
(\byear{1990})
\end{barticle}
\endbibitem

\bibitem[\protect\citeauthoryear{Facchinei and
  Pang}{2003}]{facchinei2003finite}
\begin{bbook}
\bauthor{\bsnm{Facchinei}, \binits{F.}},
\bauthor{\bsnm{Pang}, \binits{J.-S.}}:
\bbtitle{Finite-dimensional Variational Inequalities and Complementarity
  Problems}.
\bpublisher{Springer}
(\byear{2003})
\end{bbook}
\endbibitem

\bibitem[\protect\citeauthoryear{Mor{\'e} and
  Sorensen}{1983}]{more1983computing}
\begin{barticle}
\bauthor{\bsnm{Mor{\'e}}, \binits{J.J.}},
\bauthor{\bsnm{Sorensen}, \binits{D.C.}}:
\batitle{Computing a trust region step}.
\bjtitle{SIAM Journal on scientific and statistical computing}
\bvolume{4}(\bissue{3}),
\bfpage{553}--\blpage{572}
(\byear{1983})
\end{barticle}
\endbibitem

\bibitem[\protect\citeauthoryear{Facchinei}{1995}]{facchinei1995minimization}
\begin{barticle}
\bauthor{\bsnm{Facchinei}, \binits{F.}}:
\batitle{Minimization of sc1 functions and the maratos effect}.
\bjtitle{Operations Research Letters}
\bvolume{17}(\bissue{3}),
\bfpage{131}--\blpage{137}
(\byear{1995})
\end{barticle}
\endbibitem

\bibitem[\protect\citeauthoryear{Zhang and Cheng}{2017}]{zhang2017projected}
\begin{barticle}
\bauthor{\bsnm{Zhang}, \binits{H.}},
\bauthor{\bsnm{Cheng}, \binits{L.}}:
\batitle{Projected shrinkage algorithm for box-constrained
  $\ell_1$-minimization}.
\bjtitle{Optimization Letters}
\bvolume{11}(\bissue{1}),
\bfpage{55}--\blpage{70}
(\byear{2017})
\end{barticle}
\endbibitem

\bibitem[\protect\citeauthoryear{Cheng et~al.}{2023}]{CHENG2023179}
\begin{barticle}
\bauthor{\bsnm{Cheng}, \binits{W.}},
\bauthor{\bsnm{LinPeng}, \binits{Z.}},
\bauthor{\bsnm{Li}, \binits{D.}}:
\batitle{An inexact quasi-newton algorithm for large-scale $\ell_1$
  optimization with box constraints}.
\bjtitle{Applied Numerical Mathematics}
\bvolume{193},
\bfpage{179}--\blpage{195}
(\byear{2023})
\doiurl{10.1016/j.apnum.2023.07.004}
\end{barticle}
\endbibitem

\end{thebibliography}

\end{document}